\documentclass[12pt]{article}
\usepackage[T1]{fontenc}
\usepackage[english]{babel}
\usepackage{lmodern}
\usepackage{microtype}

\usepackage{mathtools}
\mathtoolsset{showonlyrefs}
\usepackage{amsfonts}
\usepackage{amssymb}
\usepackage{amsthm}
\usepackage{mathrsfs}
\usepackage{esint}

\usepackage[a4paper,left=1.5cm,right=1.5cm,top=2.5cm,bottom=2.5cm]{geometry}
\usepackage{xcolor}
\definecolor{myurlcolor}{rgb}{0,0,0.4}
\definecolor{mycitecolor}{rgb}{0,0.5,0}
\definecolor{myrefcolor}{rgb}{0.5,0,0}
\usepackage{graphicx}
\usepackage{tikz}
\usepackage{tikz-cd}
\usepackage{caption}

\usepackage{csquotes}

\usepackage{url}
\usepackage{etaremune}
\usepackage[shortlabels]{enumitem} 
\usepackage[normalem]{ulem}
  
\usepackage{comment}
\usepackage{varwidth}
\usepackage{makeidx}
\usepackage{aliascnt}

\newcommand{\T}{\mathrm{T}}
\DeclareMathOperator\id{id}
\DeclareMathOperator\ad{ad}
\DeclareMathOperator\Ad{Ad}

\DeclareMathOperator\GL{GL}
\DeclareMathOperator\SL{SL}

\DeclareMathOperator{\Tr}{Tr}
\newcommand{\FR}{\mathrm{FR}}

\DeclareMathOperator{\Tor}{Tor}
\DeclareMathOperator{\im}{Im}

\newcommand{\sa}{sa}
\newcommand{\dd}{\mathrm{d}}
\newcommand{\E}{\mathcal{E}}
\newcommand{\A}{\mathcal{A}}

\usepackage[
backend=biber,
sorting=nyt, 
sortcites=true, 
giveninits=true, 
url=false,
backref=true,
maxbibnames=99 
]{biblatex}
\DefineBibliographyStrings{english}{%
	backrefpage = {$\downarrow$\,p.},
	backrefpages = {$\downarrow$\,pp.}
}
\AtEveryBibitem{\clearfield{note}}
\AtEveryBibitem{\clearfield{isbn}}
\AtEveryBibitem{\clearfield{issn}}
\AtEveryBibitem{%
  \clearfield{urlyear}%
  \clearfield{urlmonth}%
  \clearfield{urlday}%
  \clearfield{urldate}%
}
\AtEveryBibitem{\clearfield{abstract}}
\AtEveryBibitem{\clearfield{keywords}}

\newtheorem*{proof*}{Proof}
\newtheoremstyle{maybestyle}
  {12pt}                
  {12pt}                
  {\itshape}         
  {0pt}                 
  {\normalfont\bfseries}           
  {.}                   
  {6pt}                 
  {}                    

\theoremstyle{maybestyle}

\newtheorem{theorem}{Theorem}[section]
\newaliascnt{proposition}{theorem}
\newtheorem{proposition}[proposition]{Proposition}
\aliascntresetthe{proposition}
\newaliascnt{lemma}{theorem}
\newtheorem{lemma}[lemma]{Lemma}
\aliascntresetthe{lemma}

\newtheoremstyle{mystyle}
  {12pt}                
  {12pt}                
  {\normalfont}         
  {0pt}                 
  {\bfseries}           
  {.}                   
  {6pt}                 
  {}                    

\newif\ifqedenvironment
\newcount\qedplaced
\newcommand{\exampleqed}{\leavevmode\unskip\nobreak\hfill$\circ$}
\newcommand{\placeexampleqed}{\ifnum\qedplaced=0 \exampleqed\global\qedplaced=1 \fi}
\AtBeginEnvironment{example}{\qedenvironmenttrue\global\qedplaced=0 }
\AtBeginEnvironment{remark}{\qedenvironmenttrue\global\qedplaced=0 }
\AtBeginEnvironment{definition}{\qedenvironmenttrue\global\qedplaced=0 }
\AtEndEnvironment{enumerate}{\ifqedenvironment\placeexampleqed\fi}
\AtEndEnvironment{itemize}{\ifqedenvironment\placeexampleqed\fi}
\AtEndEnvironment{description}{\ifqedenvironment\placeexampleqed\fi}
\AtEndEnvironment{example}{\placeexampleqed\qedenvironmentfalse}  
\AtEndEnvironment{remark}{\placeexampleqed\qedenvironmentfalse}
\AtEndEnvironment{definition}{\placeexampleqed\qedenvironmentfalse}
\theoremstyle{mystyle}
\newtheorem{example}{Example}[section]
\newtheorem{remark}{Remark}[section]
\newtheorem{definition}{Definition}[section]

\usepackage{hyperref}
\hypersetup{colorlinks,linkcolor=myrefcolor,citecolor=mycitecolor,urlcolor=myurlcolor}
\usepackage[capitalize]{cleveref}

\usepackage{parskip}
\title{Potential functions in information geometry via bi-forms}

\author{F. M. Ciaglia$^{1,6}$ \href{https://orcid.org/0000-0002-8987-1181}{\includegraphics[scale=0.7]{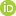}}, G. Marmo$^{2,3,7}$ \href{https://orcid.org/0000-0003-2662-2193}{\includegraphics[scale=0.7]{ORCID.png}}, M. Pacelli$^{2,4,5,8}$ \href{https://orcid.org/0009-0008-4437-6970}{\includegraphics[scale=0.7]{ORCID.png}}, \\
L. Schiavone$^{5,9}$  \href{https://orcid.org/0000-0002-1817-5752}{\includegraphics[scale=0.7]{ORCID.png}}, 
A. Zampini$^{2,4,5,10}$ \href{https://orcid.org/0000-0003-0980-6003}{\includegraphics[scale=0.7]{ORCID.png}}}

\begin{document}

\maketitle 

\noindent
{\footnotesize{$^{1}$   Universidad Carlos III de Madrid, 
Departamento de Matemáticas, 
Leganés (Madrid), Spain\\
$^{2}$ INFN-Sezione di Napoli, Naples, Italy  \\
$^{3}$ Dipartimento di Fisica ``E. Pancini'', Università degli Studi di Napoli Federico II,  Naples, Italy  \\
$^{4}$ Scuola Superiore Meridionale,  Naples, Italy  \\
$^{5}$ Dipartimento di Matematica e Applicazioni ``R. Caccioppoli'', Università degli Studi di Napoli Federico II, Naples, Italy

\noindent
$^{6}$\texttt{fciaglia[at]math.uc3m.es} \quad 
$^{7}$\texttt{marmo[at]na.infn.it} \quad 
$^{8}$\texttt{marco.pacelli-ssm[at]unina.it}  \\
$^{9}$\texttt{luca.schiavone[at]unina.it}  \quad
$^{10}$\texttt{alessandro.zampini[at]unina.it}}

\begin{abstract}
\fontsize{8}{10}\selectfont 
\noindent In this paper we develop a general framework for potentials on Lauritzen manifolds, namely smooth manifolds equipped with a pseudo-Riemannian metric and a pair of conjugate affine connections that may have non-vanishing torsion. We show how the theory of bi-forms accommodates torsion-full statistical structures and unifies contrast and pre-contrast functions in a cohomological framework. 
Within this formalism, we construct a canonical contrast bi-form on dually curvature-free Lauritzen manifolds and establish its principal structural properties. Several illustrative examples are analysed.
\end{abstract}

{\footnotesize
\tableofcontents
}
 
\section{Introduction}\label{Sec: introduction}
Contrast functions, or divergence functions, are central notions in statistics and information theory, providing a means to quantify the distinguishability between probability distributions and forming the basis for parameter estimation, hypothesis testing, and model selection. When viewed through the lens of geometry, contrast functions appear as suitable real-valued functions on the Cartesian square of a smooth manifold, whose mixed second order derivatives, when restricted to the diagonal submanifold, define a symmetric and non-degenerate $2$-covariant tensor endowing the manifold with a pseudo-Riemannian metric. A further action of partial derivations then gives rise to a pair of conjugate, torsion-free affine connections \cite{Eguchi1985}. In this sense, a contrast function acts as a potential that equips the underlying smooth manifold with the structure of a statistical manifold \cite{Lauritzen-1987}.

While contrast functions describe 
the metric-affine structures underlying classical parametric models, which motivated Lauritzen's definition of statistical manifold, 
they are not suited to describe more general information-geometric structures in which one of the conjugate affine connections, or both,
may have non vanishing torsion: we refer for example to the space of faithful quantum states endowed with a quantum monotone metric tensor and the connection encoding its natural convexity \cite{Jencova-2001}. 
This observation motivated the study of statistical manifolds admitting torsion (SMAT), which generalize statistical manifolds by allowing a non-vanishing torsion for either the connection or its dual one. Such a generalization naturally raises the question of describing  suitable potentials. Henmi and Matsuzoe addressed this issue by introducing the notion of pre-contrast functions, namely asymmetric potentials defined on the Cartesian product of a smooth manifold with its tangent bundle, fiberwise linear in the tangent component \cite{H-M-2011}. Following the same line of reasoning, Khan and Zhang later introduced the notion of super-contrast functions as suitable fiberwise bilinear functions on the Cartesian square of the tangent bundle, to describe torsion-full statistical structures \cite{Z-K-2020}. Manifolds endowed with such data, namely a pseudo-Riemannian metric and a pair of conjugate affine connections, are referred to in the  literature with several names, including that of dualistic geometries \cite{Z-K-2020} and of Norden-Sen geometries \cite{Kostecki2023}
. In this paper we use the term \underline{Lauritzen manifold} for this general setting, motivated by  the viewpoint underlying Lauritzen's formulation of statistical manifolds. Classes of motivating examples comprise teleparallel \cite{C-DC-I-M-2023, GSI} or pseudo-Weitzenb\"ock geometries \cite{Z-K-2019}, as well as Lie groups endowed with Cartan-Schouten connections \cite{CartanSchouten, D-M-S-2024}.

The present study provides a systematic development of the bi-form approach to the generation problem in information geometry. The two-point formalism underlying this viewpoint was introduced in the framework of absolute partial differential calculus \cite{Ruse1931}, and was later applied to general relativity \cite{Einstein1944, EinsteinBargmann1944, Synge1960, DeWittBrehme1960}, to the theory of harmonic forms \cite{NickersonSpencerSteenrod1959}, and to the theory of currents \cite{deRham1955}. Within information geometry, a formulation of the extraction procedure for geometric data from contrast functions in terms of bi-forms was given in \cite{MankoMarmoVentrigliaVitale2017}. In \cite{GSI}, this viewpoint was further developed by interpreting suitable bi-forms as generalized potentials for the generation problem of Lauritzen manifolds, leading to the notion of \underline{contrast bi-form} $\varpi$. It was proved there that the torsion properties of the induced connections are described in terms of the action of two commuting cohomological operators, the left and right exterior derivatives $\mathrm d^L$ and $\mathrm d^R$: the (so-called) primal connection, respectively the dual connection, is torsion-free precisely when $\mathrm d^L\varpi$, respectively $\mathrm d^R\varpi$, vanishes along the diagonal submanifold. A cohomological characterization theorem for pre-contrast and contrast functions was also stated there, but its proof was omitted due to space constraints. A full and detailed treatment of these constructions is provided here, together with complete proofs and further examples.

The analysis carried out here proceeds through the construction of a Cartan-like calculus for bi-forms. Built in parallel with the bi-differential calculus associated to the product structure of the Cartesian square of a manifold, this calculus rests on the identification of bi-forms with differential forms of mixed type on the Cartesian square. The cohomological characterization of pre-contrast and contrast functions then requires a detailed study of the homotopy operators associated to the left and right exterior derivatives $\mathrm d^L$ and $\mathrm d^R$. Their construction depends on a careful analysis of the local geometry close to the diagonal. In the classical setting of differential forms, such analyses rely on tubular neighbourhoods. In the present setting, the involved geometric structures suggest a different approach. Strongly convex neighbourhoods in the sense of Moretti \cite{Moretti-2021} are therefore considered, namely open neighbourhoods of the entire diagonal submanifold in the Cartesian square of the underlying manifold equipped with the unique smooth geodesic interpolation map used to define the left and right homotopy operators. Within these neighbourhoods, \underline{left} and \underline{right statistical homotopy operators} are constructed. These operators decompose any bi-form into left-exact and left-antiexact components, and likewise into right-exact and right-antiexact components, respectively. When the torsion of the primal connection vanishes, its left-antiexact component becomes statistically irrelevant, so that the geometric information encoded by $\varpi$ is completely captured by its left-exact component. The fiberwise realization of this component corresponds to a pre-contrast function, thereby embedding pre-contrast functions naturally as left-exact contrast bi-forms. Similarly, bi-exact contrast bi-forms, that is, left-exact bi-forms which are also right-exact, realize contrast functions. 

Does every Lauritzen structure arise from a contrast bi-form? We leave a complete answer to such a question to a future paper, and present contrast bi-forms for a paradigmatic class of Lauritzen manifolds, namely those that are \underline{dually curvature-free}: both connections have vanishing curvature, possibly non-vanishing torsion. Such structures are ubiquitous in information geometry. Classical parametric models, with the probability simplex as the fundamental example, provide standard instances when equipped with mixture and exponential connections \cite{A-N-2000}. Partially flat structures have been employed in estimation theory by Henmi and Matsuzoe \cite{H-M-2011}, and appear on the space of smooth probability densities \cite{A-S-2025}. Teleparallel geometries \cite{C-DC-I-M-2023, GSI, Z-K-2019}, as well as semisimple Lie groups equipped with the Killing metric and either the left or the right Cartan connection \cite{Postnikov-2001, CartanSchouten}, provide further examples of dually curvature-free Lauritzen manifolds. In this context, we construct a contrast bi-form associated to any dually curvature-free Lauritzen manifold, which exhibits several remarkable properties. In the partially flat setting, the bi-form is globally left-exact, and its potential reproduces the pre-contrast function introduced by Henmi and Matsuzoe \cite{H-M-2019}. In the dually flat case, the bi-form is bi-exact and generated by the canonical divergence of Ay and Amari \cite{A-A-2015}. In this sense, the bi-form can be considered canonical. The solution bi-form also admits additional structural properties: it can be expressed in terms of parallel transport, and on semisimple Lie groups it becomes diagonally left-invariant for the right Cartan connections and left-invariant for left Cartan connections. 

\subsection{Lauritzen manifolds}
Let $M$ be a finite dimensional, paracompact smooth manifold, which we assume throughout this paper to be connected. We 
primarily consider two specific geometric structures on $M$, namely a pseudo-Riemannian metric $g$ and an affine connection $\nabla$.

Recall that a pseudo-Riemannian metric $g$ on $M$ can be viewed either as a smooth, symmetric, nondegenerate $2$-covariant tensor on $M$, or equivalently as a nondegenerate, symmetric, $C^\infty(M)$-bilinear map:
\begin{equation}
g\colon \mathfrak X(M)\times \mathfrak X(M)\longrightarrow C^\infty(M)\,,
\end{equation}
where $C^\infty(M)$ denotes the ring of smooth real-valued functions on $M$, and $\mathfrak X(M)$ is the $C^\infty(M)$-module of vector fields on $M$. In local coordinates $(U, q)$, the metric tensor $g$ reads\footnote{Here we adopt Einstein summation convention, as in the rest of the paper unless otherwise stated.} $g=g_{ij}\, \mathrm dq^i\otimes \mathrm dq^j$. Being non-degenerate, a pseudo-Riemannian metric tensor $g$ 
 on $M$ induces a pair of (mutually inverse) $C^\infty(M)$-isomorphisms between the $C^\infty(M)$-module $\mathfrak X(M)$ and the $C^\infty(M)$-module of $1$-forms on $M$ (that we  denote by $\Omega^1(M)$), that is 
\begin{equation}
     {}^\flat\colon \mathfrak X(M)\longrightarrow \Omega^1(M)\,, \qquad {}^\sharp\colon\Omega^1(M)\longrightarrow\mathfrak X(M)\,.
\end{equation}
Such correspondences, also known as the \emph{musical isomorphisms} (associated to $g$), are implicitly defined by
\begin{equation}
\begin{split}
    \alpha(Z)=g\left(\alpha^\sharp,Z\right)\,, \\ X^\flat(Z)=g(X,Z)\,,
    \end{split}
\end{equation}
where $X,Z\in \mathfrak X(M)$, and $\alpha\in \Omega^1(M)$. In local coordinates $(U,q)$, their action is written as 
\begin{align}
    \alpha^\sharp=g^{ij}\,\alpha_i\,\partial_{q^j}\,, &\qquad \text{for }   \alpha=\alpha_i\, \mathrm dq^i\,,\\
    X^\flat=g_{ij}\,X^j\,\mathrm dq^i\,, & \qquad \text{ for } X=X^i\, \partial_{q^i}\,,
\end{align}
with $g^{ij}g_{jk}=\delta^i_k$. 

Given a (smooth) vector bundle $\tau\colon E\to M$ and therefore a $C^{\infty}(M)$-module of (smooth) sections $\Gamma^{\infty}_\tau$,
an affine connection is a map 
\begin{equation}
\nabla\colon\Gamma_\tau^\infty\longrightarrow \Omega^1(M)\,\otimes_{C^\infty(M)}\Gamma^\infty_\tau
\end{equation}
such that, for any $\eta,\tilde\eta\in\Gamma^\infty_\tau$ and any $f\in C^{\infty}(M)$, one has 
\begin{align}
&\nabla(\eta+\tilde\eta)=\nabla(\eta)+\nabla(\tilde\eta)\,, \\ &\nabla(f\eta)=\dd f\otimes\eta+f\nabla(\eta)\,.
\end{align}
If $Z\in\mathfrak{X}(M)$ is a smooth vector field on the basis manifold $M$ of the given vector bundle, the covariant derivative $\nabla_Z\colon \Gamma^\infty_\tau\to\Gamma^\infty_\tau$ along the vector field $Z$ comes upon contracting the 1-form tensorial leg of $\nabla\eta$ along $Z$. For such a map, one has that a \emph{Leibniz rule} holds, namely 
\begin{equation}
\nabla_Z\left(f\,\eta\right)=\left(\mathcal L_Zf\right)\,\eta+f\,\nabla_Z\eta\,.
\end{equation}
When the vector bundle we consider on $M$ is given by the tangent bundle $\tau\colon \T M\to M$, so that $\Gamma^\infty_\tau=\mathfrak{X}(M)$, we usually say  that $\nabla$ is an affine connection \emph{on} $M$. In local coordinates $(U,q)$, 
an affine connection on $M$ is determined as follows:
\begin{equation}
\begin{split}
    &\nabla(\partial_{q^i})=\Gamma^k_{ij}\,\dd q^j\otimes\partial_{q^k}, \\
    &\nabla_{\partial_{q^i}}\partial_{q^j}=\Gamma_{ij}^k\,\partial_{q^k}\,, \qquad i,j\in \{1,\dots,\dim M\}\,.
\end{split}
\end{equation}
We denote by $\Tor^\nabla$ and $R^\nabla$ the torsion tensor and curvature tensor corresponding to $\nabla$, defined by 
\begin{align}
    \Tor^\nabla(X,Y)&=\nabla_XY-\nabla_YX-[X,Y]\,, \label{Eq: torsion tensor}\\
      R^\nabla(X,Y)Z &= \nabla_X\nabla_YZ - \nabla_Y\nabla_XZ - \nabla_{[X,Y]}Z\,, \label{Eq: curvature tensor}
\end{align}
with $X,Y,Z\in\mathfrak{X}(M)$. In local coordinates $(U,q)$ we have 
\begin{align}
\Tor^\nabla&=\left(\Gamma_{ij}^k-\Gamma_{ji}^k\right)\, \left(\mathrm dq^i\otimes \mathrm dq^j\right)\otimes \partial_{q^k}\,,\\R^\nabla
&= \left(\frac{\partial\Gamma^k_{jl}}{\partial q^i} - \frac{\partial\Gamma^k_{il}}{\partial q^j} 
+ \Gamma^m_{jl}\, \Gamma^k_{im} - \Gamma^m_{il}\, \Gamma^k_{jm} \right)
\, \mathrm d q^i \otimes \mathrm dq^j \otimes \mathrm dq^l \otimes \partial_{q^k}\,.
\end{align}
We define  $\nabla$ to be  \underline{torsion-free} if and only if the torsion tensor of $\nabla$ vanishes. Similarly, $\nabla$ is called \underline{curvature-free} if and only if the curvature tensor of $\nabla$ vanishes. Lastly, we say that $\nabla$ is \underline{flat} if and only if it is both torsion-free and curvature-free. 

Any affine connection $\nabla$ extends to a connection $\nabla^\ast$ on sections of the cotangent bundle, i.e. on 1-forms on $M$, 
(see \cite{Lee-2018}) by 
\begin{equation}
\label{Eq: dual connection}
\mathcal{L}_Z(\alpha(X))=(\nabla^*_Z\alpha)(X)+\alpha(\nabla_ZX), 
\end{equation}
with $\alpha\in \Omega^1(M)$ and $X,Z\in \mathfrak X(M)$. When $M$ has a pseudo-Riemannian metric $g$, the musical 
isomorphisms introduced above allow to define (see \cite{NomizuSimon1992}) another affine connection (that we denote  by $\nabla^\dag$) by 
\begin{equation}
\label{Eq: nabla dag}
    \nabla_Z^\dag (Y) =\left(\nabla_Z^\ast (Y^\flat)\right)^\sharp\,.
\end{equation}
We refer to $\nabla^\dag$ as the $g$-\underline{conjugate affine connection} of $\nabla$. In local coordinates $(U,q)$, the $g$-conjugate affine connection of $\nabla$ is determined by:
\begin{equation}
    \nabla^\dag_{\partial_{q^k}}(\partial_{q^j})=\left(\Gamma^\dag\right)_{kj}^r\,\partial_{q^r}\,, \qquad \left(\Gamma^\dag\right)_{kj}^r=g^{ir}\left(\frac{\partial g_{ij}}{\partial q^k}-g_{\ell j}\Gamma^\ell_{ki}\right)
\end{equation}
By \eqref{Eq: dual connection}, we obtain the equality 
\begin{equation}\label{Eq: compatibility condition}
    \mathcal L_Z\left(g(X,Y)\right)= g\left(\nabla_Z X,Y\right)+ g\left(X,\nabla_Z^\dag Y\right)\,,
\end{equation}
where $X,Y,Z \in \mathfrak X(M)$. This relation shows that $(\nabla^\dag)^\dag=\nabla$. 
We refer to \eqref{Eq: compatibility condition} as the $g$-\underline{compatibility condition} of $\nabla$ and $\nabla^\dag$. We define an affine connection $\nabla$ to be $g$-compatible if and only if it is $g$-self-conjugate, i.e. $\nabla = \nabla^\dag$. The fundamental theorem of Riemannian geometry asserts the existence of a unique $g$-compatible affine connection with given torsion tensor \cite{GHV}: the unique torsion-free $g$-compatible affine connection is called the Levi-Civita affine connection of $g$, and is denoted by $\nabla^g$.

We recall the notions of statistical manifolds \cite{Lauritzen-1987} and of SMATs \cite{H-M-2011}.
\begin{definition}
\label{Def: Lauritzen manifolds}
Let $M$ be a smooth manifold. When it is equipped with a pseudo-Riemannian metric tensor $g$ and an affine connection $\nabla$, we refer to it as a \underline{Lauritzen manifold}. In particular, we say that $(M,g,\nabla)$ gives a 
\begin{enumerate}
\item a \underline{statistical manifold admitting torsion} -- or simply \underline{SMAT} -- if and only if $\nabla$ is torsion-free;
    \item a \underline{statistical manifold} if and only if both $\nabla$ and $\nabla^\dag$ are torsion-free.
\end{enumerate}
\end{definition}
According to such a definition, the set of statistical manifolds is strictly contained in the set of SMAT, which in turn is contained in the set of Lauritzen manifolds, for which no requirement on the torsion corresponding to the connections $\nabla, \nabla^\dag$ is assumed.

\begin{remark}\label{Remark: other names}
As we already noticed, the above terminology is not uniform in the literature. The same metric-affine setting is often described in terms of a pseudo-Riemannian metric and a pair of conjugate affine connections, and appears  under different names, including that of dualistic geometries \cite{Z-K-2020} and of Norden-Sen geometries \cite{Kostecki2023}
. We use the term Lauritzen manifold for this setting, but encode it by the triple $(M,g,\nabla)$, rather than by the quadruple $(M,g,\nabla,\nabla^\dagger)$, since $\nabla^\dagger$ is determined by $g$ and $\nabla$. Notice that the same name is used differently in \cite{K2014a}, where it denotes triples $(M,g,C)$ with $C$ completely symmetric $3$-covariant tensor, corresponding to our statistical manifold case.   
\end{remark}

\begin{remark}\label{Remark: smat}
The previous  definition of statistical manifold admitting torsion differs from the definition given in \cite{H-M-2011}, where $\nabla^\dag$ is required to be torsion-free instead of $\nabla$. Since $(\nabla^\dag)^\dag = \nabla$, the two notions are equivalent up to switching the roles of $\nabla$ and $\nabla^\dag$.
\end{remark}

\begin{remark}\label{Remark: dually curvature-free affine connections}
Conjugate affine connections are either both curvature-free or both curvature-full. As proven  in \cite{NomizuSimon1992}, the curvature tensors of $\nabla$ and its conjugate $\nabla^\dag$ are related as follows:
\begin{equation}
g\left(R^{(\nabla^\dag)}(X,Y)Z, W\right) = -g\left(Z, R^\nabla(X,Y) W\right),
\end{equation}
for all vector fields $X,Y,Z,W$ on $M$. An immediate consequence is that $\nabla$ is curvature-free if and only if $\nabla^\dag$ is curvature-free. In contrast, the torsion tensors of $\nabla$ and $\nabla^\dag$ may vanish independently.
\end{remark}

It is well known that tensors and affine connections on a smooth manifold $M$ can be obtained in terms of suitable potential functions defined on the \emph{Cartesian square} manifold $M\times M$. In order to describe a differential geometric formulation of a Lauritzen manifold, we consider a natural differential bicomplex on $M\times M$. 

\subsection{A bicomplex on the Cartesian square of a smooth manifold}
\label{Subsec: geometric preliminaries}
For a given smooth manifold $M$, the maps
$\pi_L,\pi_R\colon M\times M\to M$
\begin{equation}
\label{Eq: left and right projs}
\begin{split}
    &\pi_L(m,n)=m\,, \\
    &\pi_R(m,n)=n\,
    \end{split}
\end{equation}
are globally defined smooth surjective submersions, and $\pi_{L,R}\colon M\times M\to M$ provide trivial fiber bundles, with typical fiber $M$ \cite{Saunders-1989}. These  fiber bundle structures  allow to decompose the tangent bundle  $\tau_{M\times M}\colon \T(M\times M)\to M\times M$ as the direct sum of vertical subbundles. We write such a decomposition as the direct sum of vector subbundles
\begin{equation}\label{Eq: canonical decomposition T(MxM)}
\T(M\times M)=\ker \pi_R'\oplus \ker \pi_L'\,
\end{equation}
\noindent where $\pi_L'$ and $\pi_R'$ denote the tangent maps of $\pi_L$ and $\pi_R$, respectively. The direct summands in \eqref{Eq: canonical decomposition T(MxM)} are the eigendistributions of two $(1,1)$-tensors $\mathfrak L$ and $\mathfrak R$ on $M\times M$, that we describe  as the projectors on $\T(M\times M)$ satisfying
\begin{equation}
\label{LRd}
    \begin{cases}
        \ker \mathfrak L=\ker \pi_L'\\
        \im \mathfrak L=\ker \pi_R'
    \end{cases} \qquad \begin{cases}
        \ker \mathfrak R=\ker \pi_R'\\
        \im \mathfrak R=\ker \pi_L'
    \end{cases}
\end{equation}
In a local product chart $(U\times V,x,y)$, they read:
 \begin{align}
\mathfrak L&=\mathrm dx^i\otimes \partial_{x^i}\,,\\
\mathfrak R&=\mathrm dy^j\otimes \partial_{y^j}\,.
\end{align}
Note that these tensors further satisfy:
 \begin{equation}
\label{Eq: LR properties}
\begin{split}
&\mathfrak L^2=\mathfrak L\,, \\ & \mathfrak R^2=\mathfrak R\,,\\
    &\mathfrak L+\mathfrak R=\mathbb I\,, \\  & \mathfrak L\mathfrak R=\mathfrak R\mathfrak L=\mathbb O\,,
\end{split}
\end{equation}
\noindent where $\mathbb I$ and $\mathbb O$ denote, respectively, the identity and the zero endomorphism of $\T(M\times M)$.

\noindent A pair  $(\mathfrak L,\mathfrak R)$ of $(1,1)$-tensors satisfying the above properties is known as an  \emph{almost product structure} \cite{YanoKon1984}.
The pair $(\mathfrak L,\mathfrak R)$ we are considering share indeed further properties.
Recall that the Fr\"olicher--Nijenhuis bracket of two $(1,1)$-tensors $A,B$ on a smooth manifold $Q$ is the $(1,2)$-tensor defined by
\begin{equation}\label{Eq: FN bracket}
\begin{split}
    [A,B]_{\mathrm{FN}}(X,Y)&=[AX,BY]+[BX,AY]+(AB+BA)[X,Y] \\ & \quad-A\big([X,BY]+[BX,Y]\big)
-B\big([X,AY]+[AX,Y]\big)
\end{split}
\end{equation}
for any vector fields $X,Y$ on $Q$ \cite{KolarMichorSlovak1993}. We say that $A$ and $B$ \underline{commute in the sense of Fr\"olicher--Nijenhuis} if and only if $[A,B]_{\mathrm{FN}}=0$. Furthermore, the \emph{Nijenhuis tensor} of $A$ is defined by
\begin{equation}
    N_A=\frac{1}{2}\,[A,A]_{\mathrm{FN}}\,:
\end{equation}
the tensor $A$ is called \emph{Nijenhuis-integrable} if and only if $N_A=0$.

\noindent Both projectors  $\mathfrak L,\mathfrak R$ define a generalized connection on the fiber bundles $\pi_{L,R}\colon M\times M\to M$ , i.e. a splitting of the tangent space $\T(M\times M)$ into a vertical (namely $\ker\,\pi_{L,R}'$) and a horizontal (namely  $\mathrm{Im}\,\mathfrak L, \,\mathrm{Im}\,\mathfrak R$)
distribution. For both generalized connections, one easily proves (see chapter 3 in \cite{Saunders-1989}, for example) that the horizontal distributions are involutive, and therefore such connections are flat. This gives that  $\mathfrak L,\mathfrak R$ are Nijenhuis-integrable. Moreover, since the Fr\"ohlicher-Nijenhuis bracket is bilinear, with $[\mathbb I,\cdot]_{\operatorname{FN}}=0$, one directly computes
 $$
[\mathfrak L,\mathfrak R]_{\operatorname{FN}}=[\mathfrak L,\mathbb I - \mathfrak L]_{\operatorname{FN}}=-[\mathfrak L,\mathfrak L]_{\operatorname{FN}}=-2N_{\mathfrak L}= 0.
$$
We summarise what we have proven as follows.
\begin{proposition} \label[proposition]{Prop: Nijenhuis}
    Let $M$ be a smooth manifold.
     \begin{enumerate}[(a)] \item \label{Item: LR Nijenhuis integrable} $\mathfrak L$ and $\mathfrak R$ are Nijenhuis-integrable;
        \item \label{Item: LR FN commute} $\mathfrak L$ and $\mathfrak R$ commute in the sense of Fr\"olicher--Nijenhuis.
    \end{enumerate}
\end{proposition}
\noindent The pair $(\mathfrak L,\mathfrak R)$ that we have introduced in \eqref{Eq: LR properties} gives therefore
a \emph{product structure} on $M\times M$.

The smooth fiber bundles $\pi_{L,R}\colon M\times M\to M$ also allow to define the left and right lifts of a vector field  $X\in\mathfrak{X}(M)$. Define $X^{L,R}\in\mathfrak{X}(M\times M)$ by requiring that, for any $f\in C^\infty(M)$, one has
 \begin{equation}
\label{21.04.1}
\begin{cases}
\mathcal L_{X^L}(\pi_L^*f)=\pi_L^*(\mathcal L_Xf) \\ \mathcal L_{X^L}(\pi^*_Rf)=0\,,
\end{cases}
\qquad\qquad
\begin{cases}
\mathcal L_{X^R}(\pi^*_Lf)=0 \\ \mathcal L_{X^R}(\pi^*_Rf)=\pi^*_R(\mathcal L_Xf)\,.
\end{cases}
\end{equation}
Via this definition, we see that both lifted vector fields are \emph{doubly-projectable}, namely one has
 $$
\pi_L'\circ X^L=X\circ\pi_L, \qquad \pi_R'\circ X^L=0
$$
for any $X\in\mathfrak{X}(M)$, and also
$$
\pi_R'\circ Y^R=Y\circ\pi_R, \qquad \pi_L'\circ Y^R=0
$$
for any $Y\in\mathfrak{X}(M)$. This is the reason by which we can adopt the shorthand notation
 $$
X^L+Y^R=(X,Y) \in\mathfrak{X}(M\times M)
$$
and can write the flow of such a vector field as
 \begin{equation}
    \phi^{(X,Y)}_t(m,n)=\big(\phi^X_t(m),\phi^Y_t(n)\big)\,,
\end{equation}
whenever the right-hand side is defined. From the above definitions and the properties of the flow of $(X,Y)$ it is immediate to prove the following commutation relation, for vector fields $X_1,X_2,Y_1,Y_2$ on $M$:
\begin{equation}\label{Eq: triviality}
    \left[(X_1,Y_1),(X_2,Y_2)\right]=\left([X_1,X_2],[Y_1,Y_2]\right)\,.
\end{equation}

Further natural maps within the setting of the Cartesian square $M\times M$ are given by
the diagonal embedding $\iota\colon M\to M\times M$,
\begin{equation}\label{Eq: diagonal embedding}
\iota(m)=(m,m)\,,
\end{equation}
and by the flip map $\dag\colon M\times M\to M\times M$:
\begin{equation}\label{Eq: swap map}
\dag(m,n)=(n,m)\,.
\end{equation}
The following result, which intertwines the left to the right lift of a vector field on $M$, is immediate.
\begin{lemma}
For any vector fields $X,Y,Z$ on $M$, it is
\begin{align}
\iota'\circ Z&=(Z,Z)\circ \iota\,,\label{Eq: (X,X) is iota-related to X}\\
\dag'\circ (X,Y)&=(Y,X)\circ \dag\,.\label{Eq: XL is dag related to XR}
\end{align}
\end{lemma}

 Via the (canonical) product structure in $M\times M$, one can define a (canonical) bicomplex. In order to introduce it, we need to recall the notion of differential associated to a $(1,1)$-tensor $A$ on a smooth manifold $Q$ (see \cite[Section 2.4]{M-F-LV-M-R-1990}). Let  $\dd_A\colon \Omega^k(Q)\to\Omega^{k+1}(Q)$ be defined by
\begin{equation}\label{Eq: dA}
\begin{split}
(\mathrm d_A\omega)(Z_0,\dots,Z_k)&=\sum_{j=0}^k (-1)^j\, \mathcal L_{AZ_j}\left(\omega(Z_0,\dots,\check{Z_j},\dots,Z_k)\right)\\
&\quad+\sum_{0\le a<b\le k}(-1)^{a+b}\omega([AZ_a,Z_b]+[Z_a,AZ_b]-A[Z_a,Z_b],Z_0,\dots,\check{Z}_a,\dots,\check{Z}_b,\dots,Z_k)\,,
\end{split}
\end{equation}
where $Z_0,\dots,Z_k \in \mathfrak X (Q)$ and  $\check{Z_j}$ means that the vector field $Z_j$ is omitted. One proves that $\mathrm d_A$ is a cohomology operator, i.e. $\mathrm d_A\circ \mathrm d_A=0$, if and only if $A$ is Nijenhuis-integrable. Moreover, given another $(1,1)$-tensor $B$ on $Q$, $\mathrm d_A$ and $\mathrm d_B$ anticommute, i.e. $\mathrm d_A\circ \mathrm d_B=-\mathrm d_B\circ \mathrm d_A$, if and only if $A$ and $B$ commute in the sense of Fr\"olicher--Nijenhuis.

\noindent Apply this definition  to the product structure $(\mathfrak L,\mathfrak R)$ of $M\times M$. Since $\mathfrak L$ and $\mathfrak R$ are Nijenhuis-integrable (see \cref{Prop: Nijenhuis}~\ref{Item: LR Nijenhuis integrable}), $\mathrm d_{\mathfrak L}$ and $\mathrm d_{\mathfrak R}$ are cohomology operators on $M\times M$. Moreover, since they commute in the sense of Fr\"olicher--Nijenhuis (see \cref{Prop: Nijenhuis}\ref{Item: LR FN commute}), $\mathrm d_{\mathfrak L}$ and $\mathrm d_{\mathfrak R}$ anticommute. Lastly, the property $\mathfrak L+\mathfrak R=\mathbb I$ entails that
\begin{equation}\label{Eq: d=dL+dR}
    \mathrm d=\mathrm d_{\mathfrak L}+\mathrm d_{\mathfrak R}\,.
\end{equation}
We may summarize the above properties by saying that the triple
 \begin{equation}
\left(\Omega^\bullet(M\times M),\mathrm d_{\mathfrak L},\mathrm d_{\mathfrak R}\right)
\end{equation}
carries the structure of a differential bicomplex whose total differential is  the exterior Cartan differential on  $M\times M$.

\begin{example}
Let us represent the action of $\mathrm d_{\mathfrak L}$ and $\mathrm d_{\mathfrak R}$ in local coordinates. Let $\omega$ be a $k$-form on $M\times M$, with $k\in \mathbb N$. Let $(U\times V,x,y)$ be a local product chart, and write:
\begin{equation}
    \omega=
    \sum_{\substack{I,J\\ |I|+|J|=k}}\omega_{I,J}\,\mathrm dx^I\wedge\mathrm dy^J,
\end{equation}
\noindent where $I=\{i_1,\dots,i_{|I|}\}$ and $J=\{j_1,\dots,j_{|J|}\}$ are multi-indices; moreover, $\mathrm dx^I=\mathrm dx^{i_1}\wedge\dots\wedge\mathrm dx^{i_{|I|}}$ and $\mathrm dy^J=\mathrm dy^{j_1}\wedge\dots\wedge\mathrm dy^{j_{|J|}}$, with the convention $\mathrm dx^\varnothing=\mathrm dy^\varnothing=1$. The local expressions of $\mathrm d_{\mathfrak L}\omega$ and $\mathrm d_{\mathfrak R}\omega$ therefore read:
\begin{align}
    \mathrm d_{\mathfrak L}\omega
    &=\sum_{\substack{I,J\\ |I|+|J|=k}}\frac{\partial\omega_{I,J}}{\partial x^i}\,\mathrm dx^i\wedge\mathrm dx^I\wedge\mathrm dy^J , \\
    \mathrm d_{\mathfrak R}\omega
    &=\sum_{\substack{I,J\\ |I|+|J|=k}}\frac{\partial\omega_{I,J}}{\partial y^j}\,\mathrm dy^j\wedge\mathrm dx^I\wedge\mathrm dy^J\,.
\end{align}
\end{example}
We further notice that such a  decomposition realizes the exterior calculus on $M\times M$ as a sum of two anticommuting partial differential calculi along the left and right factors, and motivates the terminology of  \emph{absolute partial differential calculus} for the pair $(\mathrm d_{\mathfrak L},\mathrm d_{\mathfrak R})$ \cite{Ruse1931}.

 {\color{red}
}

\subsection{Statistical Potentials in Information Geometry}
 In the context of statistical manifolds, the construction problem for metrics and affine connections was first addressed and solved by Eguchi, who showed that suitable two-point functions suffice to define the associated geometric structures \cite{Eguchi1985,Eguchi1992}. The left-right splitting of the usual Cartan calculus on $M\times M$ described in the previous pages allows to set the action of such two-point functions within a geometrically invariant formalism.

 \begin{definition}\label{Def: contrast function}
 A \underline{contrast function} $F$ on a smooth manifold $M$ is a smooth function $F\colon M\times M\to \mathbb R$ satisfying, with $\iota\colon M\hookrightarrow M\times M$ the diagonal embedding map, the following conditions:
 \begin{enumerate}[(1)]
 \item it is $\iota^\ast F=0$,
 \item it is $\iota^\ast \left(\left(\mathrm d_{\mathfrak L}F\right)(X^L)\right)=0$, for any $X\in\mathfrak{X}(M)$,
 \item the tensor $g^F\in\Omega^1(M)\otimes_{C^\infty(M)}\Omega^1(M)$ whose action on $X,Y\in\mathfrak{X}(M)$ is defined by
  \begin{equation}
 g^F(X,Y)=\iota^\ast
 \left(\mathrm d_{\mathfrak L}\mathrm d_{\mathfrak R}F\left(X^L,Y^R\right)\right)
 \label{Eq: metric contrast E}
 \end{equation}
 is a pseudo--Riemannian metric on $M$.
 \end{enumerate}
 \end{definition}
\noindent With respect to local coordinates $(U,q)$ on $M$, the induced metric tensor $g^F$ is
 \begin{equation}
    g^F=\iota^\ast\left(\frac{\partial^2F}{\partial x^i \partial y^j}\right)\,\mathrm dq^i\otimes \mathrm dq^j=\left(\frac{\partial^2F}{\partial x^i \partial y^j}\right)\bigg|_{x=y=q}\,\mathrm dq^i\otimes \mathrm dq^j\,,
\end{equation}
\noindent where $(U\times U,x\times y)$ is the square chart induced by $(U,q)$. The first two defining conditions locally read
\begin{equation}
\begin{split}
&\iota^\ast F=F(q,q)=0\,, \\
&\iota^\ast \left((\mathrm d_{\mathfrak L}F)(X^L)\right)
= \iota^\ast \left(\pi_L^\ast (X^i\frac{\partial F}{\partial x^i}) \right)
= X^i(q)\left(\frac{\partial F}{\partial x^i}\right)\bigg|_{x=y=q}=0\,.
\end{split}
\end{equation}
\noindent Since this holds for every $X\in\mathfrak X(M)$, we get
\begin{equation}
    \left(\frac{\partial F}{\partial x^i}\right)\bigg|_{x=y=q}=0\,,\qquad i=1,\dots,\dim M\,.
\end{equation}
\noindent If $F$ is a contrast function, then the position
 \begin{equation}\label{Eq: connection contrast E} g^F\left(\nabla_Z^F X, Y\right)=\iota^\ast\left(\left(\mathcal L_{Z^L}(\mathrm d_{\mathfrak L}\mathrm d_{\mathfrak R}F)\right)\left(X^L,Y^R\right)\right)+g^F([Z,X],Y)
 \end{equation}
 \noindent implicitly defines  an affine connection $\nabla^F$ on $M$, where $X,Y,Z \in \mathfrak X(M)$ are arbitrary vector fields. To see that $\nabla^F$ and $(\nabla^F)^\dag$ are torsion-free, we observe that
 \begin{equation}
    \mathrm d_{\mathfrak L}\mathrm d_{\mathfrak R}F\left(X^L,Y^R\right)=\mathcal L_{X^L}\,\mathcal L_{Y^R}F\,:
\end{equation}
 \noindent from  the definition of torsion tensor (see \eqref{Eq: torsion tensor}), we have
  \begin{align}
     &g^F\left(\Tor^{\nabla^F}(X,Z),Y\right)=\iota^\ast\left(\left(\mathcal L_{X^L}\mathcal L_{Z^L}-\mathcal L_{Z^L}\mathcal L_{X^L}-\mathcal L_{[X,Z]^L}\right)\mathcal L_{Y^R}F\right)\,,\\
&g^F\left(X,\Tor^{\left(\nabla^F\right)^\dag}(Y,Z)\right)=\iota^\ast\left(\left(\mathcal L_{Y^R}\mathcal L_{Z^R}-\mathcal L_{Z^R}\mathcal L_{Y^R}-\mathcal L_{[Y,Z]^R}\right)\mathcal L_{X^L}F\right)\,.
 \end{align}
 By the commutation relation \eqref{Eq: triviality}, we see that the right-hand sides of the above equalities vanish, and, since $g^F$ is nondegenerate, we conclude that the torsion tensors of both $\nabla^F$ and $(\nabla^F)^\dag$ vanish.

\begin{remark}\label{Remark: sign convention}
Notice that different definitions of suitable potential functions for a statistical manifold are considered in the literature on this subject.  A \underline{yoke} is a smooth real-valued function on $M\times M$ that satisfies the conditions (2), (3) in \cref{Def: contrast function} (see \cite{BarndorffNielsenJupp1997}). A \underline{divergence function} is a smooth real-valued function  $D\colon M\times M\to\mathbb{R}$ that satisfies the conditions set in \cref{Def: contrast function} and furthermore entails $D(m,n)\>\geq0$ with $D(m,n)=0$ if and only if $m=n$ (see \cite{Eguchi1985,Eguchi1992}). In particular, in this paper we do not insert the minus sign in front of \eqref{Eq: metric contrast E}, as we treat pseudo--Riemannian metrics.
\end{remark}

\noindent  Henmi and Matsuzoe generalized contrast functions to allow for affine connections with a non-trivial torsion tensor \cite{H-M-2011}.
\begin{definition}\label{Def: pre-contrast function}
 A \underline{pre-contrast function} on a smooth manifold $M$ is a smooth map
\begin{equation} \rho\colon M\times \T M \longrightarrow \mathbb R
 \end{equation}
 such that the following conditions hold.
 \begin{enumerate}[(1)]
 \item For any $Y \in \mathfrak X(M)$, the function $\rho_{\,\mid  \,Y}\colon M\times M\to \mathbb R$, defined by $\rho_{\,\mid  \,Y}(m,n)=\rho(m,Y(n))$, satisfies $\iota^\ast\rho_{\,\mid  \,Y}=0$;
 \item \label{Cond: pre-contrast fiberwise linearity} for any $m \in M$, the map  $\rho(m,\cdot)\colon \T M \to \mathbb R$ is fiberwise linear;
 \item the tensor $g^\rho\in\Omega^1(M)\otimes_{C^\infty(M)}\Omega^1(M)$ whose action on vector fields $X,Y$ on $M$ is given by
  \begin{equation} \label{Eq: metric pre-contrast HM}
 g^\rho(X,Y)=\iota^\ast \left((\mathrm d_{\mathfrak L} \rho_{\mid Y})(X^L)\right)\,
 \end{equation}
 is a pseudo--Riemannian metric on $M$.
 \end{enumerate}
 \end{definition}
\noindent In local coordinates $(x,y,\dot y)$ on $M\times \T M$, a pre-contrast function $\rho$ can be represented as
\begin{equation}
     \rho=\rho_j(x,y)\,\dot y^j.
\end{equation}
\noindent The corresponding metric tensor is written, along a $(U,q)$ local chart on $M$, as
\begin{equation}
     g^\rho=\iota^\ast\left(\frac{\partial \rho_{j}}{\partial x^i}\right)\, \mathrm dq^i\otimes \mathrm dq^j=\left(\frac{\partial \rho_{j}}{\partial x^i}\right)\bigg|_{x=y=q}\, \mathrm dq^i\otimes \mathrm dq^j\,,
 \end{equation}
where $(U\times U,x,y)$ is the square chart induced by $(U,q)$.

\begin{remark}
In \cite{H-M-2011,H-M-2019}, pre-contrast functions are defined in accordance with the convention that the conjugate connection $\nabla^\dag$ is torsion-free.
Since in the present paper we adopt the opposite convention (cf. \cref{Remark: smat}), our definition of pre-contrast functions is modified accordingly.
\end{remark}

\noindent If $\rho$ is a pre-contrast function, then the position (with $X,Y,Z\in \mathfrak X(M)$)
 \begin{equation} \label{Eq: connection pre-contrast HM} g^\rho\left(\nabla^\rho_Z X, Y\right)=\iota^\ast\left(\left(\mathcal L_{Z^L}(\mathrm d_{\mathfrak L}\rho_{\,\mid \,Y})\right)\left(X^L\right)\right)+g^\rho([Z,X],Y)
 \end{equation}
implicitly defines an affine connection $\nabla^\rho$ on $M$. Since
 $\mathrm d_{\mathfrak L}\rho_{\,\mid  \,Y}\left(X^L\right)=\mathcal L_{X^L}\rho_{\mid Y}$,
we have
 \begin{equation}
    g^\rho\left(\Tor^{\nabla^\rho}(X,Z),Y\right)=\iota^\ast\left(\left(\mathcal L_{X^L}\mathcal L_{Z^L}-\mathcal L_{Z^L}\mathcal L_{X^L}-\mathcal L_{[X,Z]^L}\right)\rho_{\,\mid  \,Y}\right)\,.
\end{equation}
 Again by the commutation relation in \eqref{Eq: triviality}, the right-hand side of the above equality vanishes: the non-degeneracy of $g^\rho$ then implies  that $\nabla^\rho$ is torsion-free. Thus, we have that $(M, g^\rho, \nabla^\rho)$ is a SMAT. Furthermore, if $F$ is a contrast function, then the map  $\rho^F\colon M \times \T M \to \mathbb R$ defined by \begin{equation}
 \rho^F(m, (n,v)) = (\mathrm d_{\mathfrak R}F)\left(Y^R\right)(m,n)\,,
 \end{equation}
 where $Y$ is any vector field on $M$ such that $Y_n = v$, turns out to be a pre-contrast function. This shows that  the theory of contrast functions is a particular case of the theory of pre-contrast functions, just as statistical manifolds are particular cases of SMAT.

\noindent Khan and Zhang further generalized pre-contrast functions to the general case \cite{Z-K-2020}.
 \begin{definition}\label{def:scf}
 A \underline{super-contrast function} on a smooth manifold $M$ is a smooth map  \begin{equation}
 K\colon \T M\times \T M\longrightarrow \mathbb R
 \end{equation}
 satisfying the following conditions:
 \begin{enumerate}[(1)]
 \item \label{Cond: super-contrast right fiberwise linearity} for any $(m,v)\in \T M$, the map $K(m,v,\cdot)\colon \T M\to \mathbb R$ is fiberwise linear;
 \item \label{Cond: super-contrast left fiberwise linearity} for any $(n,w)\in \T M$, the map $K(\cdot,n,w)\colon \T M\to \mathbb R$ is fiberwise linear;
 \item the tensor  $g^K\in\Omega^1(M)\otimes\Omega^1(M)$ whose action on vector fields $X,Y$ on $M$ is defined by
 \begin{equation} \label{Eq: metric super-contrast}
 g^K(X,Y)=\iota^\ast K_{X\mid Y}
 \end{equation}
 is a pseudo--Riemannian metric on $M$. Here, $K_{X\mid Y}\colon M\times M\to \mathbb R$ is the function defined by $K_{X\mid Y}(m,n)= K(X(m),Y(n))$.
 \end{enumerate}
 \end{definition}
\noindent In local coordinates $(x,\dot x,y,\dot y)$ on $\T M\times \T M$, a  super-contrast function $K$ reads
  \begin{equation}
     K=K_{i\mid j}(x,y)\,\dot x^i\,\dot y^j\,,
 \end{equation}
If $(U,q)$ is a local chart of $M$, the induced metric tensor $g^K$ is
 \begin{equation}
     g^K=(\iota^\ast K_{i\mid j})\, \mathrm dq^i\otimes \mathrm dq^j=K_{i\mid j}(q,q)\,\dd q^i\otimes\dd q^j\,.
 \end{equation}
Given a super-contrast function $K$, the position (with $X,Y,Z\in\mathfrak{X}(M)$)
 \begin{equation} \label{Eq: connection super-contrast}
 g^K\left(\nabla^K_Z X, Y\right)= \iota^\ast(\mathcal L_{Z^L}(K_{X\mid Y}))
 \end{equation}
 implicitly defines  an affine connection on $M$. The theory of pre-contrast functions embeds into that of super-contrast functions as follows. If $\rho$ is a pre-contrast function on $M$, then the map $K^\rho\colon \T M\times \T M\to \mathbb R$ defined by 
 \begin{equation}
 K^\rho((m,v),(n,w)) = \left((\mathrm d_{\mathfrak L}\rho_{\,\mid \,Y})(X^L)\right)(m,n)
 \end{equation}
 where $X$ is any vector field on $M$ such that $X_m=v$, and $Y$ is any vector field on $M$ such that $Y_n=w$, is a super-contrast function.

The previous lines have fixed the terminology and the geometric framework used throughout the paper. We now explain how the language of bi-forms enters the construction of statistical potentials.

The condition \hyperref[Cond: pre-contrast fiberwise linearity]{(2)} in \cref{Def: pre-contrast function} as well as the conditions \hyperref[Cond: super-contrast right fiberwise linearity]{(1)}--\hyperref[Cond: super-contrast left fiberwise linearity]{(2)} in \cref{def:scf} show that a metric on $M$ can be obtained via suitable fiberwise multilinear maps. In order to give an alternative analysis of such potentials, consider the cotangent bundle $\tau^*_M\colon \T^*M\to M$, and the bundles induced via the submersion  $\pi_{L,R}\colon M\times M\to M$, namely
\begin{align}
&\pi_L^\ast \T^* M=\{(m,n;c,\theta)\in M\times M\times \T^* M\mid \pi_L(m,n)=\tau^*_M(c,\theta)\}\xrightarrow{\tilde\pi_L}\,M\times M \label{sez1} \\
&\pi_R^\ast \T^* M=\{(m,n;c,\xi)\in M\times M\times \T^* M\mid \pi_R(m,n)=\tau^*_M(c,\xi)\}\,\xrightarrow{\tilde\pi_R}\,M\times M, \label{sez2}
\end{align}
with the corresponding $C^{\infty}(M\times M)$-modules of (smooth) sections, that we respectively denote by
\begin{align}
&\Gamma^\infty_{\tilde\pi_L} =\{\sigma\colon M\times M\longrightarrow \pi^*_L\T^*M\mid \widetilde\pi_L\circ\sigma=\id_{M\times M}\}, \label{sez3} \\
&\Gamma^\infty_{\tilde\pi_R} =\{\sigma\colon M\times M\longrightarrow \pi^*_R\T^*M\mid \widetilde\pi_R\circ\sigma=\id_{M\times M}\}.
\label{sez4}
\end{align}
The term \underline{bi-form} arises within such a setting. We say that elements in  $\Gamma^\infty_{\tilde\pi_L}$ are $(1,0)$-bi-forms on $M$, while elements in  $\Gamma^{\infty}_{\tilde\pi_R}$ are $(0,1)$-bi-forms on $M$. Furthermore, we can consider the tensor product of the above modules of sections over $C^\infty(M\times M)$: we say that elements in $\Gamma^\infty_{\tilde\pi_L}\otimes_{C^\infty(M\times M)}\Gamma^\infty_{\tilde\pi_R}$ are $(1,1)$-bi-forms on $M$.
In terms of such definitions, we see that a pre-contrast function $\rho$ reads as a suitable $(0,1)$-bi-form, while a super-contrast function $K$ reads as a suitable $(1,1)$-bi-form on $M$. 

This observation motivates the bi-form approach developed in the rest of the
paper. The aim of this paper is to study the geometry of bi-forms and to show
that suitable cohomological properties of bi-forms encode the existence of
potential functions in information geometry, recovering metric tensors and affine
connections both in the torsion-free and in the torsion-full settings.

\paragraph{Outline.}
The paper is organised as follows. In \cref{Sec: Geometry of bi-forms}, we describe the geometry of bi-forms and their equivalent formulations, providing a Cartan-like calculus composed of interior products, exterior derivatives, and Lie derivatives, distinguished into left and right types. We then study the cohomology of the left and right exterior derivatives in detail in \cref{Sec: cohomology}. In \cref{Section: Contrast bi-forms in Information Geometry}, the formalism of bi-forms is applied to information geometry, and the cohomological framework is related to the passage from contrast bi-forms to pre-contrast and contrast functions. In \cref{Sec: dually curvaturefree LM}, we introduce dually curvature-free Lauritzen manifolds and construct a canonical contrast bi-form within this setting. Finally, in \cref{Sec: examples} and \cref{Subsec: Lie group}, we illustrate the framework with two classes of examples: faithful states on finite-dimensional $C^\ast$-algebras (recovering the standard dually flat structure of the probability simplex, and SMAT structures on the manifold of faithful quantum states on a finite-dimensional Hilbert space associated to quantum monotone metric tensors), and semisimple Lie groups equipped with the Killing metric and left and right Cartan connections.

\section{The geometry of bi-forms}\label{Sec: Geometry of bi-forms} 
The first aim of this section is to refine and generalise the above definition of a bi-form as a section of a suitable  vector bundle over the Cartesian square manifold $M\times M$, as well as to reinterpret it both as a blockwise fiberwise multilinear map and as differential forms of mixed type on $M \times M$. We then define the diagonal restriction and dual operator, and  introduce the Cartan bi-calculus with its left and right interior products, exterior derivatives, and Lie derivatives. 

\subsection{Bi-forms on a smooth manifold}
\label{subs:b}
\begin{definition}
\label{25.04.d}
Consider a  smooth manifold $M$, with  $\pi_{L,R}\colon M\times M\to M$  the smooth projections introduced in \eqref{Eq: left and right projs}. Let\footnote{We denote as $\mathbb N_0$  the set of natural numbers and as $\mathbb N = \mathbb N_0 \setminus \{0 \}$.}  $(p,q) \in \mathbb{N}_0 \times \mathbb{N}_0$. A $(p,q)$-\underline{bi-form} on $M$ is a smooth section of the vector bundle 
\begin{equation}
    \bigwedge^{p,q}(M \mid M) = \bigwedge^p \pi_L^\ast \T^\ast M \otimes_{M \times M} \bigwedge^q \pi_R^\ast \T^\ast M\longrightarrow M\times M.
\end{equation}
\end{definition}
\noindent The fibers of such a vector bundle are  
\begin{equation}
    \left( \bigwedge^{p,q}(M \mid M) \right)_{(m,n)} = \bigwedge^p \T_m^\ast M \otimes \bigwedge^q \T_n^\ast M\,
\end{equation}
at each base point $(m,n)\in M\times M$, and we denote by  $\Omega^{p,q}(M \mid M)$ the $C^\infty(M \times M)$-module of sections of this bundle. Given $\varpi\in \Omega^{p,q}(M\mid M)$, we say that $p$ is the \underline{left degree} of $\varpi$, and that $q$ is the \underline{right degree} of $\varpi$. We refer to the pair $(p,q)$ as the \underline{bi-degree} of $\varpi$.

\medskip

\noindent One can consider bi-forms as fiberwise blockwise-multilinear maps, as much as like one can consider exterior  differential forms as multilinear alternating real-valued functions. For a $p$-form $\omega \in \Omega^p(M)$, define $\dot \omega\colon (\T M)^{\oplus p}\to \mathbb R$ by
\begin{equation} 
\dot\omega((m,u_1), \dots, (m,u_p))=\omega_m(u_1, \dots, u_p)\,,
\end{equation}
where $(\mathrm{T} M)^{\oplus p}$ denotes the $p$-fold Whitney sum of the tangent bundle over $M$. This extends to $(p,q)$-bi-forms as follows. For any $\varpi \in \Omega^{p,q}(M \mid M)$, we define $\dot \varpi\colon (\mathrm{T} M)^{\oplus p}\times (\mathrm{T} M)^{\oplus q}\to \mathbb R$ by
\begin{align} 
\dot{\varpi}((m,u_1), \dots, (m,u_p)\mid (n,v_1), \dots, (n,v_q)) =\varpi_{(m,n)}(u_1, \dots, u_p \mid v_1, \dots, v_q)\,.
\label{Eq: dot bi-form}
\end{align}  
  This highlights the $(p,q)$-blockwise alternating fiberwise multilinear structure of $(p,q)$-bi-forms.
  
\medskip

\noindent Bi-forms on $M$ pair to vector fields on $M$ as follows. Let $(p,q)\in \mathbb N_0\times \mathbb N_0$, $\varpi$ be a $(p,q)$-bi-form on $M$, and consider vector fields $\{X_i\}_{i=1}^p$ and $\{Y_j\}_{j=1}^q$ on $M$. Then, for any $(m,n)$, set 
\begin{equation}\label{Eq: pairing} 
\varpi(X_1,\dots,X_p\mid Y_1,\dots,Y_q)(m,n)=\varpi_{(m,n)}((X_1)_m,\dots,(X_p)_m\mid (Y_1)_n,\dots,(Y_q)_n)\,. 
\end{equation} 
 
It is immediate to see that the following properties hold.
\begin{enumerate}[1)] 
\item The pairing associated to $\varpi$ is $(p,q)$-\underline{blockwise alternating}, that is, for any permutations $\sigma$ of $\{1,\dots,p\}$ and $\rho$ of $\{1,\dots,q\}$, we have
\begin{equation} 
\varpi(X_{\sigma(1)},\dots,X_{\sigma(p)}\mid Y_{\rho(1)},\dots,Y_{\rho(q)})=(-1)^\sigma\,(-1)^\rho\,\varpi(X_1,\dots, X_p\mid Y_1,\dots,Y_q)\,.
\end{equation} 
\item The pairing associated to $\varpi$ is both \underline{left} and \underline{right tensorial}, that is, for any $f,h\in C^\infty(M)$, we have 
\begin{equation}\label{Eq: tensorial property of bi-form} 
\varpi(f\,X_1,X_2,\dots,X_p\mid h\,Y_1,Y_2,\dots,Y_q)=(\pi_L^\ast f)\, (\pi_R^\ast h)\,\varpi(X_1,\dots,X_p\mid Y_1,\dots,Y_q)\,.
\end{equation} 
\end{enumerate} 
Conversely, any pairing satisfying such properties uniquely determines a $(p,q)$-bi-form on $M$. For this reason we identify the two notions, and use  the same symbol both for representing the bi-form and the induced pairing.

\medskip

\noindent We begin to analyse the relations between \underline{bi-forms} on $M$, i.e. elements $\varpi\in\Omega^{p,q}(M\mid M)$, and exterior forms on $M\times M$, i.e. elements in $\Omega(M\times M)$. As a first step, we consider the following definition.

\begin{definition}
Let $M$ be a smooth manifold, and $(p,q)\in \mathbb N_0\times \mathbb N_0$. A $(p,q)$-bi-form $\varpi$ is called \underline{decomposable} if and only if there are  $\alpha \in \Omega^p(M)$ and $\beta \in \Omega^q(M)$ such that:
\begin{equation}
    \varpi= \pi_L^\ast(\alpha) \otimes \pi_R^\ast(\beta)\,.
\end{equation}
 If a bi-form $\varpi$ turns to be decomposable as above, we denote it as $\varpi=\alpha\boxtimes \beta$. 
\end{definition}
\noindent If $(U, q)$ and $(V, z)$ are local  charts on $M$, then a general $(p,q)$-bi-form $\varpi$ can be written on $U \times V$ as:
\begin{equation}\label{Eq: local decomposition bi-forms}
    \varpi = \varpi_{I\mid J} \, \left(\mathrm dq^{I}\boxtimes \mathrm d z^{J}\right)\,,
\end{equation}
where $I$ and $J$ are multi-indices with $|I|=p$ and $|J|=q$, and $\varpi_{I \mid J} \in C^\infty(U \times V)$. 
It is also possible to prove (see \cite[Proposition 2.3.10]{Saunders-1989}) that, given $\varpi\in \Omega^{p,q}(M\mid M)$, there exist finite families $\{\alpha^i\}_{i=1}^{n_\alpha}\subset \Omega^p(M),\,\{\beta^j\}_{j=1}^{n_\beta}\subset \Omega^q(M)$ and functions $\{F_{ij}\}\subset C^\infty(M\times M)$ such that
\begin{equation}\label{Eq: global decomposition of bi-forms}
\varpi= F_{ij}\, \alpha^i\boxtimes \beta^j\,.
\end{equation}

Consider the $C^\infty(M\times M)$-module homomorphism $\Phi^{p,q}\colon \Omega^{p,q}(M\mid M)\to \Omega^{p+q}(M\times M)$
defined, for each $\alpha\in \Omega^p(M)$ and $\beta\in \Omega^{q}(M)$ by
    \begin{equation}
    \label{26.04.1}
        \Phi^{p,q}(\alpha\boxtimes \beta)=\pi_L^\ast \alpha\wedge\pi_R^\ast \beta\,.
    \end{equation}
Since the tensor product in the \cref{25.04.d} is given with respect to  the $C^\infty(M\times M)$ algebra, one sees that $\Phi^{p,q}$ is injective, and therefore, recalling \eqref{Eq: global decomposition of bi-forms}, provides an embedding of $C^\infty(M \times M)$-module of bi-forms on $M$ into a $C^\infty(M \times M)$-submodule of differential forms on the Cartesian product $M \times M$, that is we can look at $(p,q)$-bi-forms on $M$ as suitable differential $(p+q)$-forms on $M\times M$. 
The pairing of $\varpi$ relates to the one of the differential forms $\Phi(\varpi)$ as follows.
\begin{equation}\label{Eq: pairing with forms} 
\varpi(X_1,\dots,X_p\mid Y_1,\dots,Y_q)=\Phi^{p,q}(\varpi)\left(X_1^L,\dots,X_p^L,Y_1^R,\dots,Y_q^R\right)\,. 
\end{equation}
By a direct computation, one proves that
  \begin{equation}\label{Eq: pq decomposition}
      \Omega^k(M\times M)=\bigoplus_{p=0}^k \im \Phi^{p,k-p}\,,
\end{equation}
i.e. that any exterior form on $M\times M$ can be decomposed as a sum of images of bi-forms under the action of the $\Phi$ maps. 
For each $(p,q)\in \mathbb N_0\times \mathbb N_0$ with $p+q=k$, we denote by
\begin{equation}\label{Eq: pq projector}
      \pi_{p,q}\colon \Omega^{p+q}(M\times M)\longrightarrow \Omega^{p+q}(M\times M)\,
  \end{equation}
  the projectors of $\Omega^k(M\times M)$, with
  \begin{equation}
      \ker \pi_{p,q}=\bigoplus_{\substack{a=0\\ a\ne p}}^k\im \Phi^{a,k-a}\,, \qquad \im \pi_{p,q}=\im \Phi^{p,q}\,.
  \end{equation}  
In local product coordinates $(U\times V,x,y)$, the action of the  projectors $\pi_{p,q}$ read:
\begin{equation}
    \pi_{p,q}\left(\sum_{\substack{I,J\\|I|+|J|=k}}\varpi_{ I\mid J}\,\mathrm dx^I\wedge \mathrm dy^J\right)= \sum_{|I|=p, |J|=q}\varpi_{I\mid J}\,\mathrm dx^I\wedge \mathrm dy^J.
\end{equation}

\medskip

\noindent By means of the diagonal embedding $\iota$ (cf. \eqref{Eq: diagonal embedding} ), any $(p,q)$-bi-form induces a $(p,q)$-blockwise alternating $(p+q)$-covariant tensor on $M$, namely a smooth section of the vector bundle 
  \begin{equation} 
  \bigwedge^p\T^\ast M\otimes_M \bigwedge^q\T^\ast M\longrightarrow M. 
  \end{equation}
In a local chart $(U,q)$ on $M$,   any such section $G$ is represented as
  \begin{equation}
  G = G_{IJ}\,\mathrm{d}q^I \otimes \mathrm{d}q^J \,,
  \end{equation}
  where $I$ and $J$ are multi-indices with $|I|=p$ and $|J|=q$,  and $G_{IJ}\in C^\infty(U)$.

  \begin{definition}\label{Def: diagonal restriction} Let $(p,q)\in \mathbb N_0\times \mathbb N_0$, and $\varpi$ be a $(p,q)$-bi-form on $M$. The \underline{diagonal restriction} of $\varpi$ is the $(p+q)$-covariant tensor on $M$ defined by
  \begin{equation}
  \label{def:dr}
  \varpi_{\restriction\Delta_M}(Z_1,\dots,Z_{p+q})= \iota^\ast\big(\varpi(Z_1,\dots,Z_p\mid Z_{p+1},\dots,Z_{p+q})\big)\,,
  \end{equation}
  where $\{Z_i\}_{i=1}^{p+q}$ is a set of vector fields on $M$.
  \end{definition} 
\noindent  In a local chart $(U, q)$ on $M$, one has 
  \begin{equation}
  \varpi_{\restriction \Delta_M} = (\iota^\ast\varpi_{I\mid J})\, \mathrm{d}q^I \otimes \mathrm{d}q^J \,,
  \end{equation}
  where $\varpi=\varpi_{I\mid J}\, \left(\mathrm dq^I\boxtimes \mathrm dq^J\right)$ is the local expression of $\varpi$ in the square chart of $M\times M$ associated to $(U,q)$.
Notice that $\varpi_{\restriction \Delta_M}$ is a $(p,q)$-blockwise alternating $(p+q)$-covariant tensor on $M$ because the pairing associated to $\varpi$ is $(p,q)$-blockwise alternating and left and right tensorial\footnote{In \cite{GSI} we used the notation $\iota^\ast\varpi$ for this restriction. We adopt the current notation to avoid potential confusion with $\iota^\ast\Phi(\varpi)$, which generally differs from $\Phi(\iota^\ast\varpi)$ (see \cref{Lemma: initial condition}).}. 

\noindent The restriction operator to the diagonal can be defined more generally for local bi-forms defined on an open neighbourhood of the diagonal in the Cartesian square. We use the symbol $\Omega^{p,q}_{\Delta_M}(M\mid M)$ to denote the collection of such local bi-forms defined on neighbourhoods of the diagonal $\Delta_M$ in $M\times M$. In the following, we simply refer to them as \underline{local bi-forms}.

\noindent   There is also a natural dual correspondence on bi-forms given by the swap map $\dag$ (cf. \eqref{Eq: swap map}). It provides a systematic way to interchange left and right structures. 
  
  \begin{definition}
  Let $(p,q)\in \mathbb N_0\times \mathbb N_0$, $\varpi \in \Omega^{p,q}(M \mid M)$. We define the bi-form $\varpi^\dag \in \Omega^{q,p}(M \mid M)$ by the action 
  \begin{equation} 
  \varpi^\dag (X_1,\dots,X_q \mid Y_1,\dots,Y_p) = \dag^\ast \big( \varpi(Y_1,\dots,Y_p \mid X_1,\dots,X_q) \big) \,, 
  \end{equation}
for vector fields $\{X_i\}_{i=1}^q,\,\{Y_j\}_{j=1}^p$ on $M$. We say that $\varpi^\dag$ is the bi-form \underline{dual} to  $\varpi$.
  \end{definition}
\noindent   In local coordinates, if $(U,q)$ and $(V,z)$ are charts on $M$, the dual bi-form of $\varpi$ reads
  \begin{equation}
      \varpi^\dag=\dag^\ast \varpi_{I\mid J}\,\left(\mathrm dz^J\boxtimes\mathrm dq^I\right)\,,
  \end{equation}
  where $\varpi=\varpi_{I\mid J}\,\left(\mathrm dq^I\boxtimes \mathrm dz^J\right)$ is the local expression of $\varpi$ on the product chart associated to $(U,q)$ and $(V,z)$.
  Equivalently, $\varpi^\dag$ can be defined by the identity
  \begin{equation}\label{Eq: sast on forms}
      \Phi^{q,p}(\varpi^\dag)=(-1)^{p\,q}\,\dag^\ast\left(\Phi^{p,q}(\varpi)\right)\,.
  \end{equation}
  
  Note that the relation $\dag \circ \iota=\iota$ implies
  \begin{equation}\label{Eq: sast iotaast}
  (\varpi^\dag)_{\restriction\Delta_M}(Z_1,\dots,Z_q,Z_{q+1},\dots,Z_{q+p})=\varpi_{\restriction\Delta_M}(Z_{q+1},\dots,Z_{q+p},Z_1,\dots,Z_q)
  \end{equation}
for vector fields $\{Z_k\}_{k=1}^{p+q}$ on $M$.

\subsection{Cartan calculus on bi-forms}
\label{subs:Cc}
We can now introduce a suitable left and right Cartan calculus on bi-forms on
$$
\Omega^{\bullet,\bullet}(M \mid M)=\bigoplus_{(p,q)\in \mathbb Z\times \mathbb Z}\Omega^{p,q}(M\mid M)\,,
$$
where we assume  $\Omega^{p,q}(M\mid M)=0$ whenever $p<0$ or $q<0$. The action of the operators we introduce below resemble (in a way which we shall soon clarify) those providing  the usual Cartan calculus.

\begin{definition}
Let $M$ be a smooth manifold, with $Z\in\mathfrak{X}(M)$. 
\begin{enumerate}[1)]
\item The map $\mathbf{i}_{Z\mid}\colon \Omega^{p,q}(M\mid M)\to\Omega^{p-1,q}(M\mid M)$ defined by 
\begin{equation}
\label{26.04.d1}
(\mathbf{i}_{Z\mid} \varpi)(X_2,\dots,X_p \mid Y_1,\dots,Y_q) = \varpi(Z, X_2, \dots, X_p \mid Y_1, \dots, Y_q)
\end{equation}
is the \underline{left interior product} of $\varpi$ with  $Z$, or the left interior \underline{contraction} of $\varpi$ along $Z$. Analogously:
\item the map $\mathbf{i}_{\mid Z}\colon \Omega^{p,q}(M\mid M)\to\Omega^{p,q-1}(M\mid M)$ defined by 
\begin{equation}
\label{26.04.d2}
(\mathbf{i}_{\mid Z} \varpi)(X_1,\dots,X_p \mid Y_2,\dots,Y_q) = \varpi(X_1, \dots, X_p \mid Z, Y_2, \dots, Y_q)
\end{equation}
is the \underline{right interior product} of $\varpi$ with $Z$, or the right interior \underline{contraction} of $\varpi$ along $Z$.
\end{enumerate}

\end{definition}
\noindent It is immediate to verify that the map $\Phi^{p,q}$ defined above in \eqref{26.04.1} intertwines the action of the left (resp. right)  interior contraction along $Z$  on bi-forms with the action of the standard Cartan interior contraction of an exterior form on $M\times M$ along the left (resp. right) lift of $Z$.  One has that:
\begin{equation}
\label{26.04.2}
\begin{split} 
\Phi^{p-1,q}\circ \mathbf i_{Z\mid }
&=\mathbf i_{Z^L}\circ \Phi^{p,q}\,,\\
\Phi^{p,q-1}\circ \mathbf i_{\mid Z}
&=(-1)^p\,\mathbf i_{Z^R}\circ \Phi^{p,q}\,.
\end{split}
\end{equation}
\noindent The following definition comes upon recalling the definition (see \eqref{21.04.1}) of left and right lift $Z^{L,R}\in\mathfrak{X}(M\times M)$ of a vector field $Z\in\mathfrak{X}(M)$. 
\begin{definition} 
Let $M$ be a smooth manifold, with $Z\in\mathfrak{X}(M)$. 
\begin{enumerate}[1)]
\item The map $\mathcal{L}_{Z\mid}\colon \Omega^{p,q}(M\mid M)\to\Omega^{p,q}(M\mid M)$ defined by 
\begin{equation}
\label{Eq:left_Lie_derivative}
\begin{split}
(\mathcal{L}_{Z\mid} \varpi)(X_1,\dots,X_p \mid Y_1,\dots,Y_q)
&= \mathcal{L}_{Z^L} \left(\varpi(X_1,\dots,X_p \mid Y_1,\dots,Y_q)\right) \notag \\
&\quad + \sum_{i=1}^p (-1)^i\, \varpi([Z,X_i],X_1,\dots,\check{X}_i,\dots,X_p \mid Y_1,\dots,Y_q) \,,
\end{split}
\end{equation}
(where $\check{X}$ means that the vector field $X$ is omitted) is the \underline{left Lie derivative} of $\varpi$ along $Z$. Analogously:
\item the map $\mathcal{L}_{\mid Z}\colon \Omega^{p,q}(M\mid M)\to\Omega^{p,q}(M\mid M)$ defined by 
\begin{equation}
\label{Eq:right_Lie_derivative}
\begin{split}
(\mathcal{L}_{\mid Z} \varpi)(X_1,\dots,X_p \mid Y_1,\dots,Y_q)
&= \mathcal{L}_{Z^R} \left(\varpi(X_1,\dots,X_p \mid Y_1,\dots,Y_q)\right) \notag \\
&\quad + \sum_{j=1}^q (-1)^j\, \varpi(X_1,\dots,X_p \mid [Z,Y_j],Y_1,\dots,\check{Y}_j,\dots,Y_q) \,
\end{split}
\end{equation}
is the \underline{right Lie derivative} of $\varpi$ along $Z$. 
\end{enumerate}
\end{definition}
\noindent For such operators, the action of the intertwining map (compare it with \eqref{26.04.2}) results in
\begin{equation}
\label{26.04.4}
\begin{split}  
\Phi^{p,q}\circ \mathcal L_{Z\mid}&=\mathcal L_{Z^L}\circ \Phi^{p,q}\,,\\
\Phi^{p,q}\circ \mathcal L_{\mid Z}&=\mathcal L_{Z^R}\circ \Phi^{p,q}\,.
\end{split}
\end{equation}

\begin{definition}
\label{def:lred}
Let $M$ be a smooth manifold. 
\begin{enumerate}[1)]
\item The \underline{left exterior derivative} is the map $\dd^L\colon \Omega^{p,q}(M\mid M)\to\Omega^{p+1,q}(M\mid M)$ defined by 
\begin{equation}
\label{Eq:dL}
\begin{split}
(\mathrm{d}^L \varpi)(X_0,\dots,X_p \mid Y_1,\dots,Y_q)
&= \sum_{i=0}^p (-1)^i\, \mathcal{L}_{X_i \mid} \varpi(X_0,\dots,\check{X}_i,\dots,X_p \mid Y_1,\dots,Y_q)  \\
&\quad + \sum_{0 \le a < b \le p} (-1)^{a+b} \, \varpi([X_a,X_b], X_0,\dots,\check{X}_a,\dots,\check{X}_b,\dots,X_p \mid Y_1,\dots,Y_q) \,.
\end{split}
\end{equation}
\item The \underline{right exterior derivative} is the map $\dd^R\colon \Omega^{p,q}(M\mid M)\to\Omega^{p,q+1}(M\mid M)$ defined by  
\begin{equation}
\label{Eq:dR}
\begin{split}
(\mathrm{d}^R \varpi)(X_1,\dots,X_p \mid Y_0,\dots,Y_q)
&=\sum_{j=0}^q (-1)^j\, \mathcal{L}_{\mid Y_j} \varpi(X_1,\dots,X_p \mid Y_0,\dots,\check{Y}_j,\dots,Y_q) \\
&\quad +\sum_{0 \le a < b \le q} (-1)^{a+b} \, \varpi(X_1,\dots,X_p \mid [Y_a,Y_b], Y_0,\dots,\check{Y}_a,\dots,\check{Y}_b,\dots,Y_q) \,,
\end{split}
\end{equation}
where again  $\check{X}_i$ (resp. $\check{Y}_j$) means that the vector field $X_i$ (resp. $Y_j$) is omitted.
\end{enumerate}
\end{definition}
\noindent As for the previous operators we have introduced, one easily proves that 
\begin{align}
   \Phi^{p+1,q}\circ \mathrm d^L&=\mathrm d_{\mathfrak L}\circ \Phi^{p,q}\,,\label{Eq: d_L and d^L}\\
    \Phi^{p,q+1}\circ \mathrm d^R&=(-1)^p\,\mathrm d_{\mathfrak R}\circ \Phi^{p,q}\,.\label{Eq: d_R and d^R}
\end{align}
\noindent Moreover, the 
 left and right operators are related by the duality operator as follows:
\begin{align}
\mathbf i_{\mid Z}\varpi &= \left(\mathbf i_{Z\mid}(\varpi^\dag)\right)^\dag\,,\label{Eq: duality interior product} \\
\mathrm{d}^R \varpi &= \left(\mathrm{d}^L\varpi^\dag\right)^\dag \,, \label{Eq: duality exterior derivative}\\
\mathcal{L}_{\mid Z} \varpi &= \left(\mathcal{L}_{Z\mid} \varpi^\dag\right)^\dag \,.\label{Eq: duality Lie derivative}
\end{align}

\noindent The following proposition addresses the compatibility of the diagonal restriction with the left/right Cartan calculus on bi-forms, describing the induced identities for Lie derivatives and exterior differentials.
\begin{proposition}
Let $M$ be a smooth manifold, and $(p,q)\in \mathbb N_0\times \mathbb N_0$. Consider on $M$ bi-forms $\varpi\in \Omega^{p,q}_{\Delta_M}(M\mid M)$, $T\in \Omega^{p,0}_{\Delta_M}(M\mid M)$ and $S\in \Omega^{0,q}(M\mid M)$.
\begin{enumerate}[(a)]
    \item For any vector field $Z$ on $M$:
    \begin{equation}
               \mathcal L_{Z}(\varpi_{\restriction \Delta_M})=(\mathcal L_{Z\mid }\varpi)_{\restriction\Delta_M}+ (\mathcal L_{\mid Z}\varpi)_{\restriction\Delta_M}\label{Eq: iota ast Lie derivative}\,.
    \end{equation}
    \item  For any pair of families of vector fields $\{X_i\}_{i=0}^p$ and $\{Y_j\}_{j=0}^q$ on $M$:
    \begin{align}
        \left(\mathrm d(T_{\restriction \Delta_M})-(\mathrm d^LT)_{\restriction \Delta_M}\right)(X_0,\dots,X_p)&=\sum_{i=0}^p(-1)^i\, (\mathrm d^RT)_{\restriction \Delta_M}(X_0,\dots,\check{X_i},\dots,X_p\mid X_i) \,,\label{Eq: dL and iota ast do not commute}\\
        \left(\mathrm d(S_{\restriction \Delta_M})-(\mathrm d^RS)_{\restriction \Delta_M}\right)(Y_0,\dots,Y_q)&=\sum_{j=0}^q(-1)^j\, (\mathrm d^LS)_{\restriction \Delta_M}(Y_j\mid Y_0,\dots,\check{Y_j},\dots,Y_q)\,.\label{Eq: dR and iota ast do not commute}
    \end{align}
    \end{enumerate}
\end{proposition}
\begin{proof}
The equality \eqref{Eq: iota ast Lie derivative} follows from applying the global formula for the Lie derivative \cite[Proposition 7.4.11, (i)]{A-A-2015} to $\varpi_{\restriction\Delta_M}$ and using the relation \eqref{Eq: (X,X) is iota-related to X}. Similarly, the identities \eqref{Eq: dL and iota ast do not commute} and \eqref{Eq: dR and iota ast do not commute} follow by applying the global formula for the exterior derivative \cite[Proposition 7.4.11, (ii)]{A-A-2015}.
\end{proof}

\section{Cohomology of bi-forms}\label{Sec: cohomology}
In this section we analyze the cohomology of the exterior derivatives by constructing abstract homotopy operators and their associated projectors onto (anti-)exact components. We then realize these operators explicitly using an affine connection $\nabla$ and the associated geodesic interpolation within strongly $\nabla$-convex neighbourhoods. We conclude by analysing  properties which turn to be essential for applications to information geometry, introducing the notions of statistical left and right homotopy operators.

\noindent Recall from \cref{Subsec: geometric preliminaries} that $(\mathrm d_{\mathfrak L},\mathrm d_{\mathfrak R})$ are anticommuting cohomology operators on differential forms on $M\times M$. In particular, relations \eqref{Eq: d_L and d^L} and \eqref{Eq: d_R and d^R} imply, via the identification map $\Phi$, that $\mathrm d^L$ and $\mathrm d^R$ are cohomology operators on bi-forms:
\begin{align}
    \mathrm d^L\circ \mathrm d^L&=0\,,\\
    \mathrm d^R\circ \mathrm d^R&=0\,.
\end{align}
Moreover, they commute. Indeed, if $\varpi$ is a $(p,q)$-bi-form on $M$, with $(p,q)\in \mathbb N_0\times \mathbb N_0$, then:
  \begin{align}
    \Phi^{p+1,q+1}(\mathrm d^L\mathrm d^R\varpi)&=(-1)^p\,\mathrm d_{\mathfrak L}\,\mathrm d_{\mathfrak R}\Phi^{p,q}(\varpi)\,,\\
    \Phi^{p+1,q+1}(\mathrm d^R\mathrm d^L\varpi)&=(-1)^{p+1}\,\mathrm d_{\mathfrak R}\,\mathrm d_{\mathfrak L}\Phi^{p,q}(\varpi)\,.
\end{align} 
Since $\mathrm d_{\mathfrak L}$ and $\mathrm d_{\mathfrak R}$ anticommute, the two expressions coincide. We now analyze the homotopic invertibility of the left and right exterior derivatives in an open neighbourhood of the diagonal submanifold.

\subsection{Left and right homotopy operators}
\label{subs:LR}
In analogy to the notion of homotopy operators for differential forms on a smooth manifold (see \cite{Edelen-1985-AEC}), we introduce left and right homotopy operators.
  \begin{definition} \label{Def: left homotopy operators}
 Let $M$ be a smooth manifold. 
A \underline{system of left homotopy operators} is a family of $\mathbb R$-linear operators 
           \begin{equation}
        J_L^{p,q}\colon \Omega^{p,q}_{\Delta_M}(M\mid M) \longrightarrow \Omega^{p-1,q}_{\Delta_M}(M\mid M)\,, \qquad (p,q)\in \mathbb N_0\times \mathbb N_0\,,
    \end{equation} 
that satisfy the following conditions: 
  \begin{itemize}
  \item it is $J_L^{p-1,q}\circ J_L^{p,q}=0$ ; 
  \item for any $\varpi\in \Omega^{p,q}_{\Delta_M}(M\mid M)$ there is an open neighbourhood $\mathscr U$ of $\Delta_M$ in $M\times M$, contained in the domain of $\varpi$, such that the relations
    \begin{equation}\label{Eq: homotopy formula bi-forms L}
    \begin{cases}
\begin{aligned}
\mathrm d^L\,J_L^{p,q}\,\varpi+J_L^{p+1,q}\,\mathrm d^L \varpi
&= \varpi\,,
&& (p,q)\in \mathbb N\times \mathbb N_0\,, \\
J_L^{1,q}\,\mathrm d^L \varpi
&= \varpi - 1\boxtimes \varpi_{\restriction \Delta_M}\,,
&& \textup{otherwise}\,
\end{aligned}
\end{cases}
\end{equation}
hold on $\mathscr U$.
\end{itemize}
\end{definition}
We shall often denote the system collectively by $J_L$ and suppress the bi-degree $(p,q)$ whenever no ambiguity arises.
Before explicitly constructing them, we analyse the algebraic properties that stem from the homotopy formula \eqref{Eq: homotopy formula bi-forms L}.
We begin by defining the linear operators in $\Omega^{p,q}_{\Delta_M}(M\mid M)$
  \begin{equation}
\label{defEA}
\begin{split}
&\E_L^{p,q}\varpi=\dd^L(J^{p,q}_L\varpi)\,, \\
&\A_L^{p,q}\varpi=J_L^{p+1,q}\dd^L\varpi\,.
\end{split}
\end{equation} 
One has:
\begin{itemize}
\item For $p=0$, one has trivially $\mathcal E_L^{p,0}=0$, and so $\mathcal E_L^{0,q}\mathcal E_L^{0,q}=\mathcal E_L^{0,q}$.  For $p\geq1$, one has the following chain of equalities, which come by recalling the homotopy identity \eqref{Eq: homotopy formula bi-forms L}   and the nilpotence of $J_L$:
\begin{equation}
\label{Epro}
\mathcal E_L^{p,q}\,(\mathcal E_L^{p,q}\varpi)=\mathrm d^L\left(J_L^{p,q}\left(\mathrm d^LJ_L^{p,q}\varpi\right)\right)=\mathrm d^L\left(J_L^{p,q}\left(\varpi-J_L^{p+1,q}\mathrm d^L\varpi\right)\right)=\mathrm d^L(J_L^{p,q}\varpi)=\mathcal E_L^{p,q}\varpi\,.
\end{equation} 
\item For $p=0$, one has
\begin{equation}
\mathcal A_L^{0,q}\varpi=\varpi-1\boxtimes \varpi_{\restriction \Delta_M}\,
\end{equation}  
from which one deduces $\mathcal A_L^{0,q}\mathcal A_L^{0,q}=\mathcal A_L^{0,q}$. For $p\geq1$, the proof of the identity $\A_L^{p,q}\left(\A_L^{p,q}\varpi\right)=\A_L^{p,q}\varpi$  is analogue to the previous \eqref{Epro}, in terms of  the nilpotence of $\mathrm d^L$ instead.
\item From   $\dd^L\circ\dd^L=0$ and $J_L^{p,q}\circ J_L^{p+1,q}=0$, it is also immediate to see that, for any $\varpi\in\Omega^{p,q}_{\Delta_M}(M\mid M)$,  one has
  \begin{equation}
\label{EAc}
\E^{p,q}_L\left(\A_L^{p,q}\varpi\right)=\A_L^{p,q}\left(\E_L^{p,q}\varpi\right)=0\,.
\end{equation} 
\item For $p\geq1$, the relation \eqref{Eq: homotopy formula bi-forms L} allow to write 
  \begin{equation}
\label{udec}
\E_L^{p,q}\varpi+\A_L^{p,q}\varpi=\varpi\,.
\end{equation} 
\end{itemize}
  Thus both $\E_L^{p,q}$ and $\A_L^{p,q}$ are $\mathbb{R}$-linear projectors on each $\Omega^{p,q}_{\Delta_M}(M\mid M)$, and for $p\geq1$, they provide a resolution of the identity. It is therefore standard linear algebra  to see that one can write, for $p\geq1$:
  \begin{equation}
\label{splEAL}
\begin{split}
&\im\A_L^{p,q}\oplus\ker\A_L^{p,q}=\Omega^{p,q}_{\Delta_M}(M\mid M)=\im\E_L^{p,q}\oplus\ker\E_L^{p,q}\,,\\
&\im\A_L^{p,q}=\ker\E_L^{p,q}\,,\\ &\im\E_L^{p,q}=\ker\A_L^{p,q}\,
\end{split}
\end{equation} 
Together with the above lines, one proves the following result.
\begin{theorem}\label{Thm: left projectors}
    Let $M$ be a smooth manifold, and consider a system of left homotopy operators $J_L$.
    \begin{enumerate}[(a)]

        \item \label{Item: left projectors exact part}  For any $(p,q)\in \mathbb N_0\times \mathbb N_0$, we have:
       \begin{align}
           \mathrm d^L\im \mathcal A^{p,q}_L&=\im\mathcal E^{p+1,q}_L \label{Eq: dL A=E}\,.
       \end{align}
       \item \label{Item: left projectors decomposition}  For each $(p,q)\in \mathbb N_0\times \mathbb N_0$, we have the vector space decomposition:
       \begin{equation}\label{Eq: decomposition left exact and left antiexact}
           \Omega^{p+1,q}_{\Delta_M}(M\mid M)=
          \mathrm d^L\im \mathcal A_L^{p,q}\oplus \im\mathcal A_L^{p+1,q}\,.
       \end{equation}

    \end{enumerate}
\end{theorem}
\begin{proof}
Let $\varpi\in \im \mathcal E_L^{p+1,q}$. By \eqref{Eq: homotopy formula bi-forms L} and \eqref{splEAL}, it is  $\varpi=\mathrm d^L\,J_L^{p+1,q}\varpi$. Via the same identities one has 
  \begin{equation}
\mathcal A_L^{p,q}\,J_L^{p+1,q}\varpi=J_L^{p+1,q}\left(\mathrm d^L\,\left(J_L^{p+1,q}\varpi\right)\right)=J_L^{p+1,q}\left(\varpi-J_L^{p+2,q}\mathrm d^L\varpi\right)=J_L^{p+1,q}\varpi\,,
\end{equation} 
which reads $\im \mathcal E_L^{p+1,q}\subseteq\mathrm d^L\im \mathcal A_L^{p,q}$. For the reverse inclusion, one directly notices that 
  \begin{equation}
\varpi=\dd^L(J_L^{p+1,q}\varpi)\,, \qquad \mathcal E_L^{p,q}J_L^{p+1,q}\varpi=0\,, 
\end{equation} 
so that $J_L^{p+1,q}\varpi\in \im \mathcal A_L^{p,q}$. Together with \eqref{splEAL}, the above relation \hyperref[Item: left projectors exact part]{(a)}  gives the relation   \hyperref[Item: left projectors decomposition]{(b)} , that is \eqref{Eq: decomposition left exact and left antiexact}.
\end{proof}
 Based on \cref{Thm: left projectors}, we introduce the following terminology.
  \begin{definition}
     Let $M$ be a smooth manifold, and $(p,q)\in \mathbb N\times \mathbb N_0$. Consider a system of left homotopy operators $J_L$, and $\varpi\in \Omega^{p,q}_{\Delta_M}(M\mid M)$. The bi-form $\mathcal E_L^{p,q}\varpi$ is called the $J_L$-\underline{left-exact component} of $\varpi$, and $\mathcal A_L^{p,q}\varpi$ the $J_L$-\underline{left-antiexact component} of $\varpi$. 
\end{definition}
Note that if a bi-form $\varpi$ has zero $J_L$-left-antiexact component, then the $J_L'$-left-antiexact component of $\varpi$ is zero for any other system of left homotopy operators $J_L'$. 

Mimicking the definition of left homotopy operators, we introduce right homotopy operators. Via the properties of the swap map $\dag$, one proves the   \underline{right-homotopy} analogue of the above identities, as follows. 

  \begin{definition} \label{Def: right homotopy operators}
 Let $M$ be a smooth manifold. 
A \underline{system of right homotopy operators} is a family of $\mathbb R$-linear operators 
           \begin{equation}
        J_R^{p,q}\colon \Omega^{p,q}_{\Delta_M}(M\mid M) \longrightarrow \Omega^{p,q-1}_{\Delta_M}(M\mid M)\,, \qquad (p,q)\in \mathbb N_0\times \mathbb N_0\,,
    \end{equation} 
that satisfy the following conditions: 
  \begin{itemize}
  \item it is $J_R^{p,q-1}\circ J_R^{p,q}=0$ ; 
  \item  for any $\varpi\in \Omega^{p,q}_{\Delta_M}(M\mid M)$ there is an open neighbourhood $\mathscr U$ of $\Delta_M$ in $M\times M$, contained in the domain of $\varpi$, such that the relations
    \begin{equation}\label{Eq: homotopy formula bi-forms R}
    \begin{cases}
\begin{aligned}
\mathrm d^R\,J_R^{p,q}\,\varpi+J_R^{p,q+1}\,\mathrm d^R \varpi
&= \varpi\,,
&& (p,q)\in \mathbb N_0\times \mathbb N\,, \\
J_R^{p,1}\,\mathrm d^R \varpi
&= \varpi - \varpi_{\restriction \Delta_M}\boxtimes 1\,,
&& \textup{otherwise}\,
\end{aligned}
\end{cases}
\end{equation}
hold on $\mathscr U$.
\end{itemize}
\end{definition}
\begin{theorem}\label{Thm: right projectors}    
Let $M$ be a smooth manifold, and consider a system of right homotopy operators $J_R$. For any $(p,q)\in\mathbb{N}_0\times\mathbb{N}_0$, the operators $\mathcal E_R^{p,q}=\mathrm d^RJ_R^{p,q}$ and $\mathcal A_R^{p,q}=J_R^{p,q+1}\mathrm d^R$ are $\mathbb R$-linear projectors in $\Omega_{\Delta_M}^{p,q}(M\mid M)$. Given such operators, the following claims are true.
\begin{enumerate}[(a)]
        \item For $q\geq1$, it is 
  \begin{align}
\label{Eq: splitting e/a R}
&\E_R^{p,q}\varpi+\A_R^{p,q}\varpi=\varpi\,,\\
&\im\A_R^{p,q}=\ker\E_R^{p,q}\,,\\ &\im\E_R^{p,q}=\ker\A_R^{p,q}\\
&\im\A_R^{p,q}\oplus\ker\A_R^{p,q}=\Omega^{p,q}_{\Delta_M}(M\mid M)=\im\E_R^{p,q}\oplus\ker\E_R^{p,q}\,,
\end{align} 
\item For each $(p,q)\in \mathbb N_0\times \mathbb N_0$, one has:
  \begin{align}
&\dd^R\im\A_R^{p,q}=\im\E_R^{p,q+1} \label{nonso}\\
\label{Eq: decomposition right exact and right antiexact}
&\Omega^{p,q+1}_{\Delta_M}(M\mid M)=\mathrm d^R\im \mathcal A_R^{p,q}\oplus \im\mathcal A_R^{p,q+1}\,.
\end{align} 
\end{enumerate}
\end{theorem}
   \begin{definition}
      Let $M$ be a smooth manifold, and $(p,q)\in \mathbb N_0\times \mathbb N$. Consider a system of right homotopy operators $J_R$, and $\varpi\in \Omega^{p,q}_{\Delta_M}(M\mid M)$. The bi-form $\mathcal E_R^{p,q}\varpi$ is called the $J_R$-\underline{right-exact component} of $\varpi$, and $\mathcal A_R^{p,q}\varpi$ the $J_R$-\underline{right-antiexact component} of $\varpi$.
\end{definition}
Note that if a bi-form $\varpi$ has zero $J_R$-left-antiexact component, then the $J_R'$-left-antiexact component of $\varpi$ is zero for any other system of left homotopy operators $J_R'$.

  \subsection{Explicit construction of left and right homotopy operators}
The path that leads to a definition of left and right homotopy operators for the differential bimodule of exterior forms starts by recalling the standard local definition on a star-shaped open subset $\widehat U\subseteq \mathbb R^d$ around a point $q_0=(q^1_0,\dots,q^d_0)$. For each $t\in[0,1]$, consider the maps $H_t\colon \widehat U\to \widehat U$ given by 
$H_t(q)=q_0+t\,(q-q_0).$ 
Note that, since $\widehat U$ is star-shaped with respect to $q_0$, these maps are well-defined and, for $t>0$, they are diffeomorphisms onto their images. Define the radial vector field $D\in\mathfrak X(\widehat U)$ by
\begin{equation}\label{Eq: local difference vector field}
D=(q^i-q^i_0)\,\partial_{q^i}\,.
\end{equation}
 
For every $k\in\mathbb N_0$, define the operator $J\colon \Omega^k(\widehat U)\to\Omega^{k-1}(\widehat U)$ by setting, for every $\omega\in\Omega^k(\widehat U)$,
\begin{equation}\label{Eq: standard homotopy operator J}
J\omega = \mathbf i_D\int_0^1 \frac{1}{t}\, H_t^\ast \omega\, \mathrm dt\,.
\end{equation}
This is proven to satisfy (see, for example, \cite[Theorem 5.3.1]{Edelen-1985-AEC}) the relation $J\circ J=0$, as well as
\begin{equation}
\mathrm dJ\omega+J\mathrm d\omega=\begin{cases}
    \omega-H_0^\ast \omega & k=0\\
       \omega & k>0
\end{cases}
\end{equation}
for all $\omega\in\Omega^k(\widehat U)$.

Notice that the radial vector field can equivalently be written as
\begin{equation}
    D_q=\frac{\mathrm d}{\mathrm dr}\Big|_{r=0}
    \big(q_0+e^r(q-q_0)\big)\,.
\end{equation}
This relation expresses  $D$ in terms of the geodesic interpolation associated to the Euclidean connection on $\widehat U$, so one is led to look  for analogous homotopy operators on forms defined on a neighbourhood of $\Delta_M$ in $M\times M$ by using the geodesic interpolation corresponding to an affine connection. For this generalisation to be well-defined, homotopy operators have to defined on domains on which existence and uniqueness of the geodesic segment joining each pair of points must be suitably satisfied. This is the point we aim to  address in the following lines.

 \subsubsection{Strongly convex neighbourhoods}

 Let $M$ be a smooth manifold endowed with an affine connection $\nabla$. Denote by $\exp^\nabla$ the exponential map corresponding to $\nabla$, and by $\mathscr V_0\subseteq \T M$ its domain. Recall that $(m,v)\in\mathscr V_0$ if and only if the unique $\nabla$-geodesic $\gamma_{(m,v)}$ with initial conditions $\gamma_{(m,v)}(0)=m$ and $\dot\gamma_{(m,v)}(0)=v$ is defined on an open neighbourhood of $[0,1]\subset\mathbb R$, in which case one has 
\begin{equation}
    \exp^\nabla(m,v)=\gamma_{(m,v)}(1)\,.
\end{equation}

 Let $c$ be a point in $M$, and  $U$  an open neighbourhood of $c$ in $M$. We say that $U$ is $\nabla$-\emph{normal} at $c$ if there exists a star-shaped open subset $V\subseteq \T_cM$ containing the zero tangent vector $\mathbf 0_c\in \T_cM$ such that:
\begin{equation}
    U=\exp^\nabla_c(V)\,,
\end{equation}
where $\exp^\nabla_c$ denotes the restriction of $\exp^\nabla$ to 
$\mathscr V_c=\mathscr V_0\cap(\{c\}\times \T_cM)$. 
Equivalently, $U$ is $\nabla$-normal at $c$ if and only if for every 
$n\in U$ there exists a unique $\nabla$-geodesic segment within $U$, denoted 
$\gamma_{n\leftarrow c}^U$, joining $c$ to $n$, that is
\begin{equation}
\gamma_{n\leftarrow c}^U(0)=c,
\qquad
\gamma_{n\leftarrow c}^U(1)=n.
\end{equation}

We say that an open subset $C\subseteq M$ is $\nabla$-\emph{convex} if it is 
$\nabla$-normal at each of its points. In particular, $C$ is $\nabla$-convex if and only if for every $m,n\in C$ there exists a unique $\nabla$-geodesic segment within $C$ joining $m$ with $n$.

The intersection of two $\nabla$-convex open subsets $C_1,C_2\subset M$ is not $\nabla$-convex in general. The intersection $C_1\cap C_2$ may even be disconnected (see \cite{ONeill-1983}): but if  $C_1$ and $C_2$ are both contained in a common $\nabla$-convex open subset $C$, then $C_1\cap C_2$ is indeed 
$\nabla$-convex. To see this, let $m,n\in C_1\cap C_2$. For each $i\in \{1,2\}$, let
  $\gamma_{n\leftarrow m}^{C_i}\colon [0,1] \to C_i$ denote the unique $\nabla$-geodesic segment within $C_i$ 
joining $m$ to $n$. Since $C_i \subseteq C$, each $\gamma_{n\leftarrow m}^{C_i}$ 
is also a $\nabla$-geodesic segment in $C$ joining $m$ to $n$. By the $\nabla$-convexity of $C$, such a segment is unique, hence
\begin{equation}
\gamma_{n\leftarrow m}^{C_1} = \gamma_{n\leftarrow m}^{C_2}\,:
\end{equation}
this common segment  is  contained in $C_1\cap C_2$, therefore proving that $C_1\cap C_2$ is $\nabla$-convex.

\begin{definition}
Let $M$ be a smooth manifold equipped with an affine connection $\nabla$.   
A \underline{strongly $\nabla$-convex covering} of $M$ is an open covering $\mathcal C$ of $M$ such that:
\begin{enumerate}[(1)]
    \item each $C\in \mathcal C$ is $\nabla$-convex;
    \item for any $C,C'\in \mathcal C$, either $C$ and $C'$ are disjoint, or $C\cap C'$ is $\nabla$-convex.
\end{enumerate}
\end{definition}

Given a strongly $\nabla$-convex covering $\mathcal C$, we define:
\begin{equation}
    \mathscr U_{\mathcal C}=\bigcup_{C\in\mathcal C}(C\times C)\,,
\end{equation}
which is an open neighbourhood of the diagonal submanifold $\Delta_M$ in $M\times M$.  
Notice that the open $\mathscr{U}_\mathcal{C}$ is stable under the action of the swap map $\dag$ (see \eqref{Eq: swap map}),  that is one has  $\dag(\mathscr U_{\mathcal C})=\mathscr U_{\mathcal C}$.

\begin{definition}
Let $M$ be a smooth manifold equipped with an affine connection $\nabla$, and let $\mathscr U$ be an open neighbourhood of $\Delta_M$ in $M\times M$.  
We say that $\mathscr U$ is a \underline{strongly $\nabla$-convex neighbourhood} if and only if there exists a strongly $\nabla$-convex covering $\mathcal C$ such that $\mathscr U=\mathscr U_{\mathcal C}$.
\end{definition}

It turns out that $\nabla$-convex coverings are cofinal among open coverings, and that strongly $\nabla$-convex neighbourhoods form a fundamental system of neighbourhoods of the diagonal in $M\times M$. We describe such results as claims of the following theorem.

\begin{theorem}\label{Thm: moretti}
Let $M$ be a smooth manifold equipped with an affine connection $\nabla$. Every open covering $\mathcal A$ of $M$ admits a refinement which is a strongly $\nabla$-convex covering, namely, there exists a strongly $\nabla$-convex covering $\mathcal C$ of $M$ such that for every $C\in\mathcal C$ there exists $A\in\mathcal A$ with $C\subseteq A$. In particular, strongly $\nabla$-convex neighbourhoods exist and form a fundamental system of neighbourhoods of $\Delta_M$ in $M\times M$. 
\end{theorem}

\begin{proof}
  The refinement property is proved in \cite[Proposition~9]{Moretti-2021}. Applying this result to the open cover $\mathcal A=\{M\}$ given by the whole manifold proves the existence of strongly $\nabla$-convex neighbourhoods.  If $\mathscr U$ is any open neighbourhood of $\Delta_M$ in $M\times M$, define the open cover $\mathcal A$ given by the open subsets $U$ of $M$ such that $U\times U\subseteq \mathscr U$. If $\mathcal C$ is a strongly $\nabla$-convex refinement of $\mathcal A$, then $\mathscr U_{\mathcal C}\subseteq \mathscr U$.
\end{proof}

Strongly $\nabla$-convex neighbourhoods allow for a well-defined geodesic interpolation map. Indeed, if $\mathscr U=\mathscr U_{\mathcal C}$ is a strongly $\nabla$-convex neighbourhood and $(m,n)\in\mathscr U$, there exists $C\in\mathcal C$ such that $m,n\in C$, and hence a unique $\nabla$-geodesic segment in $C$ joining $m$ and $n$. We denote it by
\begin{equation}
    \Gamma(m,n,t)=\gamma_{n\leftarrow m}(t)\,,
    \qquad t\in[0,1]\,.
\end{equation}
The definition is independent of the element $C$ of the covering: if $C'$ also contains $m$ and $n$, then $C\cap C'$ is $\nabla$-convex, and the two geodesic segments coincide with the unique segment contained in $C\cap C'$. Thus $\Gamma$ is well-defined on $\mathscr U\times[0,1]$ and is smooth by \cite[Theorem~10(a)]{Moretti-2021}. Moreover,   different strongly $\nabla$-convex coverings determine the same germ\footnote{If $X$ and $Y$ are topological spaces, $S\subseteq X$, and $f,h\colon X\to Y$, then $f$ and $h$ have the same germ along $S$ if there exists an open neighbourhood $U$ of $S$ in $X$ such that $f_{\restriction U}=h_{\restriction U}$.} of $\Gamma$ along $\Delta_M$.

The geodesic interpolation map allows one to introduce a local vector field on $M\times M$ that plays the role of the radial vector field used above for star-shaped domains.

  \begin{definition} \label{Def: left diff vector field}
Let $M$ be a smooth manifold equipped with an affine connection $\nabla$, and let $\mathscr U\subseteq M\times M$ be a strongly $\nabla$-convex neighbourhood.  
The \underline{leftward difference vector field} is the local vector field $\mathbb D$ on $M \times M$ defined for each $(m,n) \in \mathscr U$ by
\begin{equation}\label{Eq: leftward difference vector field}
    \mathbb D_{(m,n)} = \frac{\mathrm d}{\mathrm dr}\Big|_{r=0} \big(\Gamma(n,m,e^r),\, n\big)\,,
\end{equation}
where $\Gamma$ denotes the $\nabla$-geodesic interpolation map associated to $\mathscr U$. 
\end{definition}

Since  the germ of $\Gamma$ is uniquely determined by $\nabla$, the germ of $\mathbb D$ is also uniquely determined by $\nabla$.  From the affine reparametrisation property of $\nabla$-geodesic segments one has that, for $t_0,t_1\in[0,1]$,
\begin{equation}\label{Eq: affine reparametrisation property}
    \Gamma(n,\Gamma(n,m,t_0),t_1) = \Gamma(n,m,t_0\,t_1)\,.
\end{equation}
This reparametrisation property allows to write the local flow of $\mathbb D$ as follows
\begin{equation}
\phi^{\mathbb D}(m,n,r)= \phi^\mathbb D_r(m,n)= \big(\Gamma(n,m,e^r),n\big)\,.
\end{equation}
 The construction admits a right counterpart.
  \begin{definition} \label{Def: right difference vector field}
Let $M$ be a smooth manifold equipped with an affine connection $\nabla$, and let $\mathscr U\subseteq M\times M$ be a strongly $\nabla$-convex neighbourhood.  
The \underline{rightward difference vector field} is the local vector field $\mathbb D^\dag$ on $M\times M$ obtained from $\mathbb D$ via the pushforward of the flip map $\dag$. Its local flow is given by
  \begin{equation}
\phi^{\mathbb D^\dag}(m,n,r)=\phi^{\mathbb D^\dag}_r(m,n)=\big(m,\Gamma(m,n,e^r)\big)\,.
\end{equation} 
\end{definition}
\begin{example}\label{Eq: difference vector fields in the Euclidean setting}
In the case  $M=\mathbb R^d$ with the standard Euclidean connection $\nabla$ one easily sees that  $\mathscr U=M\times M$ is a strongly $\nabla$-convex neighbourhood, and for each $(x,y)\in \mathscr U$ and $t\in [0,1]$ one has
    \begin{equation}
        \Gamma(x,y,t)=x+t\,(y-x)\,
    \end{equation}
and therefore 
    \begin{equation}
        \mathbb D=(x^\ell-y^\ell)\, \partial_{x^\ell}\,, \qquad \mathbb D^\dag=(y^r-x^r)\, \partial_{y^r}\,,
    \end{equation}
which are to be compared to \eqref{Eq: local difference vector field}.
\end{example}

Leftward and rightward difference vector fields on $M\times M$ have interesting relations to the exponential map. To describe such relations, consider a smooth manifold $M$ with an affine connection $\nabla$ on it and let $\psi\colon \mathscr V_0 \to M \times M$
 \begin{equation}
    \psi(m,v) = (m, \exp^\nabla(m,v))\,,
\end{equation}
be defined in terms of the corresponding exponential map, with domain $\mathscr V_0\subseteq \T M$. Such map $\psi$ induces a  diffeomorphism between 
$\mathscr V=\psi^{-1}(\mathscr U)$ and $\mathscr U$, where $\mathscr U$ is a strongly $\nabla$-convex neighbourhood in $M\times M$. Let $\Delta$ be the Euler vector field on $\T M$, i.e. the infinitesimal generator of the linear fiberwise dilation on $\T M$, whose flow is given by
\begin{equation}
    \phi^\Delta(m,v,r) = \phi^\Delta_r(m,v)=(m, e^r v)\,.
\end{equation} 
For any $(m,n) \in \mathscr U$ we can write 
  \begin{equation}
\begin{split}
    \big((\psi_{\restriction \mathscr V})_\ast \Delta\big)_{(m,n)}
    &=  \frac{\mathrm d}{\mathrm dr}\Big|_{r=0} 
       \big(\psi_{\restriction \mathscr V} \circ \phi^\Delta_r \circ 
       (\psi_{\restriction \mathscr V})^{-1}\big)(m,n) \\
    &= \frac{\mathrm d}{\mathrm dr}\Big|_{r=0} 
       \psi_{\restriction \mathscr V}\big(m, e^r (\exp^\nabla_m)^{-1}(n)\big)=
     \frac{\mathrm d}{\mathrm dr}\Big|_{r=0} 
       \big(m, \exp^\nabla(m, e^r\, (\exp^\nabla_m)^{-1}(n))\big) 
    = \frac{\mathrm d}{\mathrm dr}\Big|_{r=0} 
       (m, \Gamma(m,n,e^r))\,,
\end{split}
\end{equation} 
which shows that $\mathbb D^\dag = (\psi_{\restriction \mathscr V})_\ast \Delta$ and therefore also that $\mathbb D=\dag_*(\psi_{\restriction \mathscr V})_\ast \Delta$.

Through the above results,  we have identified both an appropriate analogue of the star-shaped neighbourhood $\widehat U$ and an appropriate analogue of the radial vector field $D$ defined in \eqref{Eq: local difference vector field}. These two ingredients provide the local geometric framework for constructing, for differential forms on $M\times M$ defined only in a neighbourhood of the diagonal, an analogue of the standard homotopy operator $J$ defined in \eqref{Eq: standard homotopy operator J} for the ordinary exterior differential.

\subsubsection{A generalized Poincar\'e lemma}

We are now in the position to introduce suitable homotopy operators on $\Omega^{\bullet}_{\Delta_M}(M\times M)$, mimicking the usual form of the Poincar\'e lemma for the exterior algebra $\Omega^{\bullet}(M)$  (see for example the one given in \cite[II$.7.10$]{KolarMichorSlovak1993}).  
In the following, for each $k\in \mathbb N_0$, we denote by $\Omega^k_{\Delta_M}(M\times M)$ the vector space of differential exterior $k$-forms that are defined in an open neighbourhood of the diagonal submanifold in $M\times M$. We begin with the following definition.
\begin{definition}
\label{defJl}
Let $M$ be a smooth manifold endowed with an affine connection $\nabla$. Let $\mathscr U$ be a strongly $\nabla$-convex neighbourhood and let $k\geq1$. With respect to the leftward difference vector field $\mathbb D$, define the $\mathbb R$-linear map
\begin{equation}
    J^\nabla\colon \Omega^k_{\Delta_M}(M\times M)
    \longrightarrow 
    \Omega^{k-1}_{\Delta_M}(M\times M)\,,
\end{equation} 
whose action on $\omega\in\Omega^k_{\Delta_M}(M\times M)$ is given by
 \begin{equation}\label{Eq: Jnabla}
       J^\nabla\omega=\mathbf i_{\mathbb D}\int_0^1 \frac{1}{t}\,H_t^\ast \omega\,\mathrm dt\,
   \end{equation}
  where, for each $t\in (0,1]$, $H_t$ is the diffeomorphism determined by the flow of $\mathbb D$ at fixed time $\log t$, that is:
      \begin{equation}\label{Eq: H and flow}
        H_t(m,n)=\phi_{\log t}^{\mathbb D}(m,n)=\big(\Gamma(n,m,t),n\big).
    \end{equation} 
\end{definition}    
Notice that $H_t(\mathscr U)\subseteq \mathscr U$ because for each $(m,n)\in \mathscr U$ there is a $\nabla$-convex open subset $C$ of $M$ such that $(m,n)\in C\times C\subseteq \mathscr U$, and so $\Gamma(n,m,t)\in C$ for any $t\in [0,1]$. 
 \begin{theorem}\label{Thm: generalised Poincar lemma}
   Let $M$ be a smooth manifold endowed with an affine connection $\nabla$. The family of $\mathbb R$-linear operators $J^\nabla$ defined above satisfies the following conditions.
\begin{enumerate}[(a)]
\item It is $J^\nabla (J^\nabla\omega)=0$ for any $\omega\in\Omega^k_{\Delta_M}(M\times M)$;
\item for every $\omega\in \Omega^k_{\Delta_M}(M\times M)$ 
there exists an open neighbourhood $\mathscr U$ of $\Delta_M$ in $M\times M$, 
contained in the domain of $\omega$, such that
\begin{equation}\label{Eq: homotopy formula forms L}
    J^\nabla \mathrm d\omega+\mathrm d J^\nabla \omega=\omega-(\iota\circ \pi_R)^\ast \omega
\end{equation}
holds on $\mathscr U$.
\end{enumerate}
\end{theorem}

\begin{proof}
 We start by noticing that, for any  admissible value of the parameter $r$, the vector field $\mathbb D$ is $\phi^{\mathbb D}_r$-related to itself, with $\phi_r^{\mathbb D}$ denoting the flow of $\mathbb D$, as above. It follows that $\mathbb D$ is also $H_t$-related to itself for every $t \in (0,1)$ (cf. \eqref{Eq: H and flow}). Therefore, $J^\nabla \omega$ admits the alternative representation
\begin{equation}\label{Eq: alternative expression JNabla}
    J^\nabla \omega = \int_0^1 \frac{1}{t}\, H_t^\ast \mathbf i_{\mathbb D} \omega \,\mathrm dt\,.
\end{equation}  
With respect to the claims we aim to prove, we can therefore write:
    \begin{enumerate}[(a)]
    \item If $k\leq 1$, the claim is trivially satisfied. If $k\geq2$, one has (see the 
    previous 
    \eqref{Eq: alternative expression JNabla})
  $$    
J^\nabla(J^\nabla\omega)=\int_0^1\dd t\int_0^1\,\frac{1}{st}H^*_t(H^*_s\mathbf i_\mathbb D\mathbf i_\mathbb D\omega)\,\dd s =0
$$ 
since $\mathbf i_{\mathbb D}$ is nihilpotent. 

\item If we recall the Cartan identity 
$\mathcal L_{\mathbb D}=\mathrm d\,\mathbf i_{\mathbb D}+\mathbf i_{\mathbb D}\,\mathrm d$ and the fact that exterior derivative $\mathrm d$ commutes with  $H_t^\ast$ for any $t\in (0,1)$, we have 
\begin{equation}\label{Eq: intermediate homotopy formula}
    \mathrm d(J^\nabla \omega)
    +
    J^\nabla (\mathrm d\omega)
    =
    \mathcal L_{\mathbb D}
    \int_0^1 \frac{1}{t}\,H_t^\ast \omega\,\mathrm dt\,.
\end{equation}
     From the chain rule and the Lie derivative theorem (\cite[Theorem 6.4.1]{A-A-2015}), we compute
     \begin{equation} 
   \begin{split}
        \frac{\mathrm d}{\mathrm dr}\bigg|_{r=t}H_t^\ast \omega&=\frac{\mathrm d}{\mathrm dr}\bigg|_{r=t}\phi_{\log r}^\ast \omega \\
        &=\frac{1}{t}\,\frac{\mathrm d}{\mathrm ds}\bigg|_{s=\log t}\phi_s^\ast \omega =
        \frac{1}{t} \,\phi_{\log t}^\ast \mathcal L_{\mathbb D}\omega =
        \frac{1}{t}\,H_t^\ast \mathcal L_{\mathbb D}\omega\,.
\end{split}
    \end{equation} 
    From the fundamental theorem of calculus, if we integrate  over an interval $[t_0, t_1] \subseteq (0,1)$, we have 
    \begin{equation}
        \int_{t_0}^{t_1} \frac{1}{t} H_t^\ast \mathcal L_{\mathbb D} \omega \, \mathrm dt = H_{t_1}^\ast \omega - H_{t_0}^\ast \omega\,:
    \end{equation}
for  $t_0 \to 0^+$ and $t_1 \to 1^-$, the continuity of $\omega$ and the properties of $H_t$ yield
    \begin{equation}\label{Eq: almost cartan identity}
        \int_0^1 \frac{1}{t} H_t^\ast \mathcal L_{\mathbb D} \omega \, \mathrm dt = \omega - (\iota \circ \pi_R)^\ast \omega\,.
    \end{equation}
Notice that we have considered  that $H_1=\id$ on $\mathscr U$, and
  $$
\lim_{t_0\to 0^+}H_{t_0}(m,n)=\lim_{t_0\to 0^+}\left(\Gamma(n,m ,e^{\log t_0}),n\right)=(n,n)
$$ 
The final step is to pass the Lie derivative under the integral sign. Since $\mathcal L_{\mathbb D}$ commutes with each pull-back $H_t^\ast$, we deduce the homotopy formula
      \begin{equation}
        \mathrm dJ^\nabla \omega+J^\nabla \mathrm d \omega=\mathcal L_{\mathbb D}\int_0^1 \frac{1}{t}H_t^\ast \omega\,\mathrm dt
        =\int_0^1 \frac{1}{t}H_t^\ast \mathcal L_{\mathbb D}\omega\,\mathrm dt=\omega - (\iota \circ \pi_R)^\ast \omega\,.
    \end{equation} 
    \end{enumerate}
\end{proof}

\begin{example}
\label{Example: Jnabla}
  Let $M=\mathbb R^d$, and let $\nabla$ be the standard Euclidean connection on $\mathbb R^d$, as in \cref{Eq: difference vector fields in the Euclidean setting}, so that $\mathscr U=M\times M$ is a strongly $\nabla$-convex neighbourhood of $\Delta_M$ in $M\times M$. Let
      \begin{equation}\label{Eq: omega example}
        \omega=\mathrm dx^i\wedge \mathrm dx^j\wedge \mathrm dy^k\in\Omega^3_{\Delta_M}(M\times M),
    \end{equation} 
    where $i,j,k\in \{1,\dots,d\}$ are fixed indices with $i\ne j$, and $(x,y)$ are the Cartesian global coordinates on $M\times M$. Since for each $(m,n)\in \mathscr U$ and $t\in [0,1]$,
    \begin{equation}
        H_t(x,y)=(y+t\,(x-y),y)\,,
    \end{equation}
    we have 
    \begin{equation}
        H_t^\ast \mathrm dx^i
        = t\,\mathrm dx^i+(1-t)\,\mathrm dy^i,
        \qquad 
        H_t^\ast \mathrm dy^k=\mathrm dy^k\,
    \end{equation}
and hence
    \begin{align}
       H_t^\ast\left(\mathrm dx^i\wedge \mathrm dx^j\wedge \mathrm dy^k\right)
       &= \left(t\,\mathrm dx^i+(1-t)\,\mathrm dy^i\right)
          \wedge
          \left(t\,\mathrm dx^j+(1-t)\,\mathrm dy^j\right)
          \wedge \mathrm dy^k \notag\\
       &\quad = \left(t^2\,\mathrm dx^i\wedge \mathrm dx^j
          +t(1-t)\,\mathrm dx^i\wedge \mathrm dy^j
          +t(1-t)\,\mathrm dy^i\wedge \mathrm dx^j
          +(1-t)^2\,\mathrm dy^i\wedge \mathrm dy^j\right)
          \wedge \mathrm dy^k\,.
    \end{align}
The contraction of this term along  the leftward difference vector field $\mathbb D=(x^\ell-y^\ell)\,\partial_{x^\ell}$, which is a purely left vector field on $M\times M$, gives
    \begin{equation}\label{Eq: JNabla example}
       J^\nabla \omega
       =\frac{1}{2}(x^i-y^i)\,\mathrm dx^j\wedge \mathrm dy^k
        -\frac{1}{2}(x^j-y^j)\,\mathrm dx^i\wedge \mathrm dy^k
        +\frac{1}{2}(x^i-y^i)\,\mathrm dy^j\wedge \mathrm dy^k
        -\frac{1}{2}(x^j-y^j)\,\mathrm dy^i\wedge \mathrm dy^k\,.
    \end{equation}
\end{example}

The interest of the above   example is not only that it shows the action of the homotopy operators $J^\nabla$ for the Euclidean connection on $M=\mathbb R^d$: it shows indeed that $J^\nabla$ does not preserve the decomposition of $\Omega^3_{\Delta_M}(M\times M)$ into the direct sum of subspaces $\im\Phi^{p,q}$ with $p+q=3$. For instance, in the notations of \cref{Example: Jnabla}, 
the previous \eqref{Eq: JNabla example} allows to write
\begin{equation}
\im J^\nabla_{\restriction \im\pi_{2,1}}
    \subseteq \im\pi_{1,1}\oplus \im\pi_{0,2}.
\end{equation}
The following result clarifies that such behaviour is far from being an accident, related to the specific example.

\begin{lemma}\label{Lemma: bidegrees}
    Let $M$ be a smooth manifold equipped with an affine connection $\nabla$, and $(p,q)\in \mathbb N_0\times \mathbb N$. Then for each $a\in \{0,\dots,q-1\}$:
    \begin{equation}\label{Eq: bidegree}
        \pi_{p+a,q-1-a}\circ J^\nabla\circ \pi_{p,q}=0
    \end{equation}
\end{lemma}

\begin{proof}
It suffices to test the claim on exterior forms which can be written as  $\omega=F\,\pi_L^\ast \alpha\wedge\pi_R^\ast \beta$, with $F\in \Omega^{0,0}_{\Delta_M}(M\mid M)$, $\alpha\in \Omega^p(M)$ and $\beta\in \Omega^q(M)$. Using the  distributive law of the pullback over the wedge product and the relation $(\pi_R)_{\restriction\mathscr U}\circ H_t=(\pi_R)_{\restriction\mathscr U}$ for each $t\in (0,1)$, one has
\begin{equation}
    H_t^\ast \omega=\left(H_t^\ast \pi_L^\ast \alpha\right)\wedge\pi_R^\ast \beta\,.
\end{equation}
Therefore, for each $a\in \{0,\dots,q-1\}$, it is 
\begin{equation}
    \pi_{p+1+a,q-1-a}\int_0^1 \frac{1}{t}\,H_t^\ast \omega\,\mathrm dt=0\,.
\end{equation}
Finally, the identity
    $\mathbf i_{\mathbb D}\circ \pi_{p+1+a,q-1-a}=\pi_{p+a,q-1-a}\circ \mathbf i_{\mathbb D}$
yields 
  \begin{equation}
    \pi_{p+a,q-1-a}J^\nabla (\pi_{p,q}\omega)=\pi_{p+a,q-1-a}\mathbf i_{\mathbb D}\int_0^1 \frac{1}{t}\, H_t^\ast \omega \,\mathrm dt=
    \mathbf i_{\mathbb D}\left(\pi_{p+1+a,q-1-a}\int_0^1 \frac{1}{t}H_t^\ast \omega\,\mathrm dt\right)=0\,.
\end{equation} 
\end{proof}

\subsubsection{Homotopy operators on bi-forms} 

The previous lemma suggests that left homotopy operators on bi-forms, which means a system of operators which enjoys the properties  in the  \cref{Def: left homotopy operators}, can be introduced upon projecting the action of $J^\nabla$ onto the  $\pi_{p,q}$-component (cf.~\eqref{Eq: pq projector}). This is what we indeed set.

\begin{definition}
\label{def:JnL}
  Let $M$ be a smooth manifold endowed with an affine connection $\nabla$. We define implicitly a system of operators
\begin{equation*}
J^\nabla_L\colon \Omega^{p,q}_{\Delta_M}(M\mid M)
    \longrightarrow \Omega^{p-1,q}_{\Delta_M}(M\mid M)\,,
    \qquad (p,q)\in\mathbb N\times\mathbb N_0\,,
\end{equation*}
by requiring that, for every $\varpi\in\Omega^{p,q}_{\Delta_M}(M\mid M)$,
\begin{equation}\label{Eq: definition JLNabla}
  \Phi^{p-1,q}\left(J_L^\nabla \varpi\right)=\pi_{p-1,q}J^\nabla \Phi^{p,q}(\varpi)\,.
\end{equation}
Here $\Phi^{p,q}(\varpi)$ denotes the local exterior $(p+q)$-form on $M\times M$ corresponding to $\varpi$ via the action of 
the map defined in \eqref{26.04.1}. 
\end{definition}
The following result proves not only that the previous definition provides a system of left homotopy operators, but provides also an explicit representation for its action.
\begin{theorem}\label{Thm: left homotopy op exists}
    Let $M$ be a smooth manifold equipped with an affine connection $\nabla$. One has that:
    \begin{enumerate}[(a)]
    \item The operators $J_L^\nabla$ defined by \eqref{Eq: definition JLNabla} give a system of left homotopy operators on $M$. 
    \item For any $\varpi\in \Omega^{p,q}_{\Delta_M}(M\mid M)$ with $(p,q)\in \mathbb N\times \mathbb N_0$, there exists a strongly $\nabla$-convex neighbourhood $\mathscr U$ contained in the domain of $\varpi$, such that for all $(m,n)\in \mathscr U$ and for families of vector fields $\{X_i\}_{i=2}^p$ and $\{Y_j\}_{j=1}^q$ on $M$, it is 
\begin{multline}
    \left(J^\nabla_L\varpi\right)(X_2,\dots,X_p\mid Y_1,\dots,Y_q)(m,n)
    \\
    =\int_0^1 
    \varpi_{(\gamma_{m\leftarrow n}(t),n)}\!
    \left(
        \dot \gamma_{m\leftarrow n}(t),
        (X_2)_{\gamma_{m\leftarrow n}(t)},\dots, (X_p)_{\gamma_{m\leftarrow n}(t)}
        \mid (Y_1)_n,\dots,(Y_q)_n
    \right)\mathrm dt\,,\label{Eq: JL explicit}
\end{multline}
where we have denoted $\gamma_{m\leftarrow n}(t)=\Gamma(n,m,t)$.
\end{enumerate}
\end{theorem}
\begin{proof}
    Let $\varpi\in \Omega^{p,q}_{\Delta_M}(M\mid M)$, with $(p,q)\in \mathbb N_0\times \mathbb N$, and consider the local differential $(p+q)$-form $\Phi^{p,q}(\varpi)$. Let $\mathscr U$ be a strongly $\nabla$-convex neighbourhood of $\Delta_M$ in $M\times M$ as in \cref{Thm: generalised Poincar lemma}. 
    \begin{itemize}
    \item  We first prove the nilpotency in claim (a).  If $p\leq 1$, then $J_L^\nabla J_L^\nabla\varpi=0$ follows by the definition of $J_L^\nabla$. Assume therefore $p>1$. Iterating the definition of $J_L^{\nabla}$, we have
\begin{equation}\label{Eq: JLJL}
    \Phi^{p-2,q}\left(J_L^{\nabla}(J_L^{\nabla} \varpi)\right) = \pi_{p-2,q} J^{\nabla} \, (\pi_{p-1,q} J^{\nabla} \Phi^{p,q}(\varpi))\,.
\end{equation}
If we recall that the map $\Phi$ is surjective, with the decomposition \eqref{Eq: pq decomposition}, and the lemma (see \eqref{Eq: bidegree})
  \begin{equation}
\begin{split}
\pi_{p-2,q} J^{\nabla} \big(J^{\nabla} \Phi^{p,q}(\varpi)\big)
&= \pi_{p-2,q} J^{\nabla}
           \left(\sum_{a=0}^{p-1} \pi_{p-1-a,q+a} J^{\nabla} \Phi^{p,q}(\varpi)\right)\\
        &=\pi_{p-2,q} J^{\nabla} (\pi_{p-1,q} J^{\nabla} \Phi^{p,q}(\varpi)) + \sum_{a=1}^{p-1} \pi_{p-2,q} J^{\nabla}
           ( \pi_{p-1-a,q+a} J^{\nabla} \Phi^{p,q}(\varpi))\\
        &= \pi_{p-2,q} J^{\nabla} \, \pi_{p-1,q} J^{\nabla} \Phi^{p,q}(\varpi)= \Phi^{p-2,q}\big(J_L^{\nabla} J_L^{\nabla} \varpi\big)\,,
\end{split}
\end{equation} 
where the last equality comes from the definition.
By \cref{Thm: generalised Poincar lemma}, the left-hand side vanishes, and so we conclude 
\begin{equation}
J_L^{\nabla} J_L^{\nabla} \varpi = 0\,.
\end{equation}

    \item We now prove the homotopy identity in claim (a). We begin by considering the term $\Phi^{p,q}(J_L^\nabla \mathrm d^L\varpi)$. Since $\mathrm d^L\varpi$ has bi-degree $(p+1,q)$, we have, recalling the action of the left and right exterior differential on bi-forms (see \eqref{Eq: d_L and d^L})
  $$
\Phi^{p,q}(J_L^\nabla \mathrm d^L\varpi)=\pi_{p,q}J^\nabla (\Phi^{p+1,q}(\mathrm d^L\varpi))=
    \pi_{p,q}J^\nabla (\mathrm d_{\mathfrak L}\Phi^{p,q}(\varpi))\,.
$$ 
From the decomposition of the exterior differential $\dd$ on $M\times M$ as the sum of a left and a right term (see \eqref{Eq: d=dL+dR}), together with  the identity \eqref{Eq: d_R and d^R}, the expression above reads
  $$
\pi_{p,q}J^\nabla (\mathrm d_{\mathfrak L}\Phi^{p,q}(\varpi))=\pi_{p,q}J^\nabla \mathrm d\Phi^{p,q}(\varpi)-\pi_{p,q}J^\nabla \mathrm d_{\mathfrak R}\Phi^{p,q}(\varpi)=
    \pi_{p,q}J^\nabla \mathrm d\Phi^{p,q}(\varpi)-\pi_{p,q}J^\nabla \Phi^{p,q+1}(\mathrm d^R\varpi)\,.
$$ 
The second term on the right-hand side vanishes by \cref{Lemma: bidegrees}, since $\mathrm d^R\varpi$ has bi-degree $(p,q+1)$, so one has 
\begin{equation}
    \Phi^{p,q}(J_L^\nabla \mathrm d^L\varpi)=\pi_{p,q}J^\nabla (\mathrm d\Phi^{p,q}(\varpi))\,.
\end{equation}
If we consider the homotopy formula \eqref{Eq: homotopy formula forms L}, the r.h.s. of the previous expression reads
\begin{align}
\Phi^{p,q}(J_L^\nabla \mathrm d^L\varpi)
    &=\pi_{p,q}J^\nabla \mathrm d\Phi^{p,q}(\varpi)\\
    &=\pi_{p,q}\Phi^{p,q}(\varpi)-\pi_{p,q}(\iota\circ \pi_R)^\ast \Phi^{p,q}(\varpi)
      -\pi_{p,q}\mathrm dJ^\nabla \Phi^{p,q}(\varpi)\,.
\end{align}
Since $(\iota \circ \pi_R)^\ast \Phi^{p,q}(\varpi)\in \im \pi_{0,p+q}$, one has that
\begin{align}
   \pi_{p,q}\!\left(\Phi^{p,q}(\varpi)- (\iota \circ \pi_R)^\ast \Phi^{p,q}(\varpi)\right)
   &=\begin{cases}
       \Phi^{p,q}(\varpi)-(\iota \circ \pi_R)^\ast \Phi^{p,q}(\varpi) = \Phi^{p,q}\!\left(\varpi-(1\boxtimes \varpi_{\restriction\Delta_M})\right)& (p,q)=(0,q),\\[4pt]
       \Phi^{p,q}(\varpi) & \textup{otherwise.}
   \end{cases}
\end{align}
Therefore, one has 
\begin{equation}
    \Phi^{p,q}(J_L^\nabla \mathrm d^L\varpi)+\pi_{p,q}\mathrm d(J^\nabla \Phi^{p,q}(\varpi))=
    \begin{cases}
       \Phi^{p,q}\!\left(\varpi-(1\boxtimes \varpi_{\restriction\Delta_M})_{\restriction\mathscr U}\right) & (p,q)=(0,q),\\[4pt]
       \Phi^{p,q}(\varpi) & \textup{otherwise}.
   \end{cases}
\end{equation}

Furthermore, analogously to what we have already described, we compute 
  \begin{align}
\pi_{p,q}\mathrm d(J^\nabla\Phi^{p,q}(\varpi))
&=\mathrm d_{\mathfrak L}(\pi_{p-1,q}J^\nabla \Phi^{p,q}(\varpi))
      +\mathrm d_{\mathfrak R}(\pi_{p,q-1}J^\nabla \Phi^{p,q}(\varpi))\\
    &=\mathrm d_{\mathfrak L}(\Phi^{p-1,q}(J_L^\nabla\varpi))
      +\mathrm d_{\mathfrak R}(\pi_{p,q-1}J^\nabla\Phi^{p,q}(\varpi))\\
    &=\Phi^{p,q}(\mathrm d^LJ_L^\nabla\varpi)
      +\mathrm d_{\mathfrak R}(\pi_{p,q-1}J^\nabla \Phi^{p,q}(\varpi))\,.
\end{align} 
The second term on the right-hand side vanishes by \cref{Lemma: bidegrees}. Recollecting the various terms, we have
\begin{equation}
    \Phi^{p,q}(J_L^\nabla \mathrm d^L\varpi)+\Phi^{p,q}(\mathrm d^LJ_L^\nabla\varpi)=\begin{cases}
       \Phi^{p,q}\!\left(\varpi-(1\boxtimes \varpi_{\restriction\Delta_M})\right) & (p,q)=(0,q),\\[4pt]
       \Phi^{p,q}(\varpi) & \textup{otherwise}.
   \end{cases}
\end{equation}
Since $\Phi$ is a $C^\infty(M\times M)$-module monomorphism, this equality is equivalent to the homotopy identity  \eqref{Eq: homotopy formula bi-forms L}, ending the proof of the claim $(a)$.
\item       
 It remains to prove claim (b).  Formula \eqref{Eq: JL explicit} may be verified on local bi-forms which are written as $\varpi = F\,\alpha \boxtimes \beta$, where 
$F\in \Omega^{0,0}_{\Delta_M}$, $\alpha\in \Omega^p(M)$, and $\beta\in \Omega^q(M)$. 
The relation \eqref{Eq: pairing with forms} together with the identities
\begin{equation}
\begin{split}
    &(\pi_L\circ H_t)'_{(m,n)} \left(X^L_{(m,n)}\right)= X_{\gamma_{m\leftarrow n}(t)}\,,
    \\
&    (\pi_R\circ H_t)'_{(m,n)} \left(Y^R_{(m,n)}\right)=Y_n\,,
\end{split}
\end{equation}
gives the claim by a direct computation, whose details we omit. 
\end{itemize}
\end{proof}

\begin{example}\label{Example: JL}
Let $M=\mathbb R^d$, and $\nabla$ be the standard Euclidean affine connection as in \cref{Eq: difference vector fields in the Euclidean setting} and \cref{Example: Jnabla}. Choose $\mathscr U=M\times M$ as a strongly $\nabla$-convex neighbourhood of $\Delta_M$ in $M\times M$, and consider the $(2,1)$-bi-form
    \begin{equation}\label{Eq: bi-form example}
        \varpi=\left(\mathrm dq^i\wedge \mathrm dq^j\right)\boxtimes \mathrm dq^k\,.
    \end{equation}
    Here $q$ are the Cartesian coordinates on $M$, and $x=\pi_L^\ast q$, and $y=\pi_R^\ast q$.
    Let us calculate $J_L^\nabla\varpi$ using the definition \eqref{Eq: definition JLNabla}.

    Since $\Phi^{2,1}(\varpi)$ is the $3$-form $\omega$ in \eqref{Eq: omega example}), using the explicit expression 
    given in \eqref{Eq: JNabla example}, we obtain
    \begin{align}
        J_L^\nabla \varpi&=\frac{1}{2}(x^i-y^i)\,\mathrm dq^j\boxtimes \mathrm dq^k-\frac{1}{2}(x^j-y^j)\,\mathrm dq^i\boxtimes \mathrm dq^k\,.
    \end{align}
\end{example}
Right homotopy operators can be constructed by duality, as follows. Let $J^\nabla_L$ be the left homotopy operator determined by a connection $\nabla$ on a smooth manifold $M$. Set 
    \begin{equation}\label{Eq: JRNabla definition}
        \left(J^\nabla_R\varpi\right)^\dag=J^\nabla_L\varpi^\dag\,.
    \end{equation}
A path which is analogue to the path leading to the proof of the previous results allows to see that  $J_R^\nabla$ is a right homotopy operator 
on $M$, and that the following theorem is valid.
\begin{theorem}
    Let $M$ be a smooth manifold. Any affine connection $\nabla$ on $M$ determines a system of right homotopy operators. Moreover, for any $\varpi\in \Omega^{p,q}_{\Delta_M}(M\mid M)$ with $(p,q)\in \mathbb N_0\times \mathbb N$, there exists a strongly $\nabla$-convex neighbourhood $\mathscr U$ contained in the domain of $\varpi$, such that for all $(m,n)\in \mathscr U$ and for families of vector fields $\{X_i\}_{i=1}^p$ and $\{Y_j\}_{j=2}^{q}$ on $M$, one has 
\begin{multline}  \left(J^\nabla_R\varpi\right)(X_1,\dots,X_p\mid Y_2,\dots,Y_q)(m,n) \\ =\int_0^1  \varpi_{(m,\gamma_{n\leftarrow m}(t))}\!\left((X_1)_m,\dots,(X_p)_m\mid  \dot \gamma_{n\leftarrow m}(t), (Y_2)_{\gamma_{n\leftarrow m}(t)},\dots, (Y_q)_{\gamma_{n\leftarrow m}(t)}\right)\mathrm dt\,.\label{Eq: JR explicit}
\end{multline}
\end{theorem}

Notice that the system of right homotopy operators $J_R^\nabla$ may also be directly defined in analogy with \eqref{Eq: definition JLNabla} by 
$$
\Phi^{p,q-1}\left(J_R^\nabla\varpi\right)=(-1)^{pq}\,\pi_{p,q-1}\prescript{\nabla}{}{J}(\Phi^{p,q}(\varpi))\,
$$
where $\prescript{\nabla}{}{J}$ is constructed as in \cref{defJl},  using the rightward difference vector field $\mathbb D^\dag$ instead of the leftward difference vector field $\mathbb D$, that is 
 \begin{equation}\label{Eq: nablaJ}
       \prescript{\nabla}{}J\omega=\mathbf i_{\mathbb D^\dag}\int_0^1 \frac{1}{t}\,H_t^\ast \omega\,\mathrm dt\,.
   \end{equation}

  \subsection{Statistical homotopy operators}
 Beyond the standard homotopy relations, left and right homotopy operators associated to affine connections exhibit additional properties that are crucial within the setting of  information geometry.  
In particular, while the action of the left and right exterior derivatives  \underline{decrease}  the order of vanishing along the diagonal, the action of statistical left and right homotopy operators   \underline{increase} it.

Recall that, given an immersion between smooth manifolds $i\colon S \hookrightarrow Q$, an exterior form $\omega\in\Omega^p(Q)$ is said to \emph{vanish along} $S$ if and only if it vanishes at every point of $i(S)$ 
(see \cite[p.~287]{Lee-2013}), that is 
\begin{equation}
i^\ast\left(\omega(Z_1,\dots,Z_p)\right)=0\,
\end{equation}
for any set $\{Z_1, \dots, Z_p\}$ of vector fields on $Q$. We notice that such a condition is  stronger than requiring $i^\ast\omega=0$, since such  latter is equivalent to
\begin{equation}
i^\ast\big(\omega(Z_1,\dots,Z_p)\big)=0
\end{equation}
for all vector fields $Z_1,\dots,Z_p$ on $Q$ that are tangent to the range  $i(S)\hookrightarrow Q$.
More generally, for each $k \in \mathbb{N}$, we say that a differential form $\omega$ on $Q$   \emph{vanishes to order $k$ along} $S$ if it vanishes along $S$, and for every vector field $Z$ on $Q$, the Lie derivative $\mathcal{L}_Z \omega$ vanishes to order $(k-1)$ along $S$. This definition can be extended to bi-forms as follows  

  \begin{definition}
 Let $M$ be a smooth manifold, $k\in \mathbb N_0$, $(p,q)\in \mathbb N_0\times \mathbb N_0$, and let $\varpi\in \Omega^{p,q}_{\Delta_M}(M\mid M)$.  
We say that $\varpi$ \underline{vanishes to order $k$ along $\Delta_M$} if and only if $\Phi^{p,q}(\varpi)$ vanishes to order $k$ along $\Delta_M$.
\end{definition}

\begin{remark}
Notice that, if $\omega\in\Omega^k(Q)$ vanishes to order $r$ along $S\hookrightarrow Q$, it follows by induction on $r$ that the  Frölicher–Nijenhuis derivative   $\dd_A\omega\in\Omega^{k+1}(Q)$  (see the relation \eqref{Eq: dA}) with respect to a $(1,1)$-tensor $A$ on $Q$ vanishes to order $(r-1)$ along $S$.

This allows to prove that,  if $\varpi$ vanishes to order $k$ along $\Delta_M$, then its left (i.e. $\dd^L\varpi$) and right (i.e. $\dd^R\varpi$) exterior derivatives vanish to order $(k-1)$ along $\Delta_M$. 
\end{remark}

We can now prove the following result, about the vanishing of terms $J^\nabla\omega$ and ${}^\nabla\! J\omega$ (with respect to left and right homotopy operators defined in \eqref{Eq: Jnabla} and \eqref{Eq: nablaJ}) along the diagonal embedding $\iota\colon \Delta_M\hookrightarrow M\times M$.

\begin{proposition}\label{Prop: further properties JNabla}
Let $M$ be a smooth manifold endowed with an affine connection $\nabla$, and $\omega\in \Omega^k_{\Delta_M}(M\times M)$, with $k\in \mathbb N_0$. The exterior form  $J^\nabla\omega$ vanishes along $\Delta_M$. Moreover, if $\omega$ vanishes to order $r\in \mathbb N_0$ along $\Delta_M$, then both terms $J^\nabla\omega$ and ${}^\nabla\!J\omega$ vanish to order $r+1$ along $\Delta_M$.
\end{proposition}
\begin{proof}
The case $k=0$ is trivial, so we assume $k\geq1$. Notice that the $(k-1)$-form $J^\nabla \omega$ can be expressed as the interior contraction of a (local) $k$-form $\mu$ on $M\times M$ along  the leftward difference vector field $\mathbb D$, namely 
\begin{equation}
    J^\nabla\omega=\mathbf i_{\mathbb D}\mu\,,
    \qquad
    \mu=\int_0^1 \frac{1}{t}\, H_t^\ast \omega\, \mathrm dt\,.
\end{equation}
Since $\mathbb D$ vanishes at any point in $\Delta_M$, it follows that $J^\nabla \omega$ vanishes along $\Delta_M$.

Assume now that $\omega$ vanishes to order $r$ along $\Delta_M$. Since $H_t$ fixes $\Delta_M$ for any   $t\in [0,1]$ , the form $\mu$ also vanishes to order $r$ along $\Delta_M$. To prove that $J^\nabla\omega$ vanishes to order $(r+1)$ along $\Delta_M$, consider a vector field  $\mathbb Z$  on $M\times M$. From  the Cartan identity, we have 
\begin{equation}\label{Eq: Cartan identity forms}
    \mathcal L_{\mathbb Z}\left(J^\nabla \omega\right)=
    \mathbf i_{\mathbb D}(\mathcal L_{\mathbb Z}\mu)+\mathbf i_{[\mathbb Z,\mathbb D]}\mu\,.
\end{equation}
The first term on the right-hand side vanishes to order $r$ along $\Delta_M$ because $\mathbb D$ vanishes along $\Delta_M$ and $\mu$ vanishes to order $r$. The second term also vanishes to order $r$ along $\Delta_M$, since $\mu$ does.

The proof for the term ${}^\nabla\!J\omega$ is analogue.
 \end{proof}

This result clarifies the reasons for the definition we now set.

  \begin{definition} \label{Def: statistical homotopy operators}
Let $M$ be a smooth manifold. 
\begin{itemize}
\item A \underline{system of statistical left homotopy operators} on $M$ is a system of left homotopy operators on $M$ such that for each $\varpi \in \Omega^{\bullet,\bullet}_{\Delta_M}(M \mid M)$, the bi-form $J_L\varpi$ vanishes along $\Delta_M$. 
\item A \underline{system of statistical right homotopy operators} on $M$ is a system of right homotopy operators on $M$ such that for each $\varpi \in \Omega^{\bullet,\bullet}_{\Delta_M}(M \mid M)$, the bi-form $J_R\varpi$ vanishes along $\Delta_M$. 
\end{itemize}
\end{definition}

Moreover, the   \cref{Prop: further properties JNabla}  shows that if $\varpi$ vanishes to order $r\in \mathbb N_0$ along $\Delta_M$, then both $J^\nabla_L\varpi$ and $J^\nabla_R\varpi$ vanish to order $(r+1)$ along $\Delta_M$.
We now prove that a left homotopy system and a right homotopy system obtained from possibly different affine connections are automatically statistical.

\begin{theorem}\label{Thm: stat homotopy op exists}
Let $M$ be a smooth manifold equipped with affine connections $\nabla_1$ and $\nabla_2$. One has: 
\begin{enumerate}[(a)]
\item The operators  $J_L^{\nabla_1}$ give  a system of statistical left homotopy operators and the operators $J_R^{\nabla_2}$ give  a system of statistical right homotopy operators. 
\item For any $\varpi\in \Omega^{p,q}_{\Delta_M}(M\mid M)$, there exists a strongly $\nabla_1$-convex neighbourhood $\mathscr U$ that is also strongly $\nabla_2$-convex contained in the domain of $\varpi$, such that for all $(m,n)\in \mathscr U$ one has 
\begin{multline}\label{Eq: JLJR formula mixed}
    \left(J^{\nabla_2}_RJ_L^{\nabla_1}\varpi\right)(X_2,\dots,X_p\mid Y_2,\dots,Y_q)(m,n)
    \\
    =\iint_{[0,1]^2}
    \varpi_{(a_r(t),b(r))}
    \big(\dot a_r(t), (X_2)_{a_r(t)},\dots,(X_p)_{a_r(t)}\mid
    \dot b(r), (Y_2)_{b(r)},\dots,(Y_q)_{b(r)}\big)\,\mathrm dt\,\mathrm dr\,,
\end{multline}
(with $\{X_i\}_{i=2}^p$ and $\{Y_j\}_{j=2}^q$ in $\mathfrak{X}(M)$) where for any $(t,r)\in [0,1]^2=[0,1]\times[0,1]$ it is 
\begin{equation}
    b(r)=\gamma^{\nabla_2}_{n\leftarrow m}(r)\,, \qquad a_r(t)=\gamma^{\nabla_1}_{m\leftarrow b(r)}(t)\,.
\end{equation}
\item If $\nabla_1=\nabla_2=\nabla$, the previous expression \eqref{Eq: JLJR formula mixed} reduces to
\begin{multline}\label{Eq: JLJR formula}
    \left(J^{\nabla}_RJ_L^{\nabla}\varpi\right)(X_2,\dots,X_p\mid Y_2,\dots,Y_q)(m,n)
    \\
    =\iint_{[0,1]^2}\, r\,
    \varpi_{(\gamma_{m\leftarrow n}(1-r+t\,r),\gamma_{n\leftarrow m}(r))}\!
    \big(\dot\gamma_{m \leftarrow n}(1-r+t\,r), (X_2)_{\gamma_{m\leftarrow n}(1-r+t\,r)},\dots,(X_p)_{\gamma_{m\leftarrow n}(1-r+t\,r)}\mid\\
      \dot \gamma_{n\leftarrow m}(r),
        (Y_2)_{\gamma_{n\leftarrow m}(r)},\dots, (Y_q)_{\gamma_{n\leftarrow m}(r)}\big)\,\mathrm dt\,\mathrm dr\,.
\end{multline}
\end{enumerate}
\end{theorem}

\begin{proof} We prove the three assertions separately.
\begin{enumerate}[(a)]
\item The case $p=0$ is trivial, so we consider $p\geq 1$. Consider a common neighbourhood $\mathscr U$ of $\Delta_M$ in $M\times M$ which is both strongly $\nabla_1$-convex and strongly $\nabla_2$-convex.
By \cref{Prop: further properties JNabla}, if $\Phi^{p,q}(\varpi)$ vanishes to order $k$ along $\Delta_M$, then $J^{\nabla_1}\Phi^{p,q}(\varpi)$ vanishes to order $k+1$ along $\Delta_M$. Hence its $\pi_{p-1,q}$-component also vanishes to order $k+1$ along $\Delta_M$, and therefore $\Phi^{p-1,q}\!\big(J_L^{\nabla_1}\varpi\big)$ vanishes to order $k+1$ along $\Delta_M$ (cf.~\eqref{Eq: definition JLNabla}). By the definition of vanishing bi-forms, it follows that $J_L^{\nabla_1}\varpi$ vanishes to order $k+1$ along $\Delta_M$. The argument for $J_R^{\nabla_2}$ is analogous. 

\item In order to prove \eqref{Eq: JLJR formula mixed}, we consider the expression \eqref{Eq: JR explicit}. We have 
\begin{multline}
    \left(J_R^{\nabla_2} J_L^{\nabla_1} \varpi\right)(X_2,\dots,X_p\mid Y_2,\dots,Y_q)(m,n)
    \\
    =\int_0^1 \left(J_L^{\nabla_1} \varpi\right)_{(m,b(r))}
    \big((X_2)_m,\dots,(X_p)_m\mid  \dot b(r),(Y_2)_{b(r)},\dots,(Y_q)_{b(r)}\big)\,\mathrm dr\,,
\end{multline}
In order to apply the equality \eqref{Eq: JL explicit}, we must determine for each $(m,n)\in \mathscr U$ and $r\in [0,1]$, the $\nabla_1$-geodesic segment joining $b(r)$ to $m$ within $\mathscr U$, together with its associated velocity vector. By strong $\nabla_1$-convexity, there exists a unique such segment, which we denote by
  \begin{equation}
    a_r\colon [0,1]\longrightarrow M\,,
    \qquad
    a_r(0)=b(r)\,,
    \qquad
    a_r(1)=m\,.
\end{equation} 
If we consider the expression  \eqref{Eq: JL explicit} for the pair $(m,b(r))$, we have
\begin{multline}
    \left(J_L^{\nabla_1} \varpi\right)_{(m,b(r))}
    \big((X_2)_m,\dots,(X_p)_m\mid \dot b(r),(Y_2)_{b(r)},\dots,(Y_q)_{b(r)}\big)
    \\
    =
    \int_0^1
    \varpi_{(a_r(t),b(r))}
    \big(
        \dot a_r(t),(X_2)_{a_r(t)},\dots,(X_p)_{a_r(t)}
        \mid
        \dot b(r),(Y_2)_{b(r)},\dots,(Y_q)_{b(r)}
    \big)\,\mathrm dt\,.
\end{multline}
Inserting this expression into the previous equality yields \eqref{Eq: JLJR formula mixed}.

\item Finally, assume $\nabla_1=\nabla_2=\nabla$. In this case, for each fixed $r\in [0,1]$, the $\nabla$-geodesic segment joining $\gamma^{\nabla}_{n\leftarrow m}(r)$ to  $m$ is obtained from $\gamma^{\nabla}_{m\leftarrow n}$ by an affine reparametrization. Therefore, we have 
\begin{equation}
    a_r(t)=\gamma^{\nabla}_{m\leftarrow n}(1-r+t\,r)\,,
\end{equation}
and 
\begin{equation}
    \dot a_r(t)=r\,\dot\gamma^{\nabla}_{m\leftarrow n}(1-r+t\,r)\,.
\end{equation}
Inserting these expressions into \eqref{Eq: JLJR formula mixed} yields \eqref{Eq: JLJR formula}. \hfill\qedhere
 \end{enumerate}
\end{proof}

\section{Contrast bi-forms in Information Geometry}\label{Section: Contrast bi-forms in Information Geometry}
In this section we introduce the notion of a contrast bi-form, and show how it provides a statistical potential for Lauritzen manifolds. We describe  the equivalence between contrast bi-forms and super-contrast functions, showing how classical contrast functions and pre-contrast functions emerge as special cases distinguished by their cohomological properties when restricted  to the diagonal submanifold. More precisely, we show that,  when a contrast bi-form induces a SMAT, the geometric structure is encoded in its left-exact component; when it induces a statistical manifold, the geometry is captured by the right-exact component of its left-exact part.  
\subsection{Contrast bi-forms}
Every local bi-form on $M\times M$ determines, by restriction to the diagonal, a covariant tensor field on $M$. 
Since our aim is to generate pseudo-Riemannian metrics and affine connections, it is natural to ask which bi-degrees may yield a nondegenerate symmetric $2$-covariant tensor under diagonal restriction. The possible types are $(2,0)$, $(1,1)$, and $(0,2)$. However, the diagonal restriction of a local $(2,0)$- or $(0,2)$-bi-form is necessarily skew-symmetric. Consequently, only $(1,1)$-bi-forms can produce symmetric $2$-covariant tensors by restriction to the diagonal. This motivates the following definition.
  \begin{definition}
 Let $M$ be a smooth manifold, and $\varpi\in \Omega^{1,1}_{\Delta_M}(M\mid M)$. Following the definition of diagonal restriction of a bi-form given in \eqref {def:dr}, we define the tensor  
\begin{equation}\label{Eq: induced metric}
g^\varpi=\varpi_{\restriction \Delta_M}
\end{equation}
on $M$ and we  say that $\varpi$ is a \underline{contrast bi-form} if and only if $g^\varpi=\varpi_{\restriction \Delta_M}$ is a pseudo--Riemannian metric on $M$.
\end{definition}

In a local chart $(U,q)$, the induced metric $g^\varpi$ is represented as 
\begin{equation}
   g^{\varpi}=(\iota^\ast \varpi_{i\mid j})\, \mathrm dq^i\otimes \mathrm dq^j\,,
\end{equation}
where $\varpi_{i\mid j}=\varpi(\partial_{q^i}\mid \partial_{q^j})$.  
 Our next result shows that the symmetry of the diagonal restriction of a bi-form is equivalently written in terms of the vanishing along the diagonal of a suitable 2-form on $M\times M$. 
\begin{lemma}\label{Lemma: initial condition}
    Let $M$ be a manifold, and $\varpi\in \Omega^{1,1}_{\Delta_M}(M\mid M)$. The $2$-form $g^\varpi=\varpi_{\restriction\Delta_M}$ is symmetric if and only if
    \begin{equation}
        \iota^\ast \left(\Phi^{1,1}(\varpi)\right)=0\,.
    \end{equation}
\end{lemma}
\begin{proof}
For $X,Y\in\mathfrak{X}(M)$, recalling  \eqref{Eq: (X,X) is iota-related to X} and \eqref{Eq: pairing with forms}, we can write
  \begin{equation}
g^\varpi(X,Y)-g^\varpi(Y,X)=\varpi_{\restriction \Delta_M}(X,Y)-\varpi_{\restriction \Delta_M}(Y,X)=\iota^*\left((\Phi^{1,1}(\varpi)(X^L,Y^R)-\Phi^{1,1}(\varpi)(Y^L,X^R)\right)\,,
\end{equation} 
and, again from \eqref{Eq: pairing with forms}, we can also write 
  \begin{equation}
\iota^*(\Phi^{1,1}(\varpi))(X,Y)=\iota^*\left(\Phi^{1,1}(\varpi)(X^L+X^R,Y^L+Y^R)\right)=\iota^*\left(\Phi^{1,1}(\varpi)(X^L,Y^R)-\Phi^{1,1}(\varpi)(Y^L,X^R)\right)\,.
\end{equation} 
The above lines show that 
$\varpi_{\restriction\Delta_M}$is symmetric if and only if $\iota^\ast\left(\Phi^{1,1}(\varpi)\right)=0$.
\end{proof}
A contrast bi-form $\varpi$ on $M$ induces implicitely an affine connection $\nabla^\varpi$ on $M$ by the position
\begin{equation}\label{Eq: induced connection}
    g^\varpi(\nabla_Z^\varpi X,Y)=\left(\mathcal L_{Z\mid}\varpi\right)_{\restriction\Delta_M}(X,Y)+g^\varpi([Z,X],Y)\,,
\end{equation}
with $X,Y,Z\in\mathfrak{X}(M)$. In local coordinates $(U,q)$, the Christoffel symbols for the above connection $\nabla^\varpi$ are
  \begin{equation}
    {}^{(\nabla^\varpi)}{\Gamma}_{ki}^\ell=(g^\varpi)^{\ell j}\frac{\partial \varpi_{i\mid j}}{\partial x^k} 
\end{equation} 
where $\varpi_{i\mid j}=\varpi\big(\partial_{q^i}\mid \partial_{q^j}\big)$, and $(U\times U,x,y)$ is the square chart associated to $(U,q)$.
\begin{proposition}
Let $M$ be a manifold, and $\varpi\in \Omega^{1,1}_{\Delta_M}(M\mid M)$. Then $\varpi$ is a contrast bi-form if and only if $\varpi^\dag$ is a contrast bi-form. In particular:
\begin{equation}
    \begin{cases}
        g^{(\varpi^\dag)}=g^\varpi\\
        \nabla^{(\varpi^\dag)}=(\nabla^\varpi)^\dag\label{Eq: duality bi-forms affine connections}
    \end{cases}
\end{equation}
\end{proposition}

\begin{proof}
The equality \eqref{Eq: sast iotaast} implies that,  for any pair of vector fields $X,Y$ on $M$, one has 
\begin{equation}\label{Eq: g dual bi-form}
    g^{(\varpi^\dag)}(X,Y)=g^\varpi(Y,X)\,,
\end{equation}
from which it is clear that $\varpi$ is a contrast bi-form if and only if $\varpi^\dag$ is so. Furthermore, let 
 $X,Y,Z$ be vector fields on $M$. From the definition \eqref{Eq: induced metric} one writes
  \begin{equation}
\begin{split}
\mathcal L_Z(g^\varpi(X,Y))=\mathcal L_Z ((\varpi_{\restriction \Delta_M})(X,Y))&=
    (\mathcal L_Z(\varpi_{\restriction \Delta_M}))(X,Y)+(\varpi_{\restriction \Delta_M})([Z,X],Y)+(\varpi_{\restriction \Delta_M})(X,[Z,Y])\\ &=(\mathcal L_Z(\varpi_{\restriction \Delta_M}))(X,Y)+g^\varpi([Z,X],Y)+g^\varpi(X,[Z,Y])
\end{split}
\end{equation} 
From the identity \eqref{Eq: iota ast Lie derivative} for the two first  terms on the r.h.s., one can recast the previous relation as 
  \begin{equation}
\begin{split}
\mathcal L_Z(g^\varpi(X,Y))&=\left(\mathcal L_{Z\mid}\varpi\right)_{\restriction\Delta_M}(X,Y)+\left(\mathcal L_{\mid Z}\varpi\right)_{\restriction\Delta_M}(X, Y)+g^\varpi([Z,X],Y)+g^\varpi(X,[Z,Y])\\
    &\,=g^\varpi(\nabla_Z^\varpi X,Y)+\left(\mathcal L_{\mid Z}\varpi\right)_{\restriction\Delta_M}(X, Y)+g^\varpi(X,[Z,Y])\,, 
\end{split}
\end{equation} 
where the last equality comes from \eqref{Eq: induced connection}. The compatibility 
 condition \eqref{Eq: compatibility condition} therefore entails
\begin{equation}\label{Eq: alternative definition of dual}
    g^{\varpi}\big(X,(\nabla^{\varpi})^\dag_ZY\big)=\left(\mathcal L_{\mid Z}\varpi\right)_{\restriction\Delta_M}(X, Y)+g^\varpi(X,[Z,Y])\,.
\end{equation}
Since the definition of right Lie derivative along $Z$ together with  \eqref{Eq: sast iotaast} imply
$$
    \left(\mathcal L_{\mid Z}\varpi\right)_{\restriction\Delta_M}(X,Y)=\left(\mathcal L_{Z\mid}\left( \varpi^\dag\right)\right)_{\restriction\Delta_M}(Y, X)\,,
$$
we can write 
  \begin{equation}
\begin{split}
    g^{(\varpi^\dag)}\left((\nabla^\varpi)^\dag_ZY,X\right)&=\left(\mathcal L_{Z\mid}\varpi^\dag\right)_{\restriction\Delta_M}(Y,X)+g^{\varpi}(X,[Z,Y])\\
    &=\left(\mathcal L_{Z\mid}\varpi^\dag\right)_{\restriction\Delta_M}(Y,X)+g^{(\varpi^\dag)}([Z,Y],X) =
    g^{(\varpi^\dag)}\left(\nabla^{(\varpi^\dag)}_ZY,X\right)
\end{split}
\end{equation} 
Finally, the claim follows by the non-degeneracy of $g^{(\varpi^\dag)}$.
\end{proof}
\begin{remark}
The proposition above shows that a contrast bi-form $\varpi$ on $M$ generates a Lauritzen manifold  $(M,g^\varpi,\nabla^\varpi)$. Along this path one notices that, if $\varpi\in\Omega^{1,1}_{\Delta_M}(M\mid M)$ is a contrast bi-form on a smooth manifold $M$, then $\dot \varpi\colon \T M\times \T M\to \mathbb R$ is a super-contrast function in the sense of Zhang and Khan. Furthermore, the algorithm à la Zhang and Khan for extracting the metric and connection data from $\dot\varpi$ (cf. \eqref{Eq: metric super-contrast} and \eqref{Eq: connection super-contrast})  coincides with the ones derived from $\varpi$ (cf. \eqref{Eq: induced metric} and \eqref{Eq: induced connection}). Thus, contrast bi-forms provide an alternative statistical potential for Lauritzen manifold equivalent to super-contrast functions.
\end{remark}

Different contrast bi-forms on a smooth manifold may induce the same statistical structures.  
This motivates the following notion.

  \begin{definition}
     Let $\varpi_1,\varpi_2\in \Omega^{1,1}_{\Delta_M}(M\mid M)$ be contrast bi-forms on a smooth manifold $M$.  
    We say that $\varpi_1$ and $\varpi_2$ are \underline{statistically equivalent} if:
    \begin{equation}
        \begin{cases}
            g^{\varpi_1}=g^{\varpi_2}\,,\\
            \nabla^{\varpi_1}=\nabla^{\varpi_2}\,.
        \end{cases}
    \end{equation}
\end{definition}

The following result characterizes statistical equivalence in terms of the behaviour of the bi-form along the diagonal.

\begin{proposition}
\label{th:equi}
    Let $\varpi_1,\varpi_2\in \Omega^{1,1}_{\Delta_M}(M\mid M)$ be contrast bi-forms on a smooth manifold $M$.  
    Then $\varpi_1$ and $\varpi_2$ are statistically equivalent if and only if $\varpi_2-\varpi_1$ vanishes to first order along $\Delta_M$.
\end{proposition}

\begin{proof}
    By the above definition, the contrast bi-forms $\varpi_1$ and $\varpi_2$ are statistically equivalent if and only if for each $Z\in \mathfrak X(M)$, that is 
    \begin{equation}
        \begin{cases}
            (\varpi_2-\varpi_1)_{\restriction\Delta_M}=0\;,\\[4pt]
            \left(\mathcal L_{Z\mid }(\varpi_2- \varpi_1)\right)_{\restriction\Delta_M}=0\;.
        \end{cases}
    \end{equation}
    When these conditions are considered together with the identities \eqref{Eq: iota ast Lie derivative}, one has 
    \begin{equation}
        \left(\mathcal L_{\mid Z}(\varpi_2-\varpi_1)\right)_{\restriction\Delta_M}=0\;.
    \end{equation}
    Therefore, $\varpi_1$ and $\varpi_2$ are statistically equivalent precisely when the 2-form $\Phi^{1,1}(\varpi_2-\varpi_1)$ on $M\times M$  vanishes along $\Delta_M$, together with its Lie derivatives along any left or right lift of a vector field on $M$. This is equivalent to vanishing to first order along $\Delta_M$.
\end{proof}

\begin{remark}
In particular, we notice that the metric and connection determined by a contrast bi-form depend only on its first-order behaviour along the diagonal. Higher-order contributions therefore carry no statistical significance.
\end{remark}

The formalism of bi-forms allows us to analyse the torsion of the induced connections in terms of  the restriction of the left and the right exterior derivatives to the diagonal submanifold.
\begin{theorem}\label{Thm: torsion bi-form}
    Let $\varpi$ be a contrast bi-form on a smooth manifold $M$. One has that:
    \begin{enumerate}[(a)]
    \item the connection  $\nabla^\varpi$ is torsion-free if and only if $(\mathrm d^L\varpi)_{\restriction \Delta_M}=0$;
    \item the connection  $\nabla^{(\varpi^\dag)}$ is torsion-free if and only if $(\mathrm d^R\varpi)_{\restriction \Delta_M}=0$.
    \end{enumerate}
\end{theorem}
\begin{proof}
We analyse the conditions as follows.
\begin{enumerate}[(a)]
\item In order to prove the first condition, consider vector fields $X_1,X_2$ and $Y$ on $M$. From the \cref{def:lred}, one has  
$$ 
\mathrm d^L\varpi(X_1,X_2\mid Y)=(\mathcal L_{X_1\mid}\varpi)(X_2\mid Y)-(\mathcal L_{X_2\mid }\varpi)(X_1\mid Y)+\varpi([X_1,X_2],Y):   
$$
by adding and subtracting the term $\varpi([X_1,X_2]\mid Y)$, we have 
$$
\mathrm d^L\varpi(X_1,X_2\mid Y)=(\mathcal L_{X_1\mid}\varpi)(X_2\mid Y)+\varpi([X_1,X_2],Y)
        -(\mathcal L_{X_2\mid }\varpi)(X_1\mid Y)-\varpi([X_2,X_1],Y)-\varpi([X_1,X_2],Y)\,.
$$
If we recall the definition  \eqref{Eq: induced connection}, and consider the restriction to the diagonal of the above terms, we have 
    \begin{equation}
        \left(\mathrm d^L\varpi\right)_{\restriction\Delta_M}(X_1,X_2,Y)=g^\varpi(\nabla^\varpi_{X_1}X_2-\nabla^{\varpi}_{X_2}X_1-[X_1,X_2],Y)\,.
    \end{equation}
    which is equivalent to  (cf. \eqref{Eq: torsion tensor})
    \begin{equation}\label{Eq: dL and torsion}
      \left(\mathrm d^L\varpi\right)_{\restriction\Delta_M}(X_1,X_2\mid Y)=g^\varpi\left(\Tor^{(\nabla^\varpi)}(X_1,X_2),Y\right).
    \end{equation}
    The non-degeneracy of $g^\varpi$ then yields the first claim.

\item Analogously to the path described in the previous item, for vector fields $X,Y_1,Y_2$  on $M$ we write 
  \begin{equation}
\begin{split}
        \left(\mathrm d^R\varpi\right)_{\restriction\Delta_M}(X, Y_0,Y_1)=
        \left(\left(\mathrm d^L\varpi^\dag\right)^\dag\right)_{\restriction \Delta_M}(X,Y_0,Y_1)&=
\left(\mathrm d^L\varpi^\dag\right)_{\restriction\Delta_M}(Y_0,Y_1, X)\\
        &=g^{(\varpi^\dag)}\Big(\Tor^{\nabla^{(\varpi^\dag)}}(Y_0,Y_1),X\Big)=
g^{(\varpi^\dag)}\Big(\Tor^{(\nabla^\varpi)^\dag}(Y_0,Y_1),X\Big)\,.
    \end{split}
    \end{equation} 
Since  $g^{(\varpi^\dag)}$ is not degenerate, we have can conclude that the claim is valid. 
    \end{enumerate}
\end{proof}

The above \cref{Thm: torsion bi-form} suggests a cohomological characterization of contrast bi-forms inducing SMAT structures or statistical manifolds: bi-forms inducing SMAT structures are characterized by the vanishing of their left exterior derivative along the diagonal submanifold, whereas bi-forms inducing statistical manifolds are characterized by the vanishing of both their left and right exterior derivatives along the diagonal. We now examine these conditions in detail by considering the two cases separately.

\subsubsection{Left-exact contrast bi-forms}\label{subsubs:lec}
Given the commutative bi-complex $(\Omega^{\bullet, \bullet}_{\Delta_M}(M\mid M),\dd^L, \dd^R)$ introduced in \cref{Sec: Geometry of bi-forms}, we consider the following definition.
\begin{definition}
\label{def:leb}
A local bi-form $\varpi\in \Omega^{1,1}_{\Delta_M}(M\mid M)$ is   \underline{left-exact}  if and only if there is a local bi-form $S\in \Omega^{0,1}_{\Delta_M}(M\mid M)$ such that $\varpi=\mathrm d^LS$.
\end{definition}
Notice that for every system of (statistical) left homotopy operators, a left-exact bi-form $\varpi$ satisfies
\begin{equation}
\varpi=\mathcal E_L^{1,1}\varpi\,,\qquad \mathcal A_L^{1,1}\varpi=0\,.
\end{equation}
From \cref{Thm: torsion bi-form}, we see that left-exact contrast bi-forms on a manifod $M$ induce  SMAT structures, namely the corresponding connection  $\nabla^\varpi$ on $M$ has vanishing torsion.

Consider an element $S\in \Omega^{0,1}_{\Delta_M}(M\mid M)$ and write for it the identity \eqref{Eq: dR and iota ast do not commute}. For   $X,Y\in\mathfrak{X}(M)$, it is
  $$
\left(\mathrm d(S_{\restriction \Delta_M})-(\mathrm d^RS)_{\restriction \Delta_M}\right)(X,Y)=(\mathrm d^LS)_{\restriction \Delta_M}(X,Y)-(\mathrm d^LS)_{\restriction \Delta_M}(Y,X)\,. 
$$ 
This relation proves the following result.
\begin{proposition}\label{Prop: constraint pre-contrast function}
 Let $S\in \Omega^{0,1}_{\Delta_M}(M\mid M)$ on a smooth manifold $M$. If the element  $\varpi=\mathrm d^LS\in\Omega^{1,1}_{\Delta_M}(M\mid M)$ is a contrast bi-form, that is $g^\varpi=(\dd^LS)_{\restriction\Delta_M}$ is a pseudo-Riemannian metric on $M$, then 
    \begin{equation}\label{Eq: constraint pre-contrast function}
        \mathrm d(S_{\restriction \Delta_M})=\left(\mathrm d^RS\right)_{\restriction\Delta_M}\,.
    \end{equation}
\end{proposition}

\begin{remark}\label{rem:vS}
When dealing with left-exact contrast bi-forms $\varpi$, one may always choose a left potential that vanishes along the diagonal. Indeed, suppose $\varpi=\mathrm d^L S$ for a given  $(0,1)$-bi-form $S$. Define
  \begin{equation}
\label{12.05.3}
    \overline S=\mathcal A_R^{0,1}S=S-1\boxtimes S_{\restriction\Delta_M}\,.
\end{equation} 
Since the second term depends only on the right variable, it is $\mathrm d^L(1\boxtimes S_{\restriction\Delta_M})=0$. Consequently, one has
  \begin{equation}
 \mathrm d^L\overline S=\mathrm d^L S=\varpi\,.
\end{equation} 
Moreover, the bi-form   $\overline S=\mathcal A_R^{0,1}S$ vanishes along the diagonal, i.e. one directly sees that $\overline S_{\restriction\Delta_M}=0$. In terms of this sort of   \emph{normalized}  potential, the constraint \eqref{Eq: constraint pre-contrast function} takes the simpler form $\left(\mathrm d^R\overline S\right)_{\restriction\Delta_M}=0$.
\end{remark}
Even though left-exact contrast bi-forms yield SMATs, it is not generally true that a contrast bi-form $\varpi$ generating a SMAT is left-exact. We show, however, that $\varpi$ is always statistically equivalent to   \underline{any} of its left-exact part. 
\begin{theorem}\label{Thm: from contrast bi-form to pre-contrast}
    Let $\varpi\in \Omega^{1,1}_{\Delta_M}(M\mid M)$ be a contrast bi-form on a manifold $M$ such that $(M,g^\varpi,\nabla^\varpi)$ is a SMAT (that is   $\nabla^\varpi$ is torsion-free ). Then there exists $S \in \Omega_{\Delta_M}^{0,1}(M \mid M)$ with $S_{\restriction \Delta_M} = 0$ such that $\varpi$ is statistically equivalent to $\mathrm{d}^L S$. 
\end{theorem}

\begin{proof}
    Let $J_L$ be a system of statistical left homotopy operators on $M$. Decompose $\varpi$ into its left-exact and left-antiexact components with respect to $J_L$, namely write 
      \begin{equation}\label{Eq: decomposition contrast bi-form}
        \varpi = \mathrm{d}^L S + \kappa\,, \qquad\mathrm{with}\qquad  
        \begin{cases}
            S=J_L^{1,1} \varpi\,, \\
            \kappa = \mathcal A_L^{1,1} \varpi=J_L^{2,1}\dd^L\varpi.
        \end{cases}
    \end{equation} 
Since the  \cref{Def: statistical homotopy operators} applies, one has  $S_{\restriction \Delta_M} = (J_L\varpi)_{\restriction\Delta_M}= 0$. Furthermore,  by   \cref{Thm: torsion bi-form}  one has that  the term $\mathrm{d}^L \varpi$ vanishes along $\Delta_M$. Since the action of a statistical homotopy operator increases the vanishing order of a bi-form (see \cref{Prop: further properties JNabla}), the term  $\kappa=J_L\left(\mathrm d^L\varpi\right)$ vanishes to first order along $\Delta_M$. This yields that  $\varpi$ and $\mathrm{d}^L S$ are statistically equivalent.
\end{proof}
\begin{remark}\label{Remark: pre-contrast} 
Let $S$ be the local $(0,1)$-bi-form constructed along the lines of the proof of \cref{Thm: from contrast bi-form to pre-contrast} for a SMAT $(M,g^\varpi,\nabla^\varpi)$ out of a contrast bi-form $\varpi$ on $M$.
\begin{enumerate}[(1)]
\item Let $S'=J_L'\varpi$ be the local $(0,1)$-bi-form constructed from another system of statistical left homotopy operators. From the homotopy relations, one has
\begin{equation}
    \mathrm d^LS-\mathrm d^LS'=(J_L'-J_L)\,\mathrm d^L\varpi\,.
\end{equation}
Since $\mathrm d^L\varpi$ vanishes along $\Delta_M$ (again by \cref{Thm: torsion bi-form}), one has that  $\mathrm d^LS-\mathrm d^LS'$ vanishes to first order along $\Delta_M$, and therefore the contrast bi-form $\dd^LS'$ is statistically equivalent to $\dd^LS$. 

Notice that, if $\mathrm d^L\varpi=0$ on an open neighbourhood of  $\Delta_M$ in $M\times M$, then  $S$ and $S'$ have the same germ along $\Delta_M$. To prove this, note that, when $\dd^L\varpi=0$, the term  $S - S'$ is a left-closed $(0,1)$-bi-form on $\mathscr U$, and therefore the homotopy formula \eqref{Eq: homotopy formula bi-forms L} reads
\begin{equation}
    S-S'=1 \boxtimes (S - S')_{\restriction \Delta_M}\,.
\end{equation}
Since both $S$ and $S'$ vanish along $\Delta_M$, it follows that $S=S'$ on the intersection of their domains.
\item Let $\nabla$ be an affine connection on $M$, and consider the statistical left homotopy operator $J_L^\nabla$. Then, by \eqref{Eq: JL explicit}, if $(m,n)$ is a point in a strongly $\nabla$-convex neighbourhood contained in the domain of $\varpi$, one has 
  \begin{equation}\label{Eq: explicit pre-contrast}
S(\mid Y)(m,n) = J_L^\nabla\varpi(\mid Y)(m,n)=\int_0^1 \varpi_{(\gamma_{m \leftarrow n}(t),n)}\!\left(\dot{\gamma}_{m \leftarrow n}(t), Y_n\right) \,\mathrm{d}t\,,
\end{equation} 
for each $Y \in \mathfrak X(M)$. 

\item The fiberwise counterpart   $\dot{S}\colon M \times \T M \to \mathbb{R}$ is a pre-contrast function in the sense of Henmi and Matsuzoe. A direct inspection indeed proves that the algorithm of Henmi and Matsuzoe for extracting the metric and connection from $\dot{S}$ (cf.~\eqref{Eq: metric pre-contrast HM} and \eqref{Eq: connection pre-contrast HM}) coincides with that derived from $\mathrm{d}^L S$ (cf.~\eqref{Eq: induced metric} and \eqref{Eq: induced connection}). This shows that the formalism we are describing allows to consider pre-contrast functions as a special case of left-exact contrast bi-forms.
\end{enumerate}
\end{remark}

Having established \cref{Prop: constraint pre-contrast function} in the left-exact case, one may ask what remains true for a general contrast bi-form. The same homotopy construction still applies and yields a left-exact contrast bi-form inducing the same metric, but in general a different affine connection.

\begin{proposition}\label{Prop: further considerations pre-contrast}
    Let $\varpi\in \Omega^{1,1}_{\Delta_M}(M\mid M)$ be a contrast bi-form, and let $J_L$ be a system of statistical left homotopy operators. 
\begin{enumerate}[(a)]
\item The term $\varpi_S=\mathrm d^LS$, with $S=J_L\varpi$, is a contrast bi-form, and one has 
    \begin{equation}\label{Eq: same metric pre-contrast}
        g^{\varpi_S}=g^{\varpi}\,.
    \end{equation}
\item The difference tensor $C=\nabla^{\varpi}-\nabla^{\varpi_S}$ is implicitly defined by
      \begin{equation}\label{Eq: new affine connection pre-contrast}
        g^\varpi\left(C(Z,X),Y\right)=\left(\mathcal L_{Z\mid }(\mathcal A_L^{1,1}\varpi)\right)_{\restriction \Delta_M}(X,Y)\,,
    \end{equation} 
for  $X,Y,Z\in\mathfrak{X}(M)$, with  $\mathcal A_L^{1,1}\varpi=J_L\mathrm d^L\varpi$. The bi-forms $\varpi$ and $\varpi_S$ are statistically equivalent if and only if $(M,g^\varpi,\nabla^\varpi)$ is a SMAT.
\end{enumerate}
\end{proposition}
\begin{proof}
    Consider again  the decomposition \eqref{Eq: decomposition contrast bi-form}, so to have
    \begin{equation}
        \varpi = \mathrm d^LJ_L\varpi + J_L\mathrm d^L\varpi
        = \varpi_S + \mathcal A_L^{1,1}\varpi,
    \end{equation}
    where $S=J_L\varpi$ and $\mathcal A_L^{1,1}\varpi=J_L\left(\mathrm d^L\varpi\right)$. By the properties of statistical left homotopy operators, the term $\mathcal A_L^{1,1}\varpi$ vanishes along $\Delta_M$:  it then follows that $\varpi$ and $\varpi_S=\mathrm d^L S$ have the same restriction to $\Delta_M$, and therefore they induce the same metric on $M$. In particular, also $\varpi_S$ is  a contrast bi-form.

    To compute the difference tensor $C$, let $X,Y,Z\in \mathfrak X(M)$. From \eqref{Eq: induced connection}, one has
      \begin{equation}
    \begin{split}
        g^\varpi\left(\nabla^\varpi_Z X, Y\right)&=
        \left(\mathcal L_{Z\mid}\varpi\right)_{\restriction \Delta_M}\left(X,Y\right),
        \\
        g^{\varpi_S}\left(\nabla^{\varpi_S}_Z X, Y\right)
        &=
        \left(\mathcal L_{Z\mid}\varpi_S\right)_{\restriction \Delta_M}\left(X,Y\right)\,.
        \end{split}
    \end{equation} 
Since $g^\varpi=g^{\varpi_S}$, comparing the two expressions reads
      \begin{equation}
        g^\varpi\left(C(Z,X),Y\right)
        =g^\varpi\left(\nabla^\varpi_Z X, Y\right)-g^{\varpi_S}\left(\nabla^{\varpi_S}_Z X, Y\right)
        =\left(\mathcal L_{Z\mid}(\mathcal A_L^{1,1}\varpi)\right)_{\restriction \Delta_M}\left(X,Y\right)\,.
    \end{equation} 
When $(M,g^\varpi, \nabla^\varpi)$ is a SMAT one has $(\dd^L\varpi)_{\restriction\Delta_M}=0$, and therefore the bi-form $J_L\dd^L\varpi=\mathcal{A}_L^{1,1}\varpi$ vanishes to order 1, which gives that   $\left(\mathcal L_{Z\mid}(\mathcal A_L^{1,1}\varpi)\right)_{\restriction \Delta_M}(X,Y)=0$ for any $X,Y,Z$ on $M$. Since $g^\varpi$ is non-degenerate, it yields that $\nabla^\varpi=\nabla^{\varpi_S}$, which means that $\varpi$ is statistically equivalent to  $\varpi_S$.

If $\varpi$ is statistically equivalent to $\varpi_S$, the affine connection $\nabla^\varpi$ coincides with the one induced by the left-exact bi-form $\varpi_S=\mathrm d^LS$, which is torsion-free by \cref{Thm: torsion bi-form}. Therefore $(M,g^\varpi,\nabla^\varpi)$ is a SMAT.
\end{proof}
\subsubsection{Bi-exact contrast bi-forms} 
  \cref{def:leb} admits the following particular case.
\begin{definition}
\label{def:beb}
We say that a local bi-form $\varpi \in \Omega^{1,1}_{\Delta_M}(M \mid M)$ is \underline{bi-exact} if there exists $D \in \Omega^{0,0}_{\Delta_M}(M \mid M)$ such that $\varpi = \mathrm{d}^L \mathrm{d}^R D$. 
\end{definition}

Specialising the case of left-exact contrast bi-forms,   \cref{Thm: torsion bi-form}  implies that bi-exact contrast bi-forms on a manifold $M$ induce a statistical manifold structure, since for bi-exact contrast bi-form $\varpi$ on $M$ one has that both connections $\nabla^\varpi$ and $\nabla^{(\varpi^\dag)}$ have vanishing torsion.  For $\varpi=\dd^L\dd^RD$, the geometry of the induced statistical structure is encoded in the   local  $(0,0)$-bi-form $D$ on $M$, that is a smooth real-valued function defined on an open neighbourhood of the diagonal submanifold in $M\times M$. 

We begin such analysis by noticing that the  symmetry of $g^{\mathrm d^L\mathrm d^RD}$ imposes differential constraints into $\mathrm d^LD$ and $\mathrm d^RD$. In analogy to   \cref{Prop: constraint pre-contrast function} , one has the following result.  
\begin{proposition}\label{Prop: constraint contrast function}
    Let $\varpi\in\Omega^{1,1}_{\Delta_M}(M\mid M)$ be a bi-exact contrast bi-form on a manifold $M$, and let $D \in \Omega^{0,0}_{\Delta_M}(M \mid M)$ be such that $\varpi = \mathrm{d}^L \mathrm{d}^R D$. It is
    \begin{align}
        \mathrm{d} \left( \left( \mathrm{d}^L D \right)_{\restriction \Delta_M} \right) = 0\,, \qquad 
        \mathrm{d} \left( \left( \mathrm{d}^R D \right)_{\restriction \Delta_M} \right) = 0\,.
    \end{align}
\end{proposition}
\begin{proof}
    Let $X,Y$ be vector fields on $M$. If we write both identities   \eqref{Eq: dL and iota ast do not commute}-\eqref{Eq: dR and iota ast do not commute} for the exact bi-forms $T = \mathrm{d}^L D$ and $T'=\dd^RD$, we have
      \begin{equation}
    \begin{split}
 &\dd\left((\dd^RD)_{\restriction\Delta_M}\right)(X,Y)=\left(\dd^R(\dd^LD)\right)_{\restriction\Delta_M})(Y,X)-    \left(\dd^R(\dd^LD)\right)_{\restriction\Delta_M})(X,Y), \\
 &\dd\left((\dd^LD)_{\restriction\Delta_M}\right)(X,Y)=\left(\dd^L(\dd^RD)\right)_{\restriction\Delta_M})(Y,X)-    \left(\dd^L(\dd^RD)\right)_{\restriction\Delta_M})(X,Y),
\end{split}
    \end{equation} 
    Since $\varpi$ is a contrast bi-form (that is $g^\varpi=\varpi_{\restriction\Delta_M}=(\dd^L\dd^RD)_{\restriction \Delta_M}$ is symmetric), both the left-hand sides vanish, so both forms $\left( \mathrm{d}^L D \right)_{\restriction \Delta_M}$ and $\left( \mathrm{d}^R D \right)_{\restriction \Delta_M}$ are closed. 
 Before closing this proof, we notice that one can also write  
      \begin{equation}
    \label{12.05.2}
        \left( \mathrm{d}^R D \right)_{\restriction \Delta_M} +\left( \mathrm{d}^L D \right)_{\restriction \Delta_M}=\mathrm{d} \left( D_{\restriction \Delta_M} \right)\,. 
    \end{equation} 
    It follows that, since the forms $\left( \mathrm{d}^L D \right)_{\restriction \Delta_M}$ and $\left( \mathrm{d}^R D \right)_{\restriction \Delta_M}$ differ by an exact form, one of them is closed if and only if the other is closed. 
\end{proof}
In analogy with \cref{rem:vS} , one can notice what follows.
\begin{remark}
Let $\varpi = \mathrm{d}^L \mathrm{d}^R D\in \Omega^{1,1}_{\Delta_M}(M\mid M)$ be a bi-exact contrast bi-form on a smooth manifold $M$. Define a $(0,0)$-bi-form  $\overline{D}$ as follows,
\begin{equation}
    \overline{D} = D - D_{\restriction \Delta_M} \boxtimes 1 - J_R \big( 1 \boxtimes (\mathrm{d}^R D)_{\restriction \Delta_M} \big)\,,
\end{equation}
where $J_R$ is any statistical right homotopy operator. Notice that $\overline{D}$ is independent of the choice of $J_R$, since $(\mathrm{d}^R D)_{\restriction \Delta_M}$ is closed. Moreover, from the homotopy relations for $J_R$, one sees that  
\begin{equation}
\label{12.05.1}
\mathrm{d}^R \overline{D} = \mathrm{d}^R D - 1 \boxtimes (\mathrm{d}^R D)_{\restriction \Delta_M}
\end{equation}
and therefore 
\begin{equation}
\mathrm{d}^L \mathrm{d}^R \overline{D} = \mathrm{d}^L \mathrm{d}^R D\,:
\end{equation}
the potentials $D$ and $\overline{D}$ induce the same statistical structures. Since, for any bi-form $\varpi$, the bi-form $J_R\varpi$ vanishes along $\Delta_M$ (see the definition   \cref{Def: statistical homotopy operators}  of statistical homotopy operators), one has  $\overline{D}_{\restriction \Delta_M} = 0$. From the previous identity \eqref{12.05.1}, one has $(\mathrm{d}^R \overline{D})_{\restriction\Delta_M}=0$. Since the potential  $\overline D$  induces the same statistical structure as $D$ does, from \eqref{12.05.2} one can write $\left( \mathrm{d}^R \overline D \right)_{\restriction \Delta_M}+\left( \mathrm{d}^L \overline D \right)_{\restriction \Delta_M}=\mathrm{d} \left( \overline D_{\restriction \Delta_M} \right)$: together with $\overline D_{\restriction\Delta_M}=0$, this reads $(\dd^L\overline D)_{\restriction\Delta_M}=0$. Recollecting our terms, we have 
\begin{equation}
\overline{D}_{\restriction \Delta_M} = 0\,, \qquad (\mathrm{d}^L \overline{D})_{\restriction \Delta_M} = (\mathrm{d}^R \overline{D})_{\restriction \Delta_M} = 0\,,
\end{equation}
i.e., $\overline{D}$ vanishes to first order along $\Delta_M$.

Before closing this remark, we notice that the potential $D'=D-D_{\restriction\Delta_M}\boxtimes 1$  does induce the same statistical structure $\varpi=\dd^L\dd^RD'$, but such potential $D'$ vanishes on $\Delta_M$ \underline{only} to order zero.
 \end{remark}

If $\varpi$ is a contrast bi-form that induces  a statistical manifold, then $\varpi$ is not generally bi-exact. However, the following theorem shows that bi-exactness holds up to statistical equivalence.
\begin{theorem}\label{Thm: from pre-contrast to contrast functions}
    Let $\varpi$ be a contrast bi-form on a manifold $M$ such that $(M, g^\varpi, \nabla^\varpi)$ is a statistical manifold. Then there exists $D \in \Omega_{\Delta_M}^{0,0}(M \mid M)$ such that $D$ vanishes to first order along $\Delta_M$ and $\varpi$ is statistically equivalent to $\mathrm{d}^L \mathrm{d}^R D$.
\end{theorem}
\begin{proof}
    Let $J_L$ be a system of statistical left homotopy operators and let $J_R$ be a system of statistical right homotopy operators. By the previous \cref{Thm: from contrast bi-form to pre-contrast} on left-exact contrast bi-forms,  since the connection $\nabla^\varpi$ has vanishing torsion one has that  $\varpi$ is statistically equivalent to
    \begin{equation}\label{Eq: left exact part}
        \varpi_0 = \mathrm{d}^L S, \qquad S = J_L \varpi\,
    \end{equation}
with $S\in\Omega_{\Delta_M}^{0,1}(M\mid M)$. We can   decompose  $S$ into its right-exact and right-antiexact components with respect to $J_R$:
    \begin{equation}\label{Eq: decomposition bi-form contrast}
        S = \mathrm{d}^R D + \Sigma\,, \qquad\mathrm{with}\qquad 
        \begin{cases}
            D = J_R S\,, \\ 
            \Sigma = \mathcal A^{0,1}_RS=J_R^{0,2}\dd^RS\,.
        \end{cases}
    \end{equation}
    The action of the left exterior derivative gives
    \begin{equation}
        \varpi_0 = \mathrm{d}^L \mathrm{d}^R D + \mathrm{d}^L \Sigma\,.
    \end{equation}
Since $\varpi$ and $\varpi_0$ are statistically equivalent, our aim is to prove that $\varpi_0$ is statistically equivalent to $\dd^L\dd^RD$. We analyse the difference term $\dd^L\Sigma=\varpi_0-\dd^L\dd^RD$. 

We notice that, from the left-homotopy formula \eqref{Eq: homotopy formula bi-forms L}, we can write 
\begin{equation}
\label{12.05.4}
\mathrm{d}^R S = 1 \boxtimes \left( \mathrm{d}^R S \right)_{\restriction \Delta_M} + J_L \mathrm{d}^L\left( \mathrm{d}^R S \right) 
        = J_L \mathrm{d}^L \left( \mathrm{d}^R S \right)\,,
\end{equation}
where the last relation comes from the freedom (see \eqref{Eq: constraint pre-contrast function}) to select a bi-form $S$ in \eqref{Eq: left exact part} which vanish on $\Delta_M$. We have then that $\dd^RS$ is $J_L$-left-antiexact. This allows to write the following chain of equalities
  \begin{equation}
\label{12.05.5}
\Sigma=J_R(\dd^RS)=J_RJ_L(\dd^L\dd^RS)=J_RJ_L(\dd^R\dd^LS)=J_RJ_L(\dd^R\varpi_0).
\end{equation} 
Since $\varpi_0$ is statistically equivalent to $\varpi$, from   \cref{th:equi}  one has that the bi-form
$\varpi_0 - \varpi$ vanishes to first order along $\Delta_M$ and its right exterior derivative $\mathrm{d}^R (\varpi_0 - \varpi)$ vanishes along $\Delta_M$.  Moreover, since the connection $(\nabla^\varpi)^\dagger$ has vanishing torsion,   \cref{Thm: torsion bi-form}  implies that $\mathrm{d}^R \varpi$ also vanishes along $\Delta_M$. Therefore, the bi-form
    \begin{equation}
        \mathrm{d}^R \varpi_0 = \mathrm{d}^R (\varpi_0 - \varpi) + \mathrm{d}^R \varpi
    \end{equation}
    vanishes along $\Delta_M$. From the expression $\Sigma = J_R J_L \mathrm{d}^R \varpi_0$ and the defining properties of statistical left and right homotopy operators, we conclude that $\Sigma$ vanishes to second order along $\Delta_M$. Consequently, $\mathrm{d}^L \Sigma$ vanishes to first order. This shows that $\varpi_0$ is statistically equivalent to $\mathrm{d}^L \mathrm{d}^R D$, completing the proof.
\end{proof}

\begin{remark}\label{Remark: contrast function}
    Let $D$ be the local $(0,0)$-bi-form constructed along the lines of the proof of   \cref{Thm: from pre-contrast to contrast functions}  for a statistical manifold $(M, g^\varpi, \nabla^\varpi)$ out of a bi-exact contrast bi-form $\varpi$ on $M$.
    \begin{enumerate}[(1)]
        \item Proceeding as in \cref{Remark: pre-contrast}, one can show that choosing another statistical left homotopy system $J_L'$ and another statistical right homotopy system $J_R'$ leads to a potential $D'=J_R'J_L'\varpi$ which turns out to induce the same statistical structure, namely the bi-forms  $\mathrm d^L\mathrm d^RD$ and $\mathrm d^L\mathrm d^RD'$ are statistically equivalent. If $\mathrm d^L\varpi=0$ and $\mathrm d^R\varpi=0$, then $D=D'$ on  the intersection of their domains.
        \item Let $\nabla_1,\nabla_2$ be affine connections on $M$, consider  the statistical left homotopy system $J_L^{\nabla_1}$ and the statistical right homotopy system $J_R^{\nabla_2}$. Then, by \eqref{Eq: JLJR formula mixed}
\begin{equation}\label{Eq: explicit contrast mixed}
    D(m, n) = \iint_{[0,1]^2}
    \varpi_{(a_r(t),b(r))} \left( \dot a_r(t), \dot b(r) \right)\,\mathrm{d}t \, \mathrm{d}r\,,
\end{equation}
where
\begin{equation}
    b(r)=\gamma^{\nabla_2}_{n\leftarrow m}(r)\,, \qquad 
    a_r(t)=\gamma^{\nabla_1}_{m \leftarrow b(r)}(t)\,.
\end{equation}
Here $(m,n)$ belong to a neighbourhood which is both strongly $\nabla_1$-convex and strongly $\nabla_2$-convex, contained in the domain of $\varpi$. If $\nabla_1=\nabla_2=\nabla$, then by \eqref{Eq: JLJR formula}
        \begin{equation}\label{Eq: explicit contrast}
            D(m, n) = \iint_{[0,1]^2} r \, \varpi_{(\gamma_{m \leftarrow n}(1-r+t\,r), \gamma_{n \leftarrow m}(r))} \left( \dot{\gamma}_{m \leftarrow n}(1-r+t\,r), \dot{\gamma}_{n \leftarrow m}(r) \right) \mathrm{d}t \, \mathrm{d}r\,,
        \end{equation}
        where $(m, n)$ belong to a strongly $\nabla$-convex neighbourhood contained in the domain of $\varpi$.
        \item The real-valued function $D$ is a contrast function in the sense of Eguchi. Moreover, the algorithm of Eguchi for extracting the metric and connection from $D$ (cf.~\eqref{Eq: metric contrast E} and \eqref{Eq: connection contrast E}) coincides with that derived from $\mathrm{d}^L \mathrm{d}^R D$ (cf.~\eqref{Eq: induced metric} and \eqref{Eq: induced connection}). Hence, the bi-form setting allows to  describe contrast functions as a particular case of bi-forms, namely  contrast functions are realized as bi-exact contrast bi-forms. 
    \end{enumerate}
\end{remark}
In analogy to \cref{Prop: further considerations pre-contrast}, we analyze the statistical structure  induced by a bi-exact component  of a contrast bi-form.
\begin{proposition}\label{Prop: further consideration contrast}
    Let $\varpi=\mathrm d^LS$ be a left-exact contrast bi-form with $S\in \Omega^{0,1}_{\Delta_M}(M\mid M)$ vanishing along $\Delta_M$, and let $J_R$ be a system of statistical right homotopy operators. 
\begin{enumerate}[(a)]
\item The term  $\varpi_D=\mathrm d^L\mathrm d^RD$, with $D=J_RS$, is a contrast bi-form, and 
    \begin{equation}
        g^{\varpi_D}=g^\varpi\,.
    \end{equation}
\item    
The difference tensor $C=\nabla^{\varpi}-\nabla^{\varpi_D}$ is implicitly defined by
    \begin{equation}\label{Eq: new affine connection contrast}
        g^\varpi\left(C(Z,X),Y\right)=\left(\mathcal L_{Z\mid }(\mathrm d^L\mathcal A_R^{0,1}S)\right)_{\restriction \Delta_M}(X,Y)\,,
    \end{equation}
for $X,Y,Z\in\mathfrak{X}(M)$, with  $\mathcal A_L^{0,1}S=J_R\mathrm d^RS$. The bi-forms $\varpi$ and $\varpi_D$ are statistically equivalent if and only if $(M,g^\varpi,\nabla^\varpi)$ is a statistical manifold.
\end{enumerate}
\end{proposition}
\begin{proof}
    Let $\Sigma=\mathcal A_R^{0,1}S=J_R\mathrm d^RS$, and consider the decomposition \eqref{Eq: decomposition bi-form contrast}
    \begin{equation}
        S=\mathrm d^RD+\Sigma\,.
    \end{equation}
The action of the left exterior derivative gives
    \begin{equation}
        \varpi=\varpi_D+\mathrm d^L\Sigma\,.
    \end{equation}
    Since $S$ vanishes along $\Delta_M$, the bi-form $\mathrm d^RS$ vanishes along $\Delta_M$ by   \cref{Prop: constraint pre-contrast function} . By the defining property of statistical right homotopy operators, $\Sigma$ vanishes to first order along $\Delta_M$. Consequently, $\mathrm d^L\Sigma$ vanishes along $\Delta_M$, and so it follows that $\varpi_D$ and $\varpi$ share the same restriction to $\Delta_M$, hence they induce the same metric. In particular, $\varpi_D$ is again a contrast bi-form.

    To compute the difference tensor $C$, let $X,Y,Z\in \mathfrak X(M)$. From \eqref{Eq: induced connection}, one has
    \begin{equation}
    \begin{split}
        g^\varpi\left(\nabla^\varpi_Z X, Y\right)&=
        \left(\mathcal L_{Z\mid}\varpi\right)_{\restriction \Delta_M}\left(X,Y\right),
        \\
        g^{\varpi_D}\left(\nabla^{\varpi_S}_Z X, Y\right)
        &=
        \left(\mathcal L_{Z\mid}\varpi_D\right)_{\restriction \Delta_M}\left(X,Y\right)\,.
        \end{split}
    \end{equation}
Since $g^\varpi=g^{\varpi_D}$, comparing the two expressions reads
      \begin{equation}
        g^\varpi\left(C(Z,X),Y\right)
        =g^\varpi\left(\nabla^\varpi_Z X, Y\right)-g^{\varpi_S}\left(\nabla^{\varpi_S}_Z X, Y\right)
        =\left(\mathcal L_{Z\mid}(\mathrm d^L\Sigma)\right)_{\restriction \Delta_M}\left(X,Y\right)\,.
    \end{equation} 
When $(M,g^\varpi, \nabla^\varpi)$ is a statistical manifold one has that $\mathrm d^L\Sigma$ vanishes to first order along $\Delta_M$ by the proof of   \cref{Thm: from pre-contrast to contrast functions} , which gives that $\left(\mathcal L_{Z\mid}(\mathrm d^L\Sigma)\right)_{\restriction \Delta_M}(X,Y)=0$ for any $X,Y,Z$ on $M$. Since $g^\varpi$ is non-degenerate, it yields that $\nabla^\varpi=\nabla^{\varpi_S}$, which means that $\varpi$ is statistically equivalent to  $\varpi_S$. 

Conversely, if $\varpi$ and $\varpi_D$ are statistically equivalent, then the induced connections of $\varpi$ coincide with the one given by $\varpi_D$, which are torsion-free by \cref{Thm: torsion bi-form}. Therefore $(M,g^\varpi,\nabla^\varpi)$ is a statistical manifold.
\end{proof}

\section{A class of Lauritzen manifolds}\label{Sec: dually curvaturefree LM}
Within the class of Lauritzen manifolds, i.e. those manifolds $M$ equipped with a pseudo-Riemannian metric tensor $g$ and an affine connection $\nabla$, we consider those whose connection has a vanishing curvature, i.e. $R^\nabla=R^{(\nabla^\dag)}=0$.
We analyse contrast bi-forms for this family of statistical structures, generalising the notion of partially flat SMAT \cite{H-M-2019} and including as a special case the dually flat statistical manifolds of information geometry \cite{Amari-1985, Amari-2016}.

\subsection{Dually curvature-free Lauritzen manifolds}
Let $(M,g,\nabla)$ be a Lauritzen manifold whose affine connection has vanishing curvature, i.e. $R^\nabla=0$. Then by \cite[Lemma 7.8]{Lee-2018}, for any $c\in M$, there exists a local frame $(U,\{E_j\}_{j=1}^d)$ of vector fields, with $U\subseteq M$ an open neighbourhood of $c$, such that
\begin{equation}
    \nabla E_j = 0 \,.
\end{equation} 
 Moreover, any other local frame $(U',\{E_j'\}_{j=1}^d)$ of this type is related to the former by a constant function $A\colon U\cap U'\to \operatorname{GL}(\mathbb R^d)$ with
  \begin{equation}\label{Eq: parallel frames}
    E_j' = A_j^i\, E_i \,.
\end{equation}
Let $(U,\{\varepsilon^j\}_{j=1}^d)$ be the dual coframe, so that $\varepsilon^j(E_k)=\delta^j_k$. Defining $F_j=(\varepsilon^j)^\sharp$, one sees that $(U,\{F_j\}_{j=1}^d)$ is another local frame on $U\subseteq M$ such that 
\begin{equation}
    \nabla^\dag F_j = 0 \,.
\end{equation}
Moreover, one has
\begin{align}
    \Tor^\nabla(E_i,E_j)&=-[E_i,E_j]\,, \qquad i,j\in \{1,\dots,d\}\,,\\
    \Tor^{(\nabla^\dag)}(F_a,F_b)&=-[F_a,F_b]\,, \qquad a,b\in \{1,\dots,d\}\,,
\end{align}
showing that $(M,g,\nabla)$ is a Lauritzen manifold which may be torsion-full. We refer to such a structure as a \underline{dually curvature-free Lauritzen manifold}. Following standard nomenclature in the literature:
\begin{itemize}
    \item a dually curvature-free Lauritzen manifold that is also a SMAT is called a \underline{partially flat SMAT} \cite{H-M-2019};
    \item a dually curvature-free Lauritzen manifold that is also a statistical manifold is called a \underline{dually flat statistical manifold}   \cite{A-T-2002}.
\end{itemize}
\begin{remark}
We notice  that teleparallel Lauritzen manifolds are a particular subclass of dually curvature-free Lauritzen manifolds \cite{GSI, C-DC-I-M-2023}. Given a parallelisable manifold $M$, we say that a (curvature-free) affine connection $\nabla$ is \underline{teleparallel} if and only if there is global frame $\mathcal B=\{\varepsilon^i\}_{i=1}^d$ of $\T^\ast M$ such that
\begin{equation}
    \nabla^\ast \varepsilon^i = 0 \,.
\end{equation}
We also write $\nabla=\nabla^{\mathcal B}$. Affine connections of this type are also called \emph{pseudo-Weitzenböck affine connections} \cite{Z-K-2019}. 
\end{remark}
\begin{remark}\label{Remark: torsion and topological/algebraic obstruction}
It is well-known that the existence of dually flat statistical structures is subject to topological obstructions. We limit ourselves to recall that  compact manifolds with finite fundamental group do not admit flat affine connections \cite{A-T-2002, A-T-2003}: when the condition of vanishing torsion is relaxed while retaining vanishing curvature, these obstructions disappear (see \cref{Subsec: Lie group} for a class of examples). 
\end{remark}

\subsubsection{A contrast bi-form for dually curvature-free Lauritzen manifolds} 
Let $(M,g,\nabla)$ be a dually curvature-free Lauritzen manifold. Let $\mathcal A$ be an open cover of $M$ such that, for every $U\in\mathcal A$, there exists a local $\nabla$-parallel frame $\{\varepsilon^i\}_{i=1}^d$ of $1$-forms on $U$. For each $U\in\mathcal A$, let $\{\alpha^i\}_{i=1}^d$ be the dual coframe to the $g$-gradient vector fields $\{F_j=(\varepsilon^j)^\sharp\}_{j=1}^d$. Define the local bi-form  
  \begin{equation}\label{Eq: solution bi-form locally}
    \varpi_{\restriction U\times U}=\sum_{i=1}^d\varepsilon^i\boxtimes \alpha^i \,.
\end{equation} 
We first notice that, by the transformation property \eqref{Eq: parallel frames}, this local expression is independent of the choice of parallel frame: for any $U'\in \mathcal A$ with $U\cap U'\ne \emptyset$, we have:
  \begin{equation}
    \left(\varpi_{\restriction U'\times U'}\right)_{\restriction (U\cap U')\times(U\cap U')}
        = \left(\varpi_{\restriction U\times U}\right)_{\restriction (U\cap U')\times(U\cap U')} \,.
\end{equation} 
Thus, the above local definitions patch together to yield a well-defined local bi-form $\varpi$ on the open neighbourhood $\mathscr U_{\mathcal A}=\bigcup_{U\in\mathcal A}(U\times U)$ of the diagonal submanifold $\Delta_M$ in $M\times M$. The germ of $\varpi$ along $\Delta_M$ does not depend on the chosen cover $\mathcal A$. Moreover, since one has $\alpha^i=g_{ij}\varepsilon^j$ for $g_{ij}=g(E_i,E_j)$, one computes that 
  \begin{equation}
    \varpi_{i\mid j}
        =\pi_R^\ast g_{ij} \,.
\end{equation} 
From 
  \begin{align}
    \varpi_{\restriction\Delta_M}
        &= \iota^\ast \varpi_{i\mid j}\, \varepsilon^i\otimes \varepsilon^j \,, \\[4pt]
    \left(\mathcal L_{E_k^L}\varpi\right)_{\restriction\Delta_M}
        &= \left(\iota^\ast\mathcal L_{E_k^L}\varpi_{i\mid j}\right)\, 
           \varepsilon^i\otimes \varepsilon^j \,, 
           \qquad k\in \{1,\dots,d\} \,.
\end{align} 
one sees that $g^\varpi=g$ and $\nabla^\varpi=\nabla$: the local bi-form $\varpi$ defined above is a contrast bi-form for the dually curvature-free Lauritzen manifold $(M,g,\nabla)$.  
We will refer to this contrast bi-form as the \underline{solution bi-form}\footnote{Upon calling such element $\varpi$ as a solution bi-form, we implicitly refer to a more general problem -- which we aim to address in a future paper -- of analysing how to have a contrast bi-form for \emph{any} Lauritzen manifold.} for the given dually curvature-free Lauritzen manifold $(M,g,\nabla)$. It enjoys several remarkable properties that we aim to analyze in the remainder of this section.

We begin by noticing that the dual bi-form $\varpi^\dag$ to the solution bi-form $\varpi$ for $(M,g,\nabla)$ is the solution bi-form of $(M,g,\nabla^\dag)$. Indeed, if $(V,\{F_j\}_{j=1}^d)$ is a local $\nabla^\dag$-parallel frame on $M$, the local frame of $g$-gradient vector fields $(V,\{E_i\}_{i=1}^d)$ of a dual coframe $(V,\{\alpha^j\}_{j=1}^d)$ is $\nabla$-parallel. Hence the solution bi-form on $V\times V$ reads:
\begin{equation}\label{Eq: solution bi-form locally dual}
    \varpi'=\sum_{i=1}^d \alpha^i\boxtimes \varepsilon^i\,,
\end{equation} 
where $(V,\{\varepsilon^i\}_{i=1}^d)$ is the dual coframe of $(V,\{E_i\}_{i=1}^d)$. The bi-form $\varpi'$ clearly is the dual bi-form of the bi-form \eqref{Eq: solution bi-form locally} .

The following result refines \cref{Thm: torsion bi-form} for the solution contrast bi-form \eqref{Eq: solution bi-form locally}. For a general contrast bi-form $\varpi$, the torsion-free condition for $\nabla^\varpi$ is equivalent to $(\dd^L\varpi)_{\restriction\Delta_M}=0$, while the torsion-free condition for $\nabla^{(\varpi^\dag)}$ is equivalent to $(\dd^R\varpi)_{\restriction\Delta_M}=0$. In the curvature-free case, these diagonal conditions strengthen to closedness properties of the solution bi-form, as the next proposition shows.

\begin{proposition}\label{Prop: canonical contrast and precontrast}
Let $\varpi$ be the solution bi-form associated to a dually curvature-free Lauritzen manifold $(M,g,\nabla)$
 \begin{enumerate}[(a)]
    \item 
    The manifold $(M,g,\nabla)$ is partially flat, that is, $(\mathrm d^L\varpi)_{\restriction \Delta_M}=0$, if and only if $\mathrm d^L\varpi=0$. In this case, the germ along $\Delta_M$ of a bi-form $S\in \Omega^{0,1}_{\Delta_M}(M\mid M)$ such that $S_{\restriction \Delta_M}=0$ and $\varpi=\mathrm d^LS$ is unique. In a strongly $\nabla$-convex neighbourhood $\mathscr U$ contained in the domain of definition of $\varpi$, it can be written as \begin{equation}\label{Eq: canonical pre-contrast}
        S(\mid Y)(m,n)
        = g_n\!\left(\dot\gamma_{m\leftarrow n}(0),\, Y_n\right)\,,
    \end{equation}
    for all $Y\in \mathfrak X(M)$ and $(m,n)\in \mathscr U$. Moreover, we have:
    \begin{equation}\label{Eq: S and torsion of dual}
        \mathrm d^RS(\mid Y_0,Y_1)(m,n)=g_n\left(\dot \gamma_{m\leftarrow n}(0),\left(\Tor^{(\nabla^\dag)}(Y_0,Y_1)\right)_n\right)\,,
    \end{equation}
    for all vector fields $Y_0,Y_1$ on $M$, and $(m,n)\in \mathscr U$.
    \item 
    The manifold $(M,g,\nabla)$ is dually flat, that is, $(\mathrm d^L\varpi)_{\restriction \Delta_M}=(\mathrm d^R\varpi)_{\restriction\Delta_M}=0$, if and only if $\mathrm d^L\varpi=0$ and $\mathrm d^R\varpi=0$. In this case, the germ along $\Delta_M$ of a bi-form $D\in \Omega^{0,0}_{\Delta_M}(M\mid M)$ such that $D_{\restriction \Delta_M}=0$, 
    $(\mathrm dD)_{\restriction \Delta_M}=0$, $S=\mathrm d^RD$,
    and $\varpi=\mathrm d^L\mathrm d^R D$ is unique. Moreover, there exists a strongly $\nabla$-convex neighbourhood $\mathscr U$ contained in the domain of definition of $\varpi$ on which $D$ is represented by
    \begin{equation}\label{Eq: canonical contrast}
        D(m,n)
        = -\int_0^1 r\, \lVert\dot\gamma_{n \leftarrow m}(r)\rVert_g^2 \, \mathrm{d}r\,.
    \end{equation}
\end{enumerate}
\end{proposition}

\begin{proof}
Let $\mathcal A$ be the cover by domains of local $\nabla$-parallel coframes upon which the solution bi-form $\varpi$ is defined. For each $U\in\mathcal A$, choose a local frame $\{\varepsilon^i\}_{i=1}^d$ of $\nabla$-parallel $1$-forms on $U$, with dual frame $\{E_i\}_{i=1}^d$, and let $\{\alpha^i\}_{i=1}^d$ denote the dual coframe of the $g$-gradient vector fields $\{F_j=(\varepsilon^j)^\sharp\}_{j=1}^d$. On $U\times U$ one has
\begin{equation}
    \mathrm d^L\varpi=\sum_{i=1}^d\mathrm d\varepsilon^i\boxtimes \alpha^i\,, \qquad
    \mathrm d^R\varpi=\sum_{i=1}^d\varepsilon^i\boxtimes \mathrm d\alpha^i\,.
\end{equation}
Moreover, for all vector fields $X,Y$ on $U$,
\begin{equation}\label{Eq: torsion and d}
   \mathrm d\varepsilon^i(X,Y)=\varepsilon^i\left(\Tor^\nabla(X,Y)\right),
   \qquad
   \mathrm d\alpha^i(X,Y)=\alpha^i\left(\Tor^{(\nabla^\dag)}(X,Y)\right).
\end{equation}
 \begin{enumerate}[(a)]
    \item 
    If $\mathrm d^L\varpi=0$, then $\nabla^\varpi=\nabla$ is torsion-free (cf.~\cref{Thm: torsion bi-form}), and so $(M,g,\nabla)$ is partially flat. 
     Conversely, assume that $(M,g,\nabla)$ is partially flat. Then $\nabla$ is torsion-free, and so \eqref{Eq: torsion and d} gives $\mathrm d\varepsilon^i=0$ on every $U\in\mathcal A$. Hence $\mathrm d^L\varpi=0$ on $\mathscr U_{\mathcal A}$.

    Using \cref{Thm: moretti}, we refine $\mathcal A$ to a strongly $\nabla$-convex covering $\mathcal C$ so that $\mathscr U=\mathscr U_{\mathcal C}$ is contained in $\mathscr U_{\mathcal A}$. There exists a unique germ along $\Delta_M$ of a bi-form $S\in \Omega^{0,1}_{\Delta_M}(M\mid M)$ such that $\varpi=\mathrm d^LS$ and $S_{\restriction\Delta_M}=0$, given on $\mathscr U$ by $S=J_L^\nabla \varpi$ (cf. \cref{Remark: pre-contrast}). 

    To compute $S$, fix any $C\in \mathcal C$, choose $U_C\in\mathcal A$ with $C\subseteq U_C$, and restrict the above frames from $U_C$ to $C$. For each $(m,n)\in C\times C$ and $Y\in \mathfrak X(M)$ we have, by \eqref{Eq: explicit pre-contrast}
    \begin{equation}
        \left(J^{\nabla}_L\varpi\right)(\mid Y)(m,n)
        = \int_0^1 \sum_{i=1}^d\varepsilon^i_{\gamma_{m\leftarrow n}(t)}(\dot \gamma_{m\leftarrow n}(t))\, 
        \alpha^i_n(Y_n)\,\mathrm dt\,.
    \end{equation}
    Since any local $\nabla$-parallel $1$-form is constant along $\nabla$-geodesics within its domain, we have
    \begin{equation}\label{Eq: geodesic curvature-free (polished)}
        \varepsilon^i_{\gamma_{m \leftarrow n}(t)}\!\left(\dot \gamma_{m \leftarrow n}(t)\right)
        =\varepsilon^i_n(\dot \gamma_{m\leftarrow n}(0))\,, \qquad i\in \{1,\dots,d\}\,.
    \end{equation}
    Hence, using $g=\sum_{i=1}^d \varepsilon^i\otimes \alpha^i$, we get
    \begin{equation}
        \left(J^\nabla_L\varpi\right)(\mid Y)(m,n)
        = g_n\!\left(\dot \gamma_{m\leftarrow n}(0),\,Y_n\right)\,,
    \end{equation}
    thus proving \eqref{Eq: canonical pre-contrast}.

  To prove \eqref{Eq: S and torsion of dual}, we apply both sides to the local frame $\{F_j\}_{j=1}^d$. Since $\nabla$ is torsion-free, it follows from \eqref{Eq: torsion and d} that each $\varepsilon^j$ is an exact $1$-form on $C$, and we denote by $\ell^j\colon C\to \mathbb R$ a smooth function such that $\varepsilon^j=\mathrm d\ell^j$. Since each $F_j$ is the $g$-gradient of $\varepsilon^j=\mathrm d\ell^j$, we have on $C\times C$:
\begin{equation}
    S(\mid F_j)=\pi_L^\ast\ell^j-\pi_R^\ast\ell^j\,, \qquad j\in \{1,\dots,d\}\,.
\end{equation}
Therefore:
\begin{equation}
    S=\sum_{j=1}^d\left(\pi_L^\ast \ell^j-\pi_R^\ast \ell^j\right)\, \left(1 \boxtimes \alpha^j\right)\,.
\end{equation}
Applying $\mathrm d^R$, we obtain:
\begin{align}
    \mathrm d^RS&=-\sum_{j=1}^d 1\boxtimes\left(\mathrm d\ell^j\wedge \alpha^j\right)
    +\sum_{j=1}^d\left(\pi_L^\ast \ell^j-\pi_R^\ast \ell^j\right)\,\left( 1\boxtimes \mathrm d\alpha^j\right)\,.
\end{align}
Since $\mathrm d\ell^j=\varepsilon^j$, the first term vanishes, since the term
\begin{equation}
    \sum_{j=1}^d \varepsilon^j\wedge \alpha^j=0\,
\end{equation}
is the alternating component of the metric tensor $g=\sum_{j=1}^d \varepsilon^j\otimes \alpha^j$. Therefore, we have 
\begin{equation}
    \mathrm d^RS=\sum_{j=1}^d\left(\pi_L^\ast \ell^j-\pi_R^\ast \ell^j\right)\,\left( 1\boxtimes \mathrm d\alpha^j\right)\,.
\end{equation}
Using again \eqref{Eq: torsion and d}, restricted to $C$, we have
\begin{equation}
    \mathrm d\alpha^j(Y_0,Y_1)=\alpha^j\left(\Tor^{(\nabla^\dag)}(Y_0,Y_1)\right)\,, \qquad Y_0,Y_1\in \mathfrak X(C)\,.
\end{equation} 
Inserting this identity into the previous expression for $\mathrm d^RS$ yields \eqref{Eq: S and torsion of dual} on $C$. Since $C$ is an arbitrary element of $\mathcal C$, the conclusion follows.
    \item 
    If $\mathrm d^L\varpi=0$ and $\mathrm d^R\varpi=0$, then $(M,g,\nabla)$ is a statistical manifold (by \cref{Thm: torsion bi-form}), and therefore dually flat. Conversely, assume that $(M,g,\nabla)$ is a dually flat statistical manifold. Then both $\nabla$ and $\nabla^\dag$ are torsion-free, and so \eqref{Eq: torsion and d} gives $\mathrm d\varepsilon^i=0$ and $\mathrm d\alpha^i=0$ on every $U\in\mathcal A$. Hence $\mathrm d^L\varpi=0$ and $\mathrm d^R\varpi=0$ on $\mathscr U_{\mathcal A}$.

    On the neighbourhood $\mathscr U=\mathscr U_{\mathcal C}$ constructed as above, there exists a unique germ along $\Delta_M$ of a bi-form $D\in \Omega^{0,0}_{\Delta_M}(M\mid M)$ such that $D_{\restriction \Delta_M}=0$, $(\mathrm dD)_{\restriction \Delta_M}=0$, $S=\mathrm d^RD$, and $\varpi=\mathrm d^L\mathrm d^R D$. It is represented on $\mathscr U$ by
    \begin{equation}
        D=J_R^\nabla J_L^\nabla\varpi=J_R^\nabla S
    \end{equation} 
     (cf. \cref{Remark: contrast function}). For each $C\in \mathcal C$ and $(m,n)\in C\times C$, by \eqref{Eq: explicit contrast},
    \begin{equation}
        D(m,n)
        = \iint_{[0,1]^2} r\, \sum_{i=1}^d
        \varepsilon^i_{\gamma_{m \leftarrow n}(1-r+t\,r)}\!\left( \dot \gamma_{m \leftarrow n}(1-r+t\,r)\right)
        \alpha^i_{\gamma_{n \leftarrow m}(r)}\!\left(\dot \gamma_{n \leftarrow m}(r)\right)
        \, \mathrm dt\,\mathrm dr\,.
    \end{equation}
    Since each $\varepsilon^i$ is $\nabla$-parallel,
    \begin{equation}
        \varepsilon^i_{\gamma_{m \leftarrow n}(1-r+t\,r)}\!\left( \dot \gamma_{m \leftarrow n}(1-r+t\,r)\right)
        =\varepsilon^i_{\gamma_{m \leftarrow n}(1-r)}\!\left( \dot \gamma_{m \leftarrow n}(1-r)\right)
        =-\varepsilon^i_{\gamma_{n \leftarrow m}(r)}\!\left( \dot \gamma_{n \leftarrow m}(r)\right)\,,
    \end{equation}
    where the last equality follows from
    $\gamma_{m\leftarrow n}(1-r)=\gamma_{n\leftarrow m}(r)$ and
    $\dot \gamma_{m\leftarrow n}(1-r)=-\dot \gamma_{n\leftarrow m}(r)$. Recalling that $g=\sum_{i=1}^d \varepsilon^i\otimes \alpha^i$, we have 
      \begin{equation}
       D(m,n)
        = -\iint_{[0,1]^2} r\,\lVert \dot \gamma_{n\leftarrow m}(r)\rVert_g^2\,\mathrm dt\,\mathrm dr
        = -\int_0^1 r\,\lVert \dot \gamma_{n\leftarrow m}(r)\rVert_g^2\,\mathrm dr\,.
    \end{equation} 
 \end{enumerate}
\end{proof}

\begin{remark}
Notice  that the fiberwise linear expression of~\eqref{Eq: canonical pre-contrast} recovers the canonical pre-contrast function introduced by Henmi and Matsuzoe in the partially flat setting~\cite{H-M-2019}. Likewise, up to our sign convention (cf.~\cref{Remark: sign convention}), the expression~\eqref{Eq: canonical contrast} coincides with the canonical divergence of Ay and Amari~\cite{A-A-2015}.

 In this regard, the solution bi-form provides a natural extension of these canonical statistical potentials to the general dually curvature-free setting. We also remark that the theory of strongly convex neighbourhoods ensures a well-defined domain containing the entire diagonal submanifold for all such potentials, therefore globalising an analysis where such potentials are  only locally defined on squares of convex subsets.
 \end{remark}

Finally, we provide an alternative expression for the solution bi-form, 
which is particularly convenient when the parallel transport is explicitly known. 
\begin{proposition}\label{Prop: canonical bi-form parallel transport}
Let $\varpi$ be the solution bi-form associated to a dually curvature-free 
Lauritzen manifold $(M,g,\nabla)$. There exists a strongly $\nabla$-convex neighbourhood $\mathscr U$ contained in the domain of definition of $\varpi$ such that, for any pair of vector fields $X,Y$ on $M$ and every $(m,n)\in\mathscr U$, one has
\begin{equation}\label{Eq: canonical bi-form parallel transport}
    \varpi(X\mid Y)(m,n)=g_n\!\left(P^\nabla_{n\leftarrow m} X_m, Y_n\right)\,,
\end{equation}
where $P^\nabla_{n\leftarrow m}\colon \T_mM\to \T_nM$ denotes the $\nabla$-parallel transport along the unique $\nabla$-geodesic joining $m$ to $n$ within $\mathscr U$.
\end{proposition}

\begin{proof}
Let $\mathcal A$ be the cover by domains of local $\nabla$-parallel coframes upon which the solution bi-form $\varpi$ is defined. Using \cref{Thm: moretti}, we refine $\mathcal A$ to a strongly $\nabla$-convex covering $\mathcal C$ so that $\mathscr U=\mathscr U_{\mathcal C}$ is contained in $\mathscr U_{\mathcal A}$.
Let $X,Y\in\mathfrak X(M)$ and $(m,n)\in\mathscr U$. Choose $C\in\mathcal C$ such that $m,n\in C$, and choose $U_C\in\mathcal A$ such that $C\subseteq U_C$. Let $\{\varepsilon^i\}_{i=1}^d$ be a local $\nabla$-parallel coframe on $U_C$, restricted to $C$, with dual frame $\{E_j\}_{j=1}^d$, and let $\{\alpha^j\}_{j=1}^d$ denote the dual coframe of the $g$-gradient vector fields associated to $\{\varepsilon^i\}_{i=1}^d$. Since each $E_j$ is $\nabla$-parallel on $C$, we have
\begin{equation}
    P^\nabla_{n \leftarrow m} X_m
    = X^i(m)\, (E_i)_n\,
\end{equation} 
where $X_{\restriction C}=X^iE_i$. Accordingly, from $g=\sum_{j=1}^d \varepsilon^j\otimes \alpha^j$, we write
  \begin{align}
    \varpi(X\mid Y)(m,n)
        &= \sum_{i,j,k=1}^d X^j(m)\,\varepsilon^i(E_j)(n) Y^k(n)\, \alpha^i(E_k)(n)\\
        &=g_n\!\left(X^i(m)\, (E_i)_n,\, Y^j(n)\, (E_j)_n\right)= g_n\!\left(P^\nabla_{n \leftarrow m} X_m,\, Y_n\right)\,.
\end{align} 
Here we also decomposed $Y_{\restriction C}=Y^jE_j$. Since $C$ is arbitrary, the claim is proved.
\end{proof}
We close our analysis by noticing  that, if $(M,\nabla)$ is a \emph{space of absolute parallelism} (that is, the $\nabla$-parallel transport is independent of the chosen curve \cite{Postnikov-2001}) then the solution bi-form can be extended to the whole Cartesian square $M\times M$. Nevertheless, its associated potentials are in general only locally defined in a neighbourhood of the diagonal submanifold.

We find it now interesting to  present two examples of dually curvature-free Lauritzen manifolds, the manifold of faithful states on a finite-dimensional $C^\ast$-algebra and semisimple Lie group with left and right Cartan connections. For each case, we derive the corresponding solution bi-form and study its specific properties.

\subsection{A first example: the manifold of faithful states on a $C^\ast$-algebra}\label{Sec: examples}
Classical probability measures and quantum states can be modeled within a common $C^\ast$-algebraic framework that we now briefly recall \cite{G-R-2006-02,G-R-2006,K2014a,K2016,N2018,N2021a,CDJS2024,CIJM2019,CJS2020a,CJS2023}.

\subsubsection{Preliminaries on $C^\ast$-algebras} 
Recall that a complex Banach algebra $(\mathsf A,+,\cdot,\|\cdot\|)$ is a $C^\ast$-algebra when there is an anti-linear involution $*\colon \mathsf{A}\rightarrow\mathsf{A}$ such that $\|x\,x^{*}\|=\|x\|\,\|x^{*}\|$.  
Typical examples of $C^\ast$-algebras are continuous functions on compact Hausdorff spaces, and bounded linear operators on Hilbert spaces.

An element $a$ in the $C^\ast$-algebra $\mathsf{A}$ can be characterized as follows:
\begin{itemize}
    \item $a$ is \emph{self-adjoint} if $a^\ast=a$; we denote by $\mathsf{A}_{\sa}$ the set of self-adjoint elements;
    \item $a$ is \emph{positive} if there exists $c\in \mathsf A$ such that $a=c^\ast c$; we denote by $\mathsf{A}_+$ the set of positive elements.
    \item if $\mathsf{A}$ has an identity element $\mathbb{I}$, $a$ is \emph{invertible} if there exists $b\in\mathsf{A}$ such that $ab=ba=\mathbb{I}$; we denote by $\GL(\mathsf A)$ the set of invertible elements.
\end{itemize}
We observe that $\mathsf{A}_{\sa}$ is a real vector space that is closed in $\mathsf A$, while $\mathsf{A}_+$ is a closed convex cone in $\mathsf A$, and $\GL(\mathsf A)$ is an open subset of $\mathsf A$ that carries a  Banach--Lie group structure \cite{U1985}. We also set
\begin{equation}
   \mathsf{A}_{++}=\GL(\mathsf A)\cap \mathsf{A}_+\,,
\end{equation}
which is open in $\mathsf{A}_{\sa}$ with respect to the relative topology.

A \emph{state}  on a $C^\ast$-algebra $\mathsf A$ is a complex-linear functional $\xi\in \mathsf A^\ast$ such that:
\begin{itemize}
    \item $\xi$ is \emph{positive}, that is, $\xi(a^\ast a)\ge 0$ for each $a\in \mathsf A$;
    \item $\xi$ is \emph{normalised}, that is, $\xi(\mathbf{1}_{\mathsf A})=1$.
\end{itemize}
We denote by $\mathcal S(\mathsf A)$ the collection of states on $\mathsf A$. Further, a state $\xi$ on $\mathsf A$:
\begin{itemize}
    \item is \emph{faithful} if $\xi(a^\ast a)>0$ for each $a\in \mathsf A\setminus\{0\}$; 
    \item is \emph{tracial} if for any $a,b\in \mathsf A$ it holds $\tau(ab)=\tau(ba)$.
\end{itemize}
We denote by $\mathcal S_{\mathrm f}(\mathsf A)$ the collection of faithful states on $\mathsf A$, and by $\mathcal T(\mathsf A)$ the set of tracial states on $\mathsf A$. Notice that $\mathcal T(\mathsf A)$ is not empty if the algebra $\mathsf A$ is finite dimensional, while  tracial states may not exist for a general $C^*$-algebra. 

For finite-dimensional $C^{\ast}$-algebras, the Wedderburn–Artin theorem implies there is a $C^\ast$-isomorphism
  \begin{equation}
\Phi\colon \mathsf A\longrightarrow \bigoplus_{j=0}^k\mathsf A_j\,,
\end{equation} 
where each $\mathsf A_j$ is a $C^\ast$-algebra which is in turn $C^\ast$-isomorphic to the $C^\ast$-algebra $\mathcal{B}(\mathcal{H}_{j})$ of bounded operators on a suitable Hilbert space $\mathcal H_j$. Given any vector $\left(\lambda_0,\dots,\lambda_k\right)\in \mathbb R^{1+k}_{\ge0}$ (i.e. $\lambda_j\geq 0$), such that $\sum_{j=0}^k \lambda_j=1$, one proves that 
\begin{equation}\label{Eq: tau decomposition}
\tau(a)=\sum_{j=0}^k \lambda^j\,\overline{\Tr}\left(\Phi_j(a)\right)\,.
\end{equation}
is a tracial state on $\mathsf A$, where $\overline{\Tr}$ denotes the normalized trace on $\mathcal B(\mathcal H_j)$, and $\Phi_j\colon \mathsf A\to \mathsf A_j$ is the composition of $\Phi$ with the projection onto the $j$-th factor. Any tracial state on $\mathsf A$ is obtained in this way, and it is faithful if and only if each $\lambda_j$ is nonnegative.

\subsubsection{The geometry of $\mathcal {S}_{\mathrm f}(\mathsf A)$}
From now on we assume $\mathsf A$ to be finite-dimensional: faithful states always exist, and we can represent states by density operators inside the algebra, as we now describe.

Fix a faithful tracial state $\tau$ on $\mathsf A$. The bilinear product $\langle a|b\rangle_{\tau}=\tau(a^{*}\,b)$ endows $\mathsf{A}$ with a Hilbert space structure so that Riesz theorem allows the identification of $\mathsf{A}$ with $\mathsf{A}^{*}$. 
In particular, for any  faithful state $\rho$ on $\mathsf A$ there exists a unique element $\hat\rho\in \mathsf A_{++}$ such that, for any $a\in\mathsf A$, it is 
\begin{equation}
\rho(a)=\tau(a\,\hat\rho)\,.
\end{equation}
We call $\hat\rho$ the \underline{density operator} of $\rho$ with respect to $\tau$. Since $\rho$ is normalised, $\tau(\hat\rho)=1$, and we obtain the bijective correspondence
\begin{equation}\label{Eq: bijective correspondence}
\mathcal S_{\mathrm f}(\mathsf A)\cong\mathsf A_{++}^\tau=\{a\in\mathsf A_{++}\mid \tau(a)=1\}\,.
\end{equation}
The space $\mathsf A_{\sa}$ is a finite-dimensional real vector space, and the subset $\mathsf A_{++}\subset\mathsf A_{\sa}$ is open, hence a smooth manifold. The set $\mathsf A_{++}^\tau$ is the intersection of $\mathsf A_{++}$ with the affine hyperplane $\{a\in\mathsf A_{\sa}\mid \tau(a)=1\}$, therefore it is a codimension-one submanifold embedded into  $\mathsf A_{\sa}$.
Tangent spaces of $\mathsf A_{\sa}$ are canonically identified with $\mathsf A_{\sa}$ itself. For $c\in\mathsf A_{\sa}$, denote by $\widetilde X_c$ the complete vector field on $\mathsf A_{\sa}$ whose flow is given by
\begin{equation}
\phi^{\widetilde X_c}_t(\hat\rho)=\hat\rho+t\,c\,, \qquad (t,\hat \rho)\in \mathbb R\times \mathsf A_{\sa}\,.
\end{equation}
The family $\{\widetilde X_c\}_{c\in\mathsf A_{\sa}}$ generates  the module $\mathfrak X(\mathsf A_{\sa})$, and any basis of $\mathsf A_{\sa}$ determines a global frame. Moreover, for any $\hat\rho\in\mathsf A_{++}$, the correspondence 
  \begin{equation}\label{Eq:iso-sa}
c\in\mathsf A_{\sa}\longmapsto (\widetilde X_c)_{\hat\rho}\in \T_{\hat\rho}\mathsf A_{\sa}
\end{equation} 
is a linear isomorphism.

Given $c\in\mathsf A_{\sa}$, the vector field $\widetilde X_c$ is tangent to $\mathsf A_{++}^\tau$ if and only if $\tau(c)=0$. In this case there exists a unique vector field $X_c$ on $\mathsf A_{++}^\tau$ which is $\mathrm{j}$-related to $\widetilde X_c$, where   $\mathrm{j}\colon \mathsf A_{++}^\tau\hookrightarrow\mathsf A_{\sa}$ denotes the canonical inclusion. Set
\begin{equation}
\mathsf A_{\sa}^0=\ker\tau\cap\mathsf A_{\sa}\,.
\end{equation}
Then $\{X_c\}_{c\in\mathsf A_{\sa}^0}$ generates $\mathfrak X(\mathsf A_{++}^\tau)$, and any basis of $\mathsf A_{\sa}^0$ determines a global frame. The isomorphism \eqref{Eq:iso-sa} restricts to
  \begin{equation}
\mathsf A_{\sa}^0\longrightarrow \T_{\hat\rho}\mathsf A_{++}^\tau\,.
\end{equation} 
We call its inverse the $\tau$-\underline{mixture representation} at $\hat\rho$. For $v\in\T_{\hat\rho}\mathsf A_{++}^\tau$, we denote by $v^\tau\in\mathsf A_{\sa}^0$ the corresponding element. For a vector field $Z$ on $\mathsf A_{++}^\tau$, define $Z^\tau\colon \mathsf A_{++}\to \mathsf A_{\sa}^0$ as 
\begin{equation}
Z^\tau(\hat\rho)=(Z_{\hat\rho})^\tau\,.
\end{equation}

\begin{remark}
The manifold structure on $\mathcal S_{\mathrm f}(\mathsf A)$ induced by the bijection \eqref{Eq: bijective correspondence} with $\mathsf A_{++}^\tau$ does not depend on $\tau$. If $\tau'$ is another faithful tracial state, the manifolds $\mathsf A_{++}^\tau$ and $\mathsf A_{++}^{\tau'}$ are diffeomorphic. An explicit diffeomorphism $\Phi_{\tau'\leftarrow\tau}\colon \mathsf A_{++}^\tau\to\mathsf A_{++}^{\tau'}$ is given, for any $\rho\in\mathsf S_f(\mathsf A)$, by the linear map
\begin{equation}
\Phi_{\tau'\leftarrow\tau}(\hat\rho)=h^{\frac{1}{2}}\,\hat\rho\,h^{\frac{1}{2}}\,,
\end{equation}
where $h\in\mathsf A_{++}$ is the density operator of $\tau$ with respect to $\tau'$, namely
\begin{equation}
    \tau(\hat \rho)=\tau'(h\,\hat\rho)\,, \qquad \rho \in \mathcal S_{\mathrm f}(\mathsf A)\,.
\end{equation}
\end{remark}

The subset $\mathsf A_{++}^\tau$ is convex in $\mathsf A_{\sa}$ and therefore inherits the affine structure of the ambient space $\mathsf A_{\sa}$, which we call the $\tau$-\underline{mixture connection} and denote by $\nabla^{m,\tau}$. This connection is characterised by the teleparallelity condition
\begin{equation}
\nabla^{m,\tau}_{X_c}X_d=0\,, \qquad c,d\in\mathsf A_{\sa}^0\,.
\end{equation}
The connection $\nabla^{m,\tau}$ is both torsion-free and curvature-free as a consequence of the commutation relation
\begin{equation}
    [X_c,X_d]=0\,, \qquad c,d\in \mathsf A_{\sa}^0\,.
\end{equation}
For $\rho,\sigma\in\mathcal S_{\mathrm f}(\mathsf A)$, the unique $\nabla^{m,\tau}$-geodesic joining $\hat\rho$ to  $\hat\sigma$ is the affine segment
\begin{equation}
\gamma_{\hat\sigma\leftarrow\hat\rho}(t)=\hat\rho+t\,(\hat\sigma-\hat\rho)\,,
\end{equation}
and the $\nabla^{m,\tau}$-parallel transport preserves the $\tau$-mixture representation,  in the sense that for any $X\in\mathfrak X(\mathsf A_{++}^\tau)$ one has
\begin{equation}\label{Eq:parallel-transport-mixture}
\left(P^{\nabla^{m,\tau}}_{\hat\sigma\leftarrow\hat\rho}X_{\hat\rho}\right)^\tau=X^\tau(\hat\rho)\,.
\end{equation}
Therefore, $(\mathsf A_{++}^\tau,\nabla^{m,\tau})$ is a space of absolute parallelism.

Let us now consider the existence of Riemannian metric tensors on $\mathsf A_{++}^\tau$ that generalize the Fisher--Rao metric on the interior of the classical simplex \cite{Rao-1945} and the quantum monotone metrics on the manifold of faithful quantum states \cite{MC1991,Petz-1996}. Consider an   \emph{operator monotone function} $f\colon \mathbb R_{>0}\to\mathbb R$  which satisfies  the symmetry condition
\begin{equation}
t\,f(t^{-1})=f(t)\,, \qquad t>0\,,
\end{equation}
and is normalized by the condition $f(1)=1$ \cite{OST2005}. For each $\rho\in\mathcal S_{\mathrm f}(\mathsf A)$ and $X,Y\in\mathfrak X(\mathsf A_{++}^\tau)$, set
\begin{equation}
g^{f,\tau}(X,Y)(\hat\rho)=\tau\!\left(X^\tau(\hat\rho)\,m_f(L_{\hat\rho},R_{\hat\rho})\,Y^\tau(\hat\rho)\right)\,,
\end{equation}
where $L_{\hat\rho},R_{\hat\rho}\colon\mathsf{A}\rightarrow\mathsf{A}$ are, respectively, the left and right multiplication operators by $\hat\rho$, and $m_f$ is the \emph{Morozova--Chentsov function}
 \begin{equation}
m_f(t,s)=\frac{1}{s\,f\!\left(\frac{t}{s}\right)}\,.
\end{equation}
For each $X,Y\in \mathfrak X(\mathsf A_{++}^\tau)$, the map $g^{f,\tau}(X,Y)$ is a smooth function on $\mathsf A_{++}^\tau$: this comes  by the smoothness of the maps $X^\tau$ and $Y^\tau$, the smooth dependence of $L_{\hat\rho}$ and $R_{\hat\rho}$ on $\hat\rho$, the smoothness of multiplication and inversion in $\GL(\mathsf A)$, and the smoothness of the functional calculus on $C^\ast$-algebras \cite{B-Z-2006}. Since for any $h\in C^\infty(\mathsf A_{++}^\tau)$ and $Z\in \mathfrak X(\mathsf A_{++}^\tau)$ one has
\begin{equation}
(h\,Z)^\tau=h\,Z^\tau\,,
\end{equation}
it follows that $g^{f,\tau}$ is a smooth $2$-covariant tensor. Symmetry and nondegeneracy follow by arguments analogous to those in \cite[Theorem 7]{Petz-1996}.

We recover the following cases:
\begin{itemize}

\item When $\mathsf A$ is the $C^\ast$-algebra of bounded operators on a finite-dimensional Hilbert space and $\tau=\overline{\Tr}$ is the normalized trace, one obtains the quantum monotone metric tensors $g^f$ \cite{MC1991,Petz-1996}. In this case,
\begin{equation}
g^f\left(X,Y\right)(\hat\rho)=\overline{\Tr}\left(X(\hat\rho)\,m_f\left(L_{\hat\rho},R_{\hat\rho}\right)\,Y(\hat \rho)\right)\,.
\end{equation}
Here $X(\hat \rho)$ and $Y(\hat \rho)$ denote the $\overline{\Tr}$-mixture representations at $\hat \rho$ of the tangent vectors determined by the vector fields $X$ and $Y$ on $\mathsf A_{++}^{\overline{\Tr}}$, respectively.

\item When $\mathsf A=\mathbb C^{1+d}$ and $\tau=\overline{\tau}$ is the arithmetic mean, the construction yields the Fisher--Rao metric tensor $g^{\FR}$ on the open interior of the scaled $d$-dimensional simplex \cite{C1981a,G-R-2006, G-R-2006-02}. In this case,
  \begin{equation}
g^{\FR}(X,Y)(\hat\rho)=\sum_{j=0}^d \frac{X^j(\hat\rho)\,Y^j(\hat \rho)}{\hat \rho^j}\,.
\end{equation} 
Here $X^j(\hat \rho)$ and $Y^j(\hat \rho)$ denote the $j$-th component of the $\overline{\tau}$-mixture representation at $\hat \rho$ of the tangent vectors determined by the vector fields $X$ and $Y$ on $\mathsf A_{++}^{\overline{\tau}}$, respectively.

\end{itemize}
Therefore, the family $g^{f,\tau}$ provides a unified framework encompassing both the classical Fisher--Rao metric and the quantum monotone metrics.

\subsubsection{Solution bi-form}
Since the $\tau$-mixture affine connection is flat, the triple
\begin{equation}\label{Eq: tau-Lauritzen} 
\left(\mathsf{A}_{++}^\tau, g^{f,\tau},\nabla^{m,\tau}\right)\, 
\end{equation} 
is a partially flat SMAT. We compute its solution bi-form $\varpi^{f,\tau}$, and integrate it, namely determine a potential for it. 

By \cref{Prop: canonical bi-form parallel transport}, and recalling \eqref{Eq:parallel-transport-mixture}, we obtain, for each $X,Y\in \mathfrak X(\mathsf{A}_{++}^\tau)$ and $\rho,\sigma\in \mathcal S_{\mathrm f}(\mathsf A)$:
\begin{equation}\label{eqn:Cstar-bi-form}
\varpi^{f,\tau}(X\mid Y)(\hat \rho,\hat \sigma)=\tau\left(X^\tau(\hat \rho)\, m_f(L_{\hat \sigma},R_{\hat\sigma})\,Y^\tau(\hat \sigma)\right)\,.
\end{equation} 
Since $\nabla^{m,\tau}$ is torsion-free, the solution bi-form is left-exact (cf. \cref{Prop: canonical contrast and precontrast}), and we write 
\begin{equation}\label{Eq: global left exactness} 
\varpi^{f,\tau}=\mathrm d^L S^{f,\tau}\,, 
\end{equation}
where $S^{f,\tau}$ is the $(0,1)$-bi-form given by
\begin{equation} 
S^{f,\tau}(\mid Y)(\hat \rho,\hat\sigma)=\tau\left((\hat \rho-\hat \sigma)\, m_f(L_{\hat \sigma},R_{\hat \sigma})\, Y^\tau(\hat \sigma)\right)\,. 
\end{equation}
The $g^{f,\tau}$-conjugate affine connection is not necessarily torsion-free \cite{N1995b,C-DC-I-M-2023}. As a consequence, the canonical contrast function (cf. \eqref{Eq: canonical contrast}) given by
\begin{equation}\label{Eq: canonical contrast function AA}
    D^{f,\tau}(\hat \rho,\hat \sigma)=-\int_0^1 r\, \|\dot\gamma_{\hat \sigma\leftarrow \hat \rho}(r)\|_{g^{f,\tau}}^2\,\mathrm dr\,
\end{equation}
 does not, in general, reproduce the original Lauritzen manifold \eqref{Eq: tau-Lauritzen} (cf. \cref{Prop: further consideration contrast}). Indeed, $D^{f,\tau}$ induces the same metric of $\varpi^{f,\tau}$, but a possibly different affine connection:
\begin{equation}
    \left(A_{++}^\tau, g^{f,\tau}, \nabla^{D^{f,\tau}}\right)\,.
\end{equation}
As discussed in \cref{Prop: further consideration contrast}, the difference tensor $C=\nabla^{m,\tau}-\nabla^{D^{f,\tau}}$ is implicitly defined by the condition
  \begin{equation}
   g^{f,\tau}\left(C(Z,X),Y\right)
    =
    \left(\mathcal L_{Z\mid }(\mathrm d^L\Sigma^{f,\tau})\right)_{\restriction \Delta_M}(X,Y)\,,
\end{equation} 
where the term 
\begin{equation}
    \Sigma^{f,\tau}
    =
    S^{f,\tau}-\mathrm d^R D^{f,\tau}
\end{equation}
is the $J_R^{\nabla^{m,\tau}}$-right-antiexact part of $S^{f,\tau}$. Moreover, $\nabla^{m,\tau}$ and $\nabla^{D^{f,\tau}}$ coincide if and only if the partially flat SMAT \eqref{Eq: tau-Lauritzen} is a dually flat statistical manifold. 

An important example in which the Lauritzen manifolds induced by $D^{f,\tau}$ and $S^{f,\tau}$ coincide is obtained by selecting the Bogoliubov--Kubo--Mori quantum monotone function
\begin{equation}
    f^{\operatorname{BKM}}(t)=\frac{t-1}{\log t}\,.
\end{equation}
Following an argument analogous to \cite{F-M-A-2019}, one proves the following expression for the canonical contrast function $D^{\operatorname{BKM}}$
\begin{equation}\label{Eq: tau-VnU}
    D^{\operatorname{BKM},\tau}(\hat \rho,\hat \sigma)=-\tau\left(\hat \rho\, (\log \hat \rho -\log\hat \sigma)\right)\,,
\end{equation}
and that it generates the dually flat statistical manifold
\begin{equation}
    \left(\mathsf A_{++}^{\tau}, g^{\operatorname{BKM},\tau},\nabla^{m,\tau}\right)\,.
\end{equation}
Note that contrast function \eqref{Eq: tau-VnU} encompasses both the classical and the quantum relative entropies as particular instances (up to our sign convention \cref{Remark: sign convention}).
\begin{itemize}
    \item If one selects $\mathsf A=\mathbb C^{1+d}$ and  $\tau$ as the arithmetic mean $\overline{\tau}$, then $D^{\operatorname{BKM},\overline{\tau}}$ provides the opposite of the \emph{Kullback--Leibler relative entropy} \cite{A-A-2015, F-M-A-2019}
    \begin{equation}
        \operatorname{KL}(\hat \rho,\hat \sigma)=\sum_{j=0}^d \hat \rho^j \log \frac{\hat \rho^j}{\hat \sigma^j}\,;
    \end{equation}
    \item If one selects $\mathsf A$ as the $C^\ast$-algebra of bounded operators on a complex Hilbert  $\mathcal H$ space and $\tau$ as the normalized trace $\overline{\Tr}$, then $D^{\operatorname{BKM},\tau}$ provides the opposite of the   \emph{von Neumann--Umegaki relative entropy}  \cite{F-M-A-2019,Jencova-2001}:
    \begin{equation}
        \operatorname{VnU}(\hat \rho,\hat \sigma)=\Tr\left(\hat \rho\,  (\log \hat \rho -\log\hat \sigma)\right)\,.
    \end{equation}
\end{itemize}

In the torsion-full setting, the canonical contrast function \eqref{Eq: canonical contrast function AA} has been obtained in some concrete examples  when $\mathsf A=\mathcal B(\mathcal H)$, even if not directly from the integral form.

\begin{equation}
f^{\operatorname{BH}}(t)=\frac{1+t}{2}\,,
\end{equation}
the canonical contrast function \eqref{Eq: canonical contrast function AA} is given by
\begin{equation}\label{eqn:BH-divergence}
D^{\operatorname{BH}}(\hat \rho,\hat \sigma)=-2\,\operatorname{Tr}\left(\hat\sigma\,\log\left(\hat\rho^{-1/2}\,\left(\hat\rho^{1/2}\,\hat\sigma\,\hat\rho^{1/2}\right)^{1/2}\,\hat\rho^{-1/2}\right)\right)\,,
\end{equation}
which is the opposite of the divergence function found in \cite{Braunstein-1996} and \cite[Example 6.2]{Jencova-2001} along a different path. 
For the quantum monotone function 
\begin{equation}
f^{\operatorname{RLD}}(t)=\frac{2\,t}{1+t}\,
\end{equation}
associated to the metric of the right logarithmic derivative, the canonical contrast function \eqref{Eq: canonical contrast function AA} takes the form:
\begin{equation}\label{eqn:RLD-divergence}
D^{\operatorname{RLD}}(\hat \rho,\hat \sigma)=-\operatorname{Tr}\left(\hat\sigma\,\log\left(\hat\sigma^{1/2}\hat\rho^{-1}\hat\sigma^{1/2}\right)\right)\,,
\end{equation}
which is the opposite of the divergence function obtained in \cite[Example 6.3]{Jencova-2001}, again following a different approach. 
It is worth stressing how both divergence functions in \eqref{eqn:BH-divergence} and \eqref{eqn:RLD-divergence} do not allow to recover the partially flat SMAT structure of the underlying model since the torsion is non-vanishing, while the left-exact bi-form in equation \eqref{eqn:Cstar-bi-form} does.

\subsection{A second example: the left and right Cartan connections on semisimple Lie groups}\label{Subsec: Lie group}
Any Lie group is a space of absolute parallelism with respect to both the left and the right Cartan connections~\cite{CartanSchouten}.
Moreover, semisimple Lie groups admit a canonical pseudo-Riemannian metric,
the Cartan--Killing form, which makes the corresponding dually curvature-free Lauritzen structure left-invariant.

Let $G$ be a Lie group with Lie algebra $\mathfrak g$, identified with the space of
left-invariant vector fields on $G$. An affine connection $\nabla$ on $G$ is said to be
\emph{left-invariant} if and only if $(L_z)^\ast\nabla=\nabla$ for every $z\in G$,
where $L_z\colon G\to G$ denotes the left translation, $L_z(x)=z\cdot x$.
Each left-invariant  connection induces an $\mathbb R$-bilinear map
$A^\nabla\colon\mathfrak g\times\mathfrak g\to\mathfrak g$ defined by
\begin{equation}
    A^\nabla(X,Y)=\nabla_XY\,.
\end{equation}
Conversely, any $\mathbb R$-bilinear map $\mathfrak g\times\mathfrak g\to\mathfrak g$
uniquely determines a left-invariant affine connection on $G$.

In terms of  the Lie bracket on $\mathfrak g$, we introduce the one-parameter family
of left-invariant affine connections $(\nabla^\lambda)_{\lambda\in\mathbb R}$ given by
\begin{equation}
    A^\lambda(X,Y)=\lambda\,[X,Y]\,,\qquad X,Y\in\mathfrak g\,,
\end{equation}
for which 
\begin{align}
    \Tor^\lambda(X,Y)&=(2\lambda-1)\,[X,Y]\,,\label{Eq: torsion invariant}\\
    R^\lambda(X,Y)Z&=\lambda(\lambda-1)[[X,Y],Z]\,.\label{Eq: curvature invariant}
\end{align}
This shows that  $\nabla^+=\nabla^0$ and $\nabla^-=\nabla^1$ are curvature-free affine connections,
and torsion-free if and only if $G$ is abelian. They are known respectively as the \emph{left} and \emph{right Cartan connections} since  $(G,\nabla^+)$ and $(G,\nabla^-)$ are curvature-free (cf. \eqref{Eq: curvature invariant}) and, for every $(x,y)\in G\times G$, it is 
\begin{align}
    P^{\nabla^+}_{y\leftarrow x}&=(L_{y\cdot x^{-1}})_x'\,,\label{Eq: + transport}\\
    P^{\nabla^-}_{y\leftarrow x}&=(R_{y\cdot x^{-1}})_x'\,\label{Eq: - transport}
\end{align}
(here, for $z\in G$, $R_z\colon G\to G$ denotes the right translation $R_z(y)=y\cdot z$).
\begin{lemma}\label{Lemma: semisimple Lie groups}
Let $g$ be a left-invariant pseudo-Riemannian metric on a Lie group $G$, that is,
\begin{equation}\label{Eq: left invariant metric}
    (L_z)^\ast g=g\,,\qquad z\in G\,.
\end{equation}
One has
\begin{enumerate}[(a)]
    \item the connection $\nabla^+$ is $g$-self-conjugate;
    \item the connection $\nabla^-$ is $g$-self-conjugate if and only if $g$ is right-invariant, i.e.,
    \begin{equation}
        (R_z)^\ast g=g\,,\qquad z\in G\,;
    \end{equation}
    \item the term $\overline\nabla=\frac{\nabla^++\nabla^-}{2}$ gives the Levi-Civita connection of $g$ if and only if $g$ is bi-invariant.
\end{enumerate}
\end{lemma}

\begin{proof}
Let $X,Y,Z\in\mathfrak g$. Since $g(X,Y)$ is constant, the compatibility
condition \eqref{Eq: compatibility condition} for a left-invariant affine connection $\nabla$ reduces to
\begin{equation}
    0=g(\nabla_ZX,Y)+g(X,\nabla^\dag_ZY)\,.
\end{equation}
\begin{enumerate}[(a)]
\item If $\nabla=\nabla^+$, we immediately have $(\nabla^+)^\dag=\nabla^+$.
\item Recall that $g$ is bi-invariant if and only if
\begin{equation}
    g([Z,X],Y)+g(X,[Z,Y])=0\,:
\end{equation}
this gives that  $g$ is bi-invariant if and only if when $\nabla^-$ is $g$-self-conjugate.
\item The term $\overline\nabla=\nabla^{1/2}$ is a torsion-free affine connection on $G$ (cf. \eqref{Eq: torsion invariant}). In order to coincide with the Levi-Civita connection of $G$, the connection $\overline\nabla$ needs to be $g$-compatible, i.e. $g$-self-conjugate. Since $\nabla^+$ is so, $\overline\nabla$ is $g$-self-conjugate if and only if $\nabla^-$ is such, and this happens precisely when $g$ is bi-invariant.
\end{enumerate}

\end{proof}
A Lie group is semisimple if and only if the Killing--Cartan form $\kappa$
is non-degenerate. For all $X,Y\in\mathfrak g$, $\kappa$ is defined by
\begin{equation}
    \kappa(X,Y)=\Tr(\ad_X\circ\ad_Y)\,,
\end{equation}
where $\Tr$ is the trace in the real vector space of $\mathbb R$-linear endomorphisms of the real vector space $\mathfrak g$, and $\ad$ is the adjoint representation of $\mathfrak g$ on $\mathfrak g$. One proves that  $\kappa$ is symmetric and bi-invariant, hence a bi-invariant pseudo--Riemannian metric on $G$. 
Since $G$ is a non-trivial semisimple, it is non-abelian, and therefore the connections $\nabla^\pm$
are not torsion-free. Thus, for a nontrivial semisimple Lie group $G$, the triples
\begin{equation}\label{Eq: semisimple Lie groups Lauritzen manifold}
    (G,\kappa,\nabla^\pm)
\end{equation}
are dually curvature-free Lauritzen manifolds which are not partially flat SMATs. 
\begin{remark}
    The triples \eqref{Eq: semisimple Lie groups Lauritzen manifold} are nontrivial instances of \emph{Cartan--Riemannian manifolds}, that is, triples $(M,g,\nabla)$ where $(M,g)$ is a pseudo--Riemannian manifold, and $\nabla$ is a $g$-self-conjugate affine connection on $M$ \cite{Agricola2006}.
\end{remark}
\begin{remark}
Semisimple Lie groups do not admit flat left-invariant affine connections \cite{M-O-1979}. In analogy with \cref{Remark: torsion and topological/algebraic obstruction}, this obstruction disappears if one allows non-zero torsion: left-invariant curvature-free connections then exist.
\end{remark}
\subsubsection{Solution bi-forms}
Let $\varpi^\pm$ denote the solution bi-forms of $(G,\kappa,\nabla^\pm)$.
Since $(G,\nabla^\pm)$ are spaces of absolute parallelism, we can extend
$\varpi^\pm$ to globally defined $(1,1)$-bi-forms using \cref{Prop: canonical bi-form parallel transport}. 

Let us compute $\varpi^+$.
Recall that for every $Z\in\mathfrak g$ and $z\in G$ we have
\begin{equation}\label{Eq: left invariant v.f.s}
    Z_z=(L_z)_e'Z_e\,.
\end{equation}
Hence, for $X,Y\in\mathfrak g$ and $(x,y)\in G\times G$, if we recall the definition \eqref{Eq: canonical bi-form parallel transport}, the properties of the tangent map on left invariant vector fields \eqref{Eq: left invariant v.f.s} together with the properties of left invariance of the metric tensor \eqref{Eq: left invariant metric}, we have the following chain of equalities
  \begin{equation}
\begin{split}
    \varpi^+(X\mid Y)(x,y)=
    \kappa_y\!\left(P^{\nabla^+}_{y\leftarrow x}X_x,Y_y\right)&=
    \kappa_y\!\left((L_{y\cdot x^{-1}})_x'X_x,Y_y\right)\\
    &= \kappa_y\!\left((L_y)_e'(L_{x^{-1}})_x'X_x,(L_y)_e'Y_e\right)=
    \kappa_e\!\left((L_{x^{-1}})_x'X_x,Y_e\right)
    =\kappa_e(X_e,Y_e)\,.
\end{split}
\end{equation} 
Thus, by left-invariance of $\kappa$, we write
\begin{equation}\label{Eq: solution bi-form +}
    \varpi^+(X\mid Y)=\kappa(X,Y)\,.
\end{equation}
Let us now compute $\varpi^-$.
We claim that
\begin{equation}\label{Eq: solution bi-form -}
    \varpi^-(X\mid Y)(x,y)=\kappa(\Ad_{y^{-1}\cdot x}X,Y)\,,
\end{equation}
or, equivalently
\begin{equation}
    \varpi^-(X\mid Y)(x,y)=\kappa(\Ad_{x}X,\Ad_y Y)\,,
\end{equation}
where $\Ad$ denotes the adjoint representation of $G$ on $\mathfrak g$.
Recall that for $z\in G$ and $Z\in\mathfrak g$,
$\Ad_zZ$ is the left-invariant vector field on $G$ satisfying
\begin{equation}
    (\Ad_zZ)_e=(C_z)_e'Z_e\,,\qquad C_z=L_z\circ R_{z^{-1}}\,.
\end{equation}
For $X,Y\in\mathfrak g$ and $(x,y)\in G\times G$, we can write, in analogy to what we have shown before:
\begin{equation}
\begin{split}
    \varpi^-(X\mid Y)(x,y)&=
    \kappa_y\!\left(P^{\nabla^-}_{y\leftarrow x}X_x,Y_y\right)\\ &=
    \kappa_y\!\left((R_{y\cdot x^{-1}})_x'X_x,Y_y\right)\\
    &=\kappa_y\!\left((R_{x^{-1}})_e'(R_y)_x'X_x,(L_y)_e'Y_e\right)\\ &=\kappa_e\!\left((L_{y^{-1}})_y'(R_{x^{-1}})_y'(R_y)_x'X_x,Y_e\right)=
    \kappa_e\!\left((L_{y^{-1}})_y'(R_{x^{-1}})_y'(R_y)_x'(L_x)_e'X_e,Y_e\right)\,.
\end{split}
\end{equation} 
From 
\begin{equation}
    L_{y^{-1}}\circ R_{x^{-1}}\circ R_y\circ L_x
    =C_{y^{-1}\cdot x}\,,
\end{equation}
we have
\begin{equation}
    \varpi^-(X\mid Y)(x,y)
    =\kappa_e\!\left((\Ad_{y^{-1}\cdot x}X)_e,Y_e\right)\,,
\end{equation}
and by left-invariance of $\kappa$, formula~\eqref{Eq: solution bi-form -} follows.

\begin{proposition}
Let $\varpi^\pm$ be the solution bi-forms on a semisimple Lie group $G$. One has:
\begin{enumerate}[(a)]
    \item the bi-form $\varpi^+$ is left-invariant, i.e. for any $(x,y), (z_1,z_2)\in G\times G$ it is
    \begin{equation}
        \varpi^+(z_1\!\cdot\!x,z_2\!\cdot\!y)=\varpi^+(x,y)\,.
    \end{equation}
    Moreover, $\varpi^+$ is the unique left-invariant contrast bi-form integrating $(G,\kappa,\nabla^+)$;
    \item the bi-form $\varpi^-$ is diagonally left-invariant, i.e. for any $x,y,z\in G$:
    \begin{equation}
        \varpi^-(z\!\cdot\!x,z\!\cdot\!y)=\varpi^-(x,y)\,.
    \end{equation}
\end{enumerate}
\end{proposition}
\begin{proof}
\begin{enumerate}[(a)] 
\item The left-invariance of $\varpi^+$ follows from the fact that the pairing of $\varpi^+$ with left-invariant vector fields is constant. 
If $\varpi$ is a left-invariant contrast bi-form such that $g^\varpi=\kappa$ and $\nabla^\varpi=\nabla^+$, then $\varpi(X\mid Y)$ is constant for each $X,Y\in \mathfrak g$. This implies that $\nabla^\varpi=\nabla^+$ (cf. \eqref{Eq: induced connection}) and that $\varpi(X\mid Y)=\varpi(X\mid Y)(e,e)=g^\varpi(X\mid Y)(e)=\kappa(X,Y)$.
\item  The diagonal left-invariance of $\varpi^-$ follows from $(z\cdot y)^{-1}\cdot (z\cdot x)=y^{-1}\cdot x$. 
\end{enumerate}
\end{proof}

\begin{example}
Let $\mathsf A$ be a finite-dimensional $C^\ast$-algebra with $n=\dim_{\mathbb C}\mathsf A$, and denote by $\GL(\mathsf A)$ the group of invertible elements of $\mathsf A$. Since $\GL(\mathsf A)$ is an open subset of $\mathsf A$, its Lie algebra $\mathfrak{gl}(\mathsf A)$ identifies with the realification of $\mathsf A$. More precisely, for each $c\in \mathsf A$, the flow of the associated left-invariant vector field $Z_c$ is 
\begin{equation}
    \phi^{Z_c}(a,t)=a\,\exp(t\,c)\,, \qquad (a,t)\in \GL(\mathsf A)\times \mathbb R\,.
\end{equation}
One proves that this correspondence yields an isomorphism of real Lie algebras, that is
\begin{equation}
   [Z_c,Z_d]=Z_{[c,d]}\,, \qquad c,d\in \mathsf A\,.
\end{equation}
We consider the derived subgroup
\begin{equation}
    \SL(\mathsf A)=[\GL(\mathsf A),\GL(\mathsf A)]\,.
\end{equation}
By the structure theorem for finite-dimensional $C^\ast$-algebras, there exists a $C^\ast$-isomorphism:
\begin{equation}\label{Eq: iso finite dim}
    \mathsf A \cong \bigoplus_{j=0}^r \mathcal B(\mathcal H_j)\,,
\end{equation}
where $\mathcal H_j$ are complex Hilbert spaces of dimension $n_j=\dim_{\mathbb C}\mathcal H_j$. It follows that
\begin{equation}
    \SL(\mathsf A)\cong \prod_{j=0}^r \SL(\mathcal H_j)\,,
\end{equation}
and in particular $\SL(\mathsf A)$ is a semisimple Lie group. Accordingly, the isomorphism $\mathfrak{gl}(\mathsf A)\cong \mathsf A$ restricts to
\begin{equation}
    \mathfrak{sl}(\mathsf A)\cong [\mathsf A,\mathsf A]\,.
\end{equation}

On $\mathfrak{sl}(\mathsf A)$, the Cartan--Killing form $\kappa$ is given by
\begin{equation}
    \kappa(Z_c,Z_d)=4\,n\,\Re\,\tau(c\,d)\,, \qquad c,d\in [\mathsf A,\mathsf A]\,,
\end{equation}
where $\tau$ is the tracial state on $\mathsf A$,  defined by
\begin{equation}
    \tau(a)=\frac{1}{n}\,\sum_{j=0}^r n_j\,\Tr(a_j)\,, \qquad a\in \mathsf A\,,
\end{equation}
and $(a_0,\dots,a_r)$ denotes the image of $a$ under the isomorphism \eqref{Eq: iso finite dim}. 

Since for each $a\in \SL(\mathsf A)$ the adjoint action $\Ad_a$ corresponds to conjugation on $\mathsf A$ by $a$, the solution bi-forms \eqref{Eq: solution bi-form +} and \eqref{Eq: solution bi-form -} can be written as 
\begin{align}
    \varpi^+(Z_c,Z_d)(a,b) &= 4n\,\tau(c\,d)\,,\\
    \varpi^-(Z_c,Z_d)(a,b) &= 4n\,\tau\bigl(b^{-1}\,a\,c\,a^{-1}\,b\,d\bigr),
\end{align}
for all $a,b\in \SL(\mathsf A)$ and $c,d\in [\mathsf A,\mathsf A]$.
\end{example}

\section*{Conclusion}
Upon analysing the notion of bi-form, in this paper we have given a geometric description  for statistical potentials on Lauritzen manifolds. Such a description is based on a suitable bicomplex, namely a  left-right Cartan calculus for these forms, and on corresponding  statistical left and right homotopy operators. Within this setting, one sees  that left-exact and bi-exact contrast bi-forms correspond to pre-contrast and contrast functions, respectively. In the last section we have shown that contrast bi-forms exist  for dually curvature-free Lauritzen manifolds. This bi-form recovers known canonical potentials in the partially flat and dually flat cases. 

We think that several interesting questions arise from this analysis.  We select two among them as a conclusion to this paper. 

\begin{itemize}
\item Contrast functions are usually interpreted as distinguishability measures between two points of a statistical manifold. For contrast bi-forms the situation appears less clear, since their evaluation also involves tangent directions. What is the statistical interpretation of tangent vectors to a Lauritzen manifold, and how is this related to the torsion tensors of the conjugate affine connections?

\item In the dually curvature-free case, a representative of the germ of the solution bi-form reads as
\begin{equation}
    \varpi(X\mid Y)(m,n)=g_n\left(P^\nabla_{n\leftarrow m}X_m,Y_n\right)\,.
\end{equation}
This formula involves only the metric tensor $g$ and the parallel transport of the affine connection $\nabla$ along its geodesics. Hence, it can be considered also for a general Lauritzen manifold, regardless of the curvature assumptions. Does this bi-form generate the original Lauritzen structure $(M,g,\nabla)$? If this is not true in general, how is the obstruction related to the curvature of the connection?
\end{itemize}


\subsection*{Acknowledgments}
This work has been supported by the Madrid Government through the project \textbf{TEC-2024/COM-84 QUITEMAD-CM} and by COST (European Cooperation in Science and Technology) through the COST Action CaLISTA CA21109.

We acknowledge financial support from Next Generation EU through the project 2022XZSAFN – PRIN2022 CUP: E53D23005970006.

We thank our Institutions, the Universidad Carlos III de Madrid,  the Universit\`a Federico II di Napoli, the Scuola Superiore Meridionale di Napoli, as well as Indam and INFN, through Gnsaga and the initiative GeoSymQFT and Quantum,  for the financial support of the research visits we had during the period this work has been developed.

\addcontentsline{toc}{section}{References}
{\footnotesize
	\printbibliography[title=References]
}
\end{document}